\documentclass[a4paper, reqno, 10pt]{amsart}

\usepackage[margin=1in]{geometry}
\usepackage{amsfonts,amssymb}
\usepackage{amsthm}
\usepackage{graphicx}
\usepackage{float}
\usepackage{lipsum}
\usepackage{comment}
\usepackage{systeme}
\usepackage{lineno}
\usepackage{mathtools}
\usepackage{esint}
\usepackage{mathrsfs}
\usepackage{colortbl}
\usepackage[usenames,dvipsnames]{xcolor}
\usepackage{amssymb}
\usepackage[active]{srcltx}

\usepackage{cite}
\usepackage[colorlinks, pdfpagelabels, pdfstartview = FitH,bookmarksopen = true, bookmarksnumbered = true,linkcolor = blue,plainpages = false,hypertexnames = false, citecolor = red]{hyperref}
\usepackage[italian, textsize=tiny, color=BlueGreen, backgroundcolor=BlueGreen, linecolor=BlueGreen]{todonotes}

\allowdisplaybreaks

\usepackage{booktabs,amsmath}
\def\Xint#1{\mathchoice
    {\XXint\displaystyle\textstyle{#1}}%
    {\XXint\textstyle\scriptstyle{#1}}%
    {\XXint\scriptstyle\scriptscriptstyle{#1}}%
    {\XXint\scriptscriptstyle\scriptscriptstyle{#1}}%
      \!\int}
\def\XXint#1#2#3{{\setbox0=\hbox{$#1{#2#3}{\int}$}
    \vcenter{\hbox{$#2#3$}}\kern-.5\wd0}}
\def\dashint{\Xint-}

\def\YYint#1#2#3{{\setbox0=\hbox{$#1{#2#3}{\int}$}
    \lower1ex\hbox{$#2#3$}\kern-.46\wd0}}
\def\YYYint#1#2#3{{\setbox0=\hbox{$#1{#2#3}{\int}$}
    \lower0.35ex\hbox{$#2#3$}\kern-.48\wd0}}

\def\ZZint#1#2#3{{\setbox0=\hbox{$#1{#2#3}{\int}$}
    \raise1.15ex\hbox{$#2#3$}\kern-.57\wd0}}
\def\ZZZint#1#2#3{{\setbox0=\hbox{$#1{#2#3}{\int}$}
    \raise0.85ex\hbox{$#2#3$}\kern-.53\wd0}}

\def\Xiint#1{\mathchoice
    {\XXiint\displaystyle\textstyle{#1}}%
    {\XXiint\textstyle\scriptstyle{#1}}%
    {\XXiint\scriptstyle\scriptscriptstyle{#1}}%
    {\XXiint\scriptscriptstyle\scriptscriptstyle{#1}}%
      \!\iint}
\def\XXiint#1#2#3{{\setbox0=\hbox{$#1{#2#3}{\iint}$}
    \vcenter{\hbox{$#2#3$}}\kern-.5\wd0}}
\def\dashiint{\Xiint-}

\def\XXiiint#1#2#3{{\setbox0=\hbox{$#1{#2#3}{\iiint}$}
    \vcenter{\hbox{$#2#3$}}\kern-.5\wd0}}

\newcommand{\Title}{Calder\'on-Zygmund theory of nonlocal parabolic equations\\
with discontinuous coefficients}
\newcommand{\Shorttitle}{Calder\'on-Zygmund theory of nonlocal parabolic equations with discontinuous coefficients}
\newcommand{\Author}{Sun-Sig Byun, Kyeongbae Kim \and Deepak Kumar}

\theoremstyle{plain}
\newtheorem{thm}{Theorem}[section]
\newtheorem{prop}{Proposition}[section]
\newtheorem{lem}[prop]{Lemma}

\theoremstyle{definition}
\newtheorem{defn}[prop]{Definition}
\newtheorem*{def*}{Definition}

\theoremstyle{remark}
\newtheorem{rmk}{Remark}

\numberwithin{equation}{section}


\newcommand{\dx}{\,dx}
\newcommand{\dt}{\,dt}
\newcommand{\dy}{\,dy}
\newcommand{\dz}{\,dz}

\newcommand{\dmut}{\,d\mu_{\tau}}
\newcommand{\dmutt}{\,d\mu_{\tau,t}}

\newcommand{\ddiv}{\mathrm{div\,}}

\DeclareMathOperator*{\dist}{dist}

\DeclareMathOperator{\Tail}{Tail}

\title[\Shorttitle]{\Title}
\author[Byun, Kim \and Kumar]{\Author}

\address{Sun-Sig Byun: Department of Mathematical Sciences and Research Institute of Mathematics, Seoul National University, Seoul 08826, Korea}
\email{byun@snu.ac.kr}

\address{Kyeongbae Kim: Department of Mathematical Sciences, Seoul National University, Seoul 08826, Korea}
\email{kkba6611@snu.ac.kr}

\address{Deepak Kumar: Research Institute of Mathematics, Seoul National University, Seoul 08826, Korea}
\email{deepak.kr0894@gmail.com}

\thanks{S. Byun was supported by NRF-2022R1A2C1009312, K. Kim and D. Kumar were supported by NRF-2021R1A4A1027378.}

\begin{document}
\maketitle
\begin{abstract}
We prove Calder\'on-Zygmund type estimates of weak solutions to non-homogeneous nonlocal parabolic equations under a minimal regularity requirement on kernel coefficients. In particular, the right-hand side is presented by a sum of fractional Laplacian type data and a non-divergence type data. Interestingly, even though the kernel coefficients are discontinuous, we obtain a significant increment of fractional differentiability for the solutions, which is not observed in the corresponding local parabolic equations.
\end{abstract}
\section{Introduction}
\noindent  
\subsection{Overview}
In this paper, we study higher regularity properties  for weak solutions to the following non-homogeneous nonlocal parabolic equation:
\begin{equation}
\label{cz non pa : eq1}
u_{t}+\mathcal{L}^{\Phi}_{A}u=(-\Delta)^{\frac{s}{2}}f+g \quad\text{ in }\Omega_{T}\equiv\Omega\times(0,T),
\end{equation}
where  $s\in(0,1)$, $T>0$ and $\Omega$ is an open and bounded set in $\mathbb{R}^{n}$ with $n\geq2$. 
The nonlocal operators appearing in problem \eqref{cz non pa : eq1} are defined by
\begin{equation*}
    \mathcal{L}^{\Phi}_{A}u(x,t)=\mathrm{p.v.}\int_{\mathbb{R}^{n}}\Phi\left(\frac{u(x,t)-u(y,t)}{|x-y|^{s}}\right)\frac{A(x,y,t)}{|x-y|^{n+s}}\dy
\end{equation*}
and
\begin{equation*}
    (-\Delta)^{\frac{s}{2}}f(x,t)=\mathrm{p.v.}\int_{\mathbb{R}^{n}}(f(x,t)-f(y,t))\frac{1}{|x-y|^{n+s}}\dy.
\end{equation*}
Here, $f:\mathbb{R}^{n}\times(0,T)\to\mathbb{R}$ and $g:\Omega_{T}\to\mathbb{R}$ are given measurable functions, $A:\mathbb{R}^{n}\times \mathbb{R}^{n}\times (0,T)\to\mathbb{R}$ is a given kernel coefficient satisfying 
\begin{equation}
\label{cz non pa : ker a defn}
L^{-1}\leq A(x,y,t)\leq L \quad\text{and}\quad A(x,y,t)=A(y,x,t)\quad\text{for }(x,y,t)\in \mathbb{R}^{n}\times \mathbb{R}^{n}\times (0,T)
\end{equation}
and for some constant $L\geq1$, and $u:\mathbb{R}^{n}\times(0,T)\to\mathbb{R}$ is the unknown. In addition, $\Phi:\mathbb{R}\to\mathbb{R}$ is a measurable function satisfying $\Phi(0)=0$ and
\begin{equation}
\label{cz non pa : Phicond}
\begin{cases}
(\Phi(\xi)-\Phi(\xi'))(\xi-\xi')\geq L^{-1}|\xi-\xi'|^{2},\\
|\Phi(\xi)-\Phi(\xi')|\leq L|\xi-\xi'| \quad\text{for any }\xi,\xi'\in\mathbb{R}.
\end{cases}
\end{equation}
Nonlocal parabolic problems appear naturally in the physical world, e.g., as in anomalous diffusion processes from the areas of physics, finance, biology, ecology, geophysics, and many others. 
Particularly, the nonlocal nonlinear operators of the above types find their application in image processing \cite{gilb} and phase transition models \cite{GiaLePr}. \par 
{It is known that when the leading operator appearing in \eqref{cz non  pa : eq1} is linear and the right-hand side is regular enough, then solutions enjoy higher H\"older regularity, which in turn, yields improved Sobolev regularity. These assertions can be justified by the functional analysis tools along with a precise integral representation of the solution through suitable heat kernel type estimates. Unfortunately, when the operator is nonlinear, the aforementioned techniques fail to apply. To this end, our objective is to obtain a fine fractional Sobolev regularity for weak solutions to \eqref{cz non pa : eq1} by using purely analytic and geometric techniques. In particular, we introduce a unified approach of covering arguments to obtain some uniform measure density estimates for some level sets involving the solution, which we will explain in the sequel.
}
{More precisely, our aim is} to establish Calder\'on-Zygmund type estimates for weak solutions to \eqref{cz non pa : eq1} under a minimal regularity requirement on the kernel coefficient $A=A(x,y,t)$. More precisely, we want to find an extra condition on $A$ besides \eqref{cz non pa : ker a defn} under which the following implication holds
{\small
\begin{align}\label{cz non pa : cal esti int}
    &f(x,t)-f(y,t)\in L^{q}_{\mathrm{loc}}\left(\Omega\times\Omega_{T}\,;\,\frac{\dx\dy\dt}{|x-y|^{n+\sigma q}}\right)\text{ with } f\in L^{q}_{\mathrm{loc}}(0,T;L^{1}_{s}(\mathbb{R}^{n}))\quad\text{and} \quad g\in L^{\frac{q\left(n+2s+\frac{2\sigma q}{q-2}\right)}{n+4s}}_{\mathrm{loc}}(\Omega_{T}) \nonumber\\
&\Longrightarrow \frac{u(x,t)-u(y,t)}{|x-y|^{s}}\in L^{q}_{\mathrm{loc}}\left(\Omega\times\Omega_T\,;\,\frac{\dx\dy\dt}{|x-y|^{n+\sigma q}}\right)
\end{align}
}\\
for any $q\in(2,\infty)$ and $\sigma\in \left(0,\left(1-\frac{2}{q}\right)\min\left\{s-\frac{2s}{q},1-s\right\}\right)$ with the desired Calder\'on-Zygmund type estimates like \eqref{cz non pa : main est}.
In particular, using the notion of fractional gradients introduced in \cite{MSf,BKc}; that is,
\begin{equation*}
    d_{\beta}u(x,y,t):=\frac{u(x,t)-u(y,t)}{|x-y|^{\beta}}\quad\text{for any }\beta\in[0,1),
\end{equation*}
the implication \eqref{cz non pa : cal esti int} can be rewritten as
{\small\begin{equation*}
\begin{aligned}
    &d_{0}f\in L^{q}_{\mathrm{loc}}\left(\Omega\times\Omega_{T}\,;\,\frac{\dx\dy\dt}{|x-y|^{n+\sigma q}}\right)\text{ with }f\in L^{q}_{\mathrm{loc}}(0,T;L^{1}_{s}(\mathbb{R}^{n}))\quad\text{and} \quad g\in L^{\frac{q\left(n+2s+\frac{2\sigma q}{q-2}\right)}{n+4s}}_{\mathrm{loc}}(\Omega_{T}) \\
    &\Longrightarrow d_{s}u\in L^{q}_{\mathrm{loc}}\left(\Omega\times\Omega_{T}\,;\,\frac{\dx\dy\dt}{|x-y|^{n+\sigma q}}\right).
\end{aligned}
\end{equation*}}

\subsection{Some known results}
For the elliptic problems, a self-improving property of a weak solution to the problem:
\begin{equation}
\label{cz non pa : elleq}
    \mathcal{L}^{\Phi}_{A}u=(-\Delta)^{\frac{s}{2}}f+g
\end{equation}
is obtained by Kuusi, Mingione and Sire \cite{KMS1} by introducing the notion of dual pairs. When $\Phi(\xi)=\xi$ and $A=A(x,y)$ is H\"older continuous, Calder\'on-Zygmund type estimate for \eqref{cz non pa : elleq} is established by Mengesha, Schikorra and Yeepo \cite{MSYc} via commutator estimates. In addition, the aforementioned articles deal with more general equations as source terms involve   $\tilde{s}$-fractional Laplacian with $\tilde{s}\neq \frac{s}{2}$. For operators with possibly discontinuous coefficients, such as VMO coefficients, Nowak \cite{Nv,Ni} obtain Calder\'on-Zygmund type estimates when $f=0$  by using the maximal function and the notion of dual pairs. We refer to \cite{AFLYg} for the global Calder\'on-Zygmund type estimate of \eqref{cz non pa : elleq} with $A=1$, $\Phi(\xi)=\xi$, $f=0$ and the zero Dirichlet condition on the exterior of the domain. The main tool employed in this work is the Green function representation of the solution.

We now mention some related results for the case of nonlinear nonlocal operators. When   $\Phi(\xi)=|\xi|^{p-2}\xi$ with $p>2$ and $f=0$, Nowak and Diening \cite{LNc} obtain sharp regularity results containing borderline cases by establishing precise pointwise bounds in terms of fractional sharp maximal functions. On the other hand, Calder\'on-Zygmund type estimates of solutions to the problem:
\begin{equation*}
    \mathcal{L}_{A}^{\Phi}u=(-\Delta_{p})^{\frac{s}{p}}f,
\end{equation*}
are established in \cite{BKc} via a maximal function free technique which was first introduced in \cite{AMp}. 
We mention the work \cite{DKo} as well for $L^{p}$-theory of a strong solution to nonlocal elliptic equations. For additional regularity results related to nonlocal elliptic equations, we refer to \cite{KMSnm,FMSYc,SMs,BLShh,BLhs,CKPl,LPPSb,Ci,Snc,Olh,BOSdh,BLDg,BKOh,MCm,BWZle,KAAYg,Nhsp,KNSg} and references therein.

For the parabolic problems, Auscher, Bortz, Egert and Saari \cite{ABES} prove a self-improving property of solutions to \eqref{cz non pa : eq1} with $f=0$ by using functional analysis techniques. When $\Phi(\xi)=\xi$, $A(x,y,t)\equiv 1$ and $f=0$, Biccari, Warma and Zuazua \cite{BWZlp} provide optimal regularity results of a weak solution by using a cut off argument. For the $L^{p}$-theory of strong solutions to nonlocal parabolic equations, we refer to \cite{DJKb,Zlp}. We further mention \cite{CCVr,BLS,Sh,Kmh,Lh,APTh,Vaf,KWh,Tah,GKSdh,Gr} and references therein for various regularity results of nonlocal parabolic equations. 

\subsection{Main results}
To explain the desired Carder\'on-Zygmund type estimate \eqref{cz non pa : cal esti int},
  we first introduce the notion of dual pairs. For a measurable function $F:\mathbb{R}^{n}\times \mathbb{R}^{n}\times (0,T)$ and $\tau\in(0,1)$, we define
\begin{equation*}
    D^{\tau}F(x,y,t)\coloneqq\frac{F(x,y,t)}{|x-y|^{\tau}}\quad\text{and}\quad
    \mu_{\tau}(\mathcal{A})\coloneqq\int_{\mathcal{A}}\frac{\dx\dy}{|x-y|^{n-2\tau}},\quad\text{for any measurable set } \mathcal{A}\subset\mathbb{R}^{n}\times\mathbb{R}^{n}.
\end{equation*}
Furthermore, we write $\dmutt=\dmut\dt$. Then we observe that 
\begin{equation}
\label{cz non pa : observ}
    d_{s}u\in L^{q}_{\mathrm{loc}}\left(\Omega\times\Omega_{T}\,;\,\frac{\dx\dy\dt}{|x-y|^{n+\sigma q}}\right)\Longleftrightarrow D^{\tau}d_{s}u\in L^{q}_{\mathrm{loc}}\left(\Omega\times\Omega_{T}\,;\,\dmutt\right),\,\mbox{where $\tau=\frac{q\sigma }{q-2}$.}
\end{equation}
From this observation, we deduce that the solution $u$ improves its integrability order as well as the differentiability order by achieving the same integrability as that of the associated non-homogeneous term in the Sobolev scale and a substantial gain in the differentiability order.

We now introduce a nonlocal tail space. We say that $u\in L^{p}(0,T;L^{1}_{2s}(\mathbb{R}^{n}))$ if 
\begin{equation*}
    \left\|\int_{\mathbb{R}^{n}}\frac{|u(x,\cdot)|}{(1+|x|)^{n+2s}}\dx\right\|_{L^{p}(0,T)}<\infty
\end{equation*}
for any $p\in [1,\infty]$. In particular, we write
\begin{equation*}
    \Tail_{p,2s}\left(u;B_{r}(x_{0})\times I\right)=\left(\dashint_{I}\left(r^{2s}\int_{\mathbb{R}^{n}\setminus B_{r}(x_{0})}\frac{|u(y,t)|}{|y-x_{0}|^{n+2s}}\dy\right)^{p}\dt\right)^{\frac{1}{p}} \quad\text{if }p\in[1,\infty)
\end{equation*}
and
\begin{equation*}
    \Tail_{\infty,2s}\left(u;B_{r}(x_{0})\times I\right)=\sup_{t\in I}\left(r^{2s}\int_{\mathbb{R}^{n}\setminus B_{r}(x_{0})}\frac{|u(y,t)|}{|y-x_{0}|^{n+2s}}\dy\right),
\end{equation*}
where $I\subset\mathbb{R}$ is a bounded time interval.
We note from H\"older's inequality that for any $1\leq p_{1}\leq p_{2}\leq \infty$,
\begin{equation}
\label{cz non pa : tailho}
    \Tail_{p_{1},2s}\left(u;B_{r}(x_{0})\times I\right)\leq \Tail_{p_{2},2s}\left(u;B_{r}(x_{0})\times I\right).
\end{equation}
As usual, a solution to \eqref{cz non pa : eq1} is defined in the weak sense as below.
\begin{defn}
Let $f\in L^{2}_{\mathrm{loc}}\left(0,T;L^{1}_{s}(\mathbb{R}^{n})\right)$ with $d_{0}f\in L^{2}_{\mathrm{loc}}\left(\Omega\times\Omega_T\,;\,\frac{\dx\dy\dt}{|x-y|^{n}}\right)$ and $g\in L_{\mathrm{loc}}^{\frac{2(n+2s)}{n+4s}}(\Omega_{T})$. We say that 
\begin{equation*}
    u\in L_{\mathrm{loc}}^{2}\left(0,T;W_{\mathrm{loc}}^{s,2}(\Omega)\right)\cap 
C_{\mathrm{loc}}\left(0,T;L_{\mathrm{loc}}^{2}(\Omega)\right)\cap L_{\mathrm{loc}}^{\infty}\left(0,T;L^{1}_{2s}(\mathbb{R}^{n})\right)
\end{equation*}is a weak solution to \eqref{cz non pa : eq1} if 
\begin{flalign*}
    &\int_{t_{1}}^{t_{2}}\int_{\Omega}-u\phi_{t}\dx\dt+\int_{t_{1}}^{t_{2}}\int_{\mathbb{R}^{n}}\int_{\mathbb{R}^{n}}\Phi\left(\frac{u(x,t)-u(y,t)}{|x-y|^{s}}\right)(\phi(x,t)-\phi(y,t))\frac{A(x,y,t)}{|x-y|^{n+s}}\dx\dy\dt\\
    &=-\int_{\Omega}(u\phi)(x,t)\dx\Bigg\rvert_{t=t_{1}}^{t=t_{2}}+\int_{t_{1}}^{t_{2}}\int_{\mathbb{R}^{n}}\int_{\mathbb{R}^{n}}(f(x,t)-f(y,t))(\phi(x,t)-\phi(y,t))\frac{1}{|x-y|^{n+s}}\dx\dy\dt\\
    &\quad+\int_{t_{1}}^{t_{2}}\int_{\Omega}g\phi\dx\dt
\end{flalign*}
holds for any $\phi\in L^{2}(t_{1},t_{2};W^{s,2}(\Omega))\cap W^{1,2}\left(t_{1},t_{2};L^{2}(\Omega)\right)$ with compact spatial support contained in $\Omega$ and $(t_{1},t_{2})\Subset (0,T)$.
\end{defn}

We next introduce the notion of $(\delta,R)$-vanishing condition on $A$, for some $\delta$ and $R>0$. We say that $A$ is $(\delta,R)$-vanishing in $\Omega_T$, if 
\begin{align*}
    \sup_{0<r\leq R, \, Q_r(z_0)\subset\Omega_T}\dashint_{\Lambda_{r}(t_0)}\dashint_{B_r(x_0)}\dashint_{B_r(x_0)} |A(x,y,t)-(A)_{r,x_0}(t)|dx\,dy\,dt \leq \delta,
\end{align*}
where $z_0=(x_0,t_0)$ and \begin{equation}
\label{cz non pa : defn ava ker}
(A)_{r,x_{0}}(t)\coloneqq \dashint_{B_{r}(x_{0})}\dashint_{B_{r}(x_{0})}A(x,y,t)\dx\dy.
\end{equation}

We now observe the following scaling invariance property for the problem \eqref{cz non pa : eq1}. 
\begin{lem}
\label{cz non pa : scalinglem}
Let $Q_{r}(z_{0})\Subset\Omega_{T}$. Suppose that $A$ is $(\delta,R)$-vanishing in $\Omega_{T}$. Then  
\begin{equation*}
    \tilde{u}(x,t)=\frac{u(rx+x_{0},r^{2s}t+t_{0})}{r^{s}}
\end{equation*}
is a weak solution to 
\begin{equation*}
   \tilde{u}_{t}+\mathcal{L}^{\Phi}_{\tilde{A}}\tilde{u}=(-\Delta)^{\frac{s}{2}}\tilde{f}+\tilde{g} \quad\text{ in }Q_{1},
\end{equation*}
where 
\[\tilde{f}(x,t)=f(rx+x_{0},r^{2s}t+t_{0}),\quad \tilde{g}(x,t)=r^{s}g(rx+x_{0},r^{2s}t+t_{0})\] and $\tilde{A}(x,y,t)=A(rx+x_{0},ry+x_{0},r^{2s}t+t_{0})$ is $\left(\delta,\frac{R}{r}\right)$-vanishing in $Q_{1}$.
\end{lem}
We now introduce our main results.
\begin{thm} 
\label{cz non pa : main thm}
Let $u$ be a weak solution to problem \eqref{cz non pa : eq1}. Let $R>0$ and $q\in(2,\infty)$ be given, and 
fix $\sigma\in \left(0,\left(1-\frac{2}{q}\right)\min\left\{s-\frac{2s}{q},1-s\right\}\right)$. Then there is a constant $\delta=\delta(n,s,L,q,\sigma)\in(0,1)$ such that if $A$ is $(\delta,R)$-vanishing in $\Omega_{T}$,  $f\in L^{q}_{\mathrm{loc}}\left(0,T;L^{1}_{s}\left(\mathbb{R}^{n}\right)\right)$ with $ d_{0}f\in L^{q}_{\mathrm{loc}}\left(\frac{\dx\dy\dt}{|x-y|^{n+\sigma q}}\,;\,\Omega\times\Omega_T\right)$ and $g\in L^{\frac{q\left(n+2s+\frac{2\sigma q}{q-2}\right)}{n+4s}}_{\mathrm{loc}}(\Omega_{T})$, then  $d_{s}u\in L^{q}_{\mathrm{loc}}\left(\frac{\dx\dy\dt}{|x-y|^{n+\sigma q}}\,;\,\Omega\times\Omega_T\right)$. Moreover, there is a constant $c=c(n,s,L,q,\sigma)$ such that 
{\small\begin{equation}
\label{cz non pa : main est}
\begin{aligned}
   &\left(\dashint_{\Lambda_{\frac{r}{2}}(t_{0})}\dashint_{B_{\frac{r}{2}}(x_{0})}\int_{B_{\frac{r}{2}}(x_{0})}\left|d_{s+\sigma}u\right|^{q}\frac{\dx\dy\dt}{|x-y|^{n}}\right)^{\frac{1}{q}}\\
   &\leq c\left(\dashint_{\Lambda_{r}(t_{0})}\dashint_{B_{r}(x_{0})}\int_{B_{r}(x_{0})}\left|\frac{d_{s}u}{r^{\sigma}}\right|^{2}\frac{\dx\dy\dt}{|x-y|^{n}}\right)^{\frac{1}{2}}+c\Tail_{\infty,2s}\left(\frac{u-(u)_{B_{r}(x_{0})}(t)}{r^{s+\sigma}}\,;\, Q_{r}(z_{0})\right)\\
   &\quad+c\sup_{t\in \Lambda_{r}(t_{0})}\left(\dashint_{B_{r}(x_{0})}\frac{|u-(u)_{Q_{r}(z_{0})}|^{2}}{r^{2s+2\sigma}}\dx\right)^{\frac{1}{2}}+c\left(\dashint_{Q_{r}(z_{0})}\left(r^{{s-\sigma}}|g|\right)^{\frac{q\left(n+2s+\frac{2\sigma q}{q-2}\right)}{n+4s}}\dx\dt\right)^{\frac{n+4s}{q\left(n+2s+\frac{2\sigma q}{q-2}\right)}}\\
    &\quad+c\left(\dashint_{\Lambda_{r}(t_{0})}\dashint_{B_{r}(x_{0})}\int_{B_{r}(x_{0})}\left|{d_{\sigma}f}\right|^{q}\frac{\dx\dy\dt}{|x-y|^{n}}\right)^{\frac{1}{q}}
    +c\Tail_{q,s}\left(\frac{f-(f)_{B_{r}(x_{0})}(t)}{r^{\sigma}}\,;\, Q_{r}(z_{0})\right),
\end{aligned}
\end{equation}}
\\
\noindent
whenever $Q_{r}(z_{0})\Subset \Omega_{T}$ and $r\in(0,R]$.
\end{thm}

\begin{rmk}\label{cz non pa : remk.1}
A few comments are in order for the restricted range of  $\sigma$ in Theorem \ref{cz non pa : main thm}. In the elliptic case, we observe that a similar type of result holds for all $\sigma\in\left(0,\min\left\{s,1-s\right\}\right)$ (see \cite[Theorem 1.2]{BKc}). However, in our case, to handle the nonlocal parabolic tail induced by the non-homogeneous term $f$, we have to impose the condition $\tau\in\left(0,s-\frac{2s}{q}\right)$ (see \eqref{cz non pa : oneretildeqnec} below). Therefore, from the observation \eqref{cz non pa : observ}, we deduce that the Calder\'on-Zygmund type estimate \eqref{cz non pa : main est} holds under $\sigma\in \left(0,\left(1-\frac{2}{q}\right)\min\left\{s-\frac{2s}{q},1-s\right\}\right)$.  In this regard, if $f=0$, then this restriction is removed and the results hold for a broader range of $\sigma$ (see Theorem \ref{cz non pa : main thm wrt g} below).  
\end{rmk}

\begin{rmk}
As we pointed out earlier, in the elliptic case, from regularity results for \eqref{cz non pa : elleq} with $g=0$, we deduce a higher regularity of a weak solution $u$ to \eqref{cz non pa : elleq}. However, in the parabolic case,  it does not hold. More specifically, if $g\in L^{\tilde{q}}_{\mathrm{loc}}(\Omega_{T})$ for some $\tilde{q}>2$, then we find a solution $f\in L^{\tilde{q}}_{\mathrm{loc}}\left(0,T;H^{s,\tilde{q}}_{\mathrm{loc}}(\Omega)\right)\cap L^{\tilde{q}}_{\mathrm{loc}}(0,T;L^{1}_{s}(\mathbb{R}^{n}))$ to
\begin{equation*}
    (-\Delta)^{\frac{s}{2}}f(\cdot,t)=g(\cdot,t)\quad\text{in }\Omega
\end{equation*}
for a.e. $t\in(0,T)$ (see \cite[Subsection 1.2]{BKc}). This implies that $u$ is a weak solution to
\begin{equation*}
    u_{t}+\mathcal{L}_{A}^{\Phi}u=(-\Delta)^{\frac{s}{2}}f.
\end{equation*} By Theorem \ref{cz non pa : main thm}, we deduce that 
\[u\in L^{\tilde{q}}_{\mathrm{loc}}\left(0,T;W_{\mathrm{loc}}^{s+\sigma,\tilde{q}}(\Omega)\right)\text{ for any } \sigma\in \left(0,\left(1-\frac{2}{\tilde{q}}\right)\min\left\{s-\frac{2s}{\tilde{q}},1-s\right\}\right).\] 
We now select $\tilde{q}=\frac{q\left(n+2s+\frac{2\sigma q}{q-2}\right)}{n+4s}$. Since $\tilde{q}<q$, we do not obtain the desired result given in Theorem \ref{cz non pa : main thm wrt g}. Therefore, we consider a more general non-homogeneous term which consists of $(-\Delta)^{\frac{s}{2}}f$ and $g$.
\end{rmk}

On account of Remark \ref{cz non pa : remk.1}, we obtain an improved Sobolev regularity when we consider only non-divergence data $g$.
\begin{thm}
\label{cz non pa : main thm wrt g}
Let $u$ be a weak solution to problem \eqref{cz non pa : eq1} with $f=0$. Let $R>0$ and $q\in(2,\infty)$ be given, and 
fix $\sigma\in \left(0,\left(1-\frac{2}{q}\right)\min\left\{s,1-s\right\}\right)$. Then there is a constant $\delta=\delta(n,s,L,q,\sigma)\in(0,1)$ such that if $A$ is $(\delta,R)$-vanishing in $\Omega_{T}$ and $g\in L^{\frac{q\left(n+2s+\frac{2\sigma q}{q-2}\right)}{n+4s}}_{\mathrm{loc}}(\Omega_{T})$, then  $u\in L^{q}_{\mathrm{loc}}\left(0,T;W^{s+\sigma,q}_{\mathrm{loc}}(\Omega)\right)$ with the estimate \eqref{cz non pa : main est}.
\end{thm}

\begin{rmk}
Let us compare the local Calder\'on-Zygmund theory and the nonlocal ones. 
It is known that if $v$ is a weak solution to 
\begin{equation*}
    v_{t}-\ddiv(BDv)=g,
\end{equation*}
then we obtain the following implication
\begin{equation*}
    g\in L^{\frac{q(n+2)}{n+4}}_{\mathrm{loc}}(\Omega_{T})\Longrightarrow v\in L^{q}_{\mathrm{loc}}(0,T;W^{1,q}_{\mathrm{loc}}(\Omega)),
\end{equation*}
whenever the coefficient $B$ is $\delta$-vanishing for sufficiently small $\delta$ depending only on the given data (see, for instance, \cite{Bg}). We observe that in the limiting case when $\sigma\to 0$ and $s\to 1$, the result in Theorem \ref{cz non pa : main thm wrt g} is the same as the local one.  However, in the nonlocal case, although the kernel coefficient is discontinuous, Theorem \ref{cz non pa : main thm wrt g} implies that the solution $u$ obtains not only higher integrability but also higher differentiability along the Sobolev-scale. This is in some sense a purely nonlocal phenomenon, since in order to observe such results in the local case, the coefficient $B$ is assumed to have some fractional regularity in the literature (see \cite{KMu}). A similar phenomenon for the nonlocal equations is observed in \cite{Nv,Ni,KMS1,ABES} and references therein. 
\end{rmk}

\subsection{Working methods and novelties}
We now briefly explain our approach to obtain the desired estimates \eqref{cz non pa : selfest} and \eqref{cz non pa : main est}.  As usual, keeping the relation \eqref{cz non pa : observ} in mind, to prove that the fractional gradient term $d_{s}u$ is in $ L^{q}_{\mathrm{loc}}\left(\Omega\times\Omega_{T}\,;\,\frac{\dx\dy\dt}{|x-y|^{n+\sigma q}}\right)$, it suffices to show that 
\begin{equation*}
    \int_{0}^{\infty}\lambda^{q-1}\mu_{\tau,t}\left(\left\{(x,y,t)\in\mathcal{Q}\,:\,|D^{\tau}d_{s}u|(x,y,t)>\lambda\right\}\right)\,d\lambda<\infty
\end{equation*}
holds for any  $\mathcal{Q}=B\times B\times\Lambda$, where $B\Subset \Omega$ is a ball and $\Lambda\Subset (0,T)$ is a time interval.  To do this, we construct coverings for upper level sets of $|D^{\tau}d_{s}u|$  inspired by the maximal function free technique as introduced in \cite{AMp}. 
Using an exit-time argument, we are able to construct coverings for the diagonal part of the upper level sets. For the off-diagonal part, we use a reverse H\"older-type inequality (see \eqref{cz non pa : reverseestimate}, below) which is obtained regardless of the information that $u$ solves \eqref{cz non pa : eq1}. As in \cite[Lemma 5.3]{KMS1}, this inequality contains additional correction terms involving diagonal cylinders which induce some serious difficulties, as such cylinders do not come from any exit-time argument. We would like to mention that in the elliptic case, Calder\'on-Zygmund cube decomposition and an involved combinatorial argument are used to overcome these difficulties (see \cite{KMS1}). However, in the parabolic case, there are additional difficulties, since the correction terms contain $L^{2}$-oscillation integrals by sup norm term (see the second term of the right-hand side in \eqref{cz non pa : peneftnlofu} and Lemma \ref{cz non pa : off rever lem} below). To this end, we employ Vitali's covering lemma along with an exit-time argument instead of Calder\'on-Zygmund cube decomposition in order to construct coverings for upper-level sets of $|D^{\tau}d_{s}u|$. We would like to mention that this argument is new even in the elliptic case, and we believe that this argument can be applied to degenerate or singular nonlocal parabolic equations. We also point out that due to the appearance of the additional $L^{2}$-oscillation integrals by sup norm term,  functionals used to apply an exit-time argument also contain a term of a similar kind which is usually not observed in the local parabolic problems (see \eqref{cz non pa : defnTheta}). We will elaborate on how to take care of this term while obtaining a good bound on the measure of exit-time cylinders in Remarks \ref{cz non pa : imprmk1} and \ref{cz non pa : imprmk2} below.
Moreover, we use a non-trivial exit time radius in the covering arguments in light of the rigorous tail estimates as in \eqref{cz non pa : findj0u} and \eqref{cz non pa : findj0f}, since the additional $L^2$-oscillation terms by sup norm are estimated by the sum of $L^{2}$-integral of $d_{s}u$ and $d_{0}f$, $L^{\frac{2(n+2s)}{n+4s}}$-integral of $g$  and tail terms of $u$ and $f$ (see Lemma \ref{cz non pa : RHI}).
Consequently, by constructing suitable coverings, we are able to make use of comparison estimates, which further require some higher H\"older continuity estimates along with a self-improving property for limiting equations, and a boot strap argument to finally obtain the desired result (see Section \ref{cz non pa :lqsection}).

We would like to remark that a similar covering argument along with Gerhing's Lemma (in the spirit of \cite{KMS1}) can be used to obtain a self-improving property of weak solutions to \eqref{cz non pa : eq1} without imposing any regularity assumption on the kernel coefficient $A$. Indeed, for the sake of completion, we prove a self-improving property of weak solutions to \eqref{cz non pa : eq1} with $f=g=0$ (see Appendix \ref{cz non pa : appen a}).
In addition, this result generalizes the ones given in \cite{ABES} by allowing nonlinear structure assumptions on the nonlocal operator.

\subsection{Plan of the paper}
This paper is organized as follows. In Section 2, we introduce some notations, embedding inequalities, properties of the measure $\mu_{\tau,t}$, tail estimates. In Section 3, we derive some energy estimates. Section 4 is devoted to establishing some comparison estimates. In Section 5, we construct coverings of upper level sets of fractional gradients for weak solutions. Section 6 contains the proof of the main theorem. We end the paper with two appendices. In the first appendix, we give the proof of a self-improving property for weak solutions to \eqref{cz non pa : eq1} with $f=g=0$, whereas the second appendix deals with the existence of a weak solution to the corresponding boundary value problem of \eqref{cz non pa : eq1}.

\section{Preliminaries and Notations}
As usual, we write $c$ to mean a general constant equal to or bigger than 1 and it possibly changes from line to line. In addition, we employ parentheses to denote the relevant dependencies on parameters such as $c=c(n,s)$, and we denote
\begin{equation*}
    \mathsf{data}=\mathsf{data}(n,s,L,q,\sigma).
\end{equation*}
For $a,b\in\mathbb{R}^+$, by the notation $a\approx_{n,s} b$, we mean  that there is a constant $c=c(n,s)$ such that $\frac{b}{c}\leq a\leq cb$.
A generic point $z\in\mathbb{R}^{n+1}$ will be denoted by 
\[z=(x,t)\in \mathbb{R}^n\times\mathbb{R}.\]
We write parameters $i,j,l,k$ and $m$ to mean nonnegative integers.
We denote a time interval as 
\[\Lambda_{r}(t_{0})\coloneqq (t_{0}-r^{2s},t_{0}+r^{2s}).\]
The parabolic cylinder is defined by
\[Q_{r}(z_{0})=B_{r}(x_{0})\times \Lambda_{r}(t_{0}),\]
where $B_{r}(x_{0})$ denotes the ball in $\mathbb{R}^{n}$ centered at $x_{0}$ with radius $r$. We write 
\begin{equation*}
    \mathcal{B}_{r}(x_{0},y_{0})=B_{r}(x_{0})\times B_{r}(y_{0}),\quad \mathcal{B}_{r}(x_{0})=B_{r}(x_{0})\times B_{r}(x_{0})
\end{equation*}
and
\begin{equation*}
    \mathcal{Q}_{r}(x_{0},y_{0},t_{0})=\mathcal{B}_{r}(x_{0},y_{0})\times \Lambda_{r}(t_{0}),\quad\mathcal{Q}_{r}(x_{0},t_{0})=\mathcal{B}_{r}(x_{0})\times \Lambda_{r}(t_{0})
\end{equation*}
for any $x_{0},y_{0}\in\mathbb{R}^{n}$, $t_{0}\in \mathbb{R}$ and $r>0$.
For a given measurable function $h:\Omega_{T}\to\mathbb{R}$, we write for any $Q_{r}(z_{0})\subset\Omega_{T}$,
\begin{equation*}
    (h)_{B_{r}(x_{0})}(t)=\dashint_{B_{r}(x_{0})}h(x,t)\dx\quad\text{and}\quad (h)_{Q_{r}(x_{0})}=\dashint_{Q_{r}(z_{0})}h(z)\dz.
\end{equation*}
We denote the parabolic Sobolev conjugate of $p\in[1,\infty)$ by 
\begin{equation}
\label{cz non pa : paraconj}
p_{\#}=p\left(1+\frac{2s}{n}\right).
\end{equation}

We are going to mention some lemmas starting with the following embedding result.
\begin{lem}(see \cite[Lemma 2.3]{DZZ})
\label{cz non pa : embed lem}
Let $p\in[1,\infty)$ and $h\in L^{p}(0,T;W^{s,p}(B_{r}))\cap L^{\infty}(0,T;L^{2}(B_{r}))$. Then there is a constant $c=c(n,s,p)$ such that
{\small\begin{equation*}
\begin{aligned}
    \int_{0}^{T}\dashint_{B_{r}}|h|^{p_{\#}}\dz&\leq c\left(r^{sp} \int_{0}^{T}\int_{B_{r}}\dashint_{B_{r}}\frac{|h(x,t)-h(y,t)|^{p}}{|x-y|^{n+sp}}\dx\dy\dt+ \int_{0}^{T}\dashint_{B_{r}}|h|^{p}\dz\right)\times \left(\sup_{t\in(0,T)}\dashint_{B_{r}}|h|^{2}\dx\right)^{\frac{sp}{n}}.
\end{aligned}
\end{equation*}}
\noindent
In particular, we have that
{\small\begin{equation*}
\begin{aligned}
    \int_{0}^{T}\dashint_{B_{r}}|h-(h)_{B_{r}}(t)|^{p_{\#}}\dz&\leq c\left(r^{sp} \int_{0}^{T}\int_{B_{r}}\dashint_{B_{r}}\frac{|h(x,t)-h(y,t)|^{p}}{|x-y|^{n+sp}}\dx\dy\dt\right)\\
    &\quad\times \left(\sup_{t\in(0,T)}\dashint_{B_{r}}|h-(h)_{B_{r}}(t)|^{2}\dx\right)^{\frac{sp}{n}}.
\end{aligned}
\end{equation*}}
\end{lem}

Next, we list some properties of the measure $\mu_{\tau,t}$.

\begin{lem}There exists a constant $C_{n}$ depending only on $n$ such that
\begin{enumerate}
\item For any $x_{0}\in\mathbb{R}^{n}$, $t_0\in\mathbb{R}$  and $R>0$, the following holds
\begin{equation}
\label{cz non pa : size of b mut}
\mu_{\tau,t}\left(\mathcal{Q}_{R}(x_{0},t_0)\right)=C_{n}\frac{R^{n+2s+2\tau}}{\tau}.
\end{equation}
\item Let $\rho$ and $R$ be any positive numbers, and let $x_{0}\in\mathbb{R}^{n}$. Then 
\begin{equation}
\label{cz non pa : doubling}
\frac{\mu_{\tau,t}\left(\mathcal{Q}_{R}(x_{0},t_0)\right)}{\mu_{\tau,t}\left(\mathcal{Q}_{\rho}(x_{0},t_0)\right)}=\left(\frac{R}{\rho}\right)^{n+2s+2\tau}.
\end{equation}
\item Let $\mathcal{K}_{r}(x_{0},y_{0})$ be any cube in $\mathcal{B}_{R}$ for $r,R>0$ and $x_{0},y_{0}\in\mathbb{R}^{n}$. Then 
\begin{equation}
\label{cz non pa : inclusion measure}
    \frac{\mu_{\tau,t}\left(\mathcal{Q}_{R}\right)}{\mu_{\tau,t}\big(\mathcal{K}_{r}(x_{0},y_{0})\times \Lambda_{r}\big)}\leq 2^{n}\frac{C_{n}}{\tau}\left(\frac{R}{r}\right)^{2n+2s}.
\end{equation}
\item Let $a\geq1$. Then we have 
\begin{align}\label{cz non pa : meascom}
    \mu_{\tau,t}\left(\mathcal{Q}_{aR}(x_0,y_0,t_0)\right)\leq {ca^{2n+2s}}{\tau}^{-1}\mu_{\tau,t}\left(\mathcal{Q}_{R}(x_0,y_0,t_0)\right)
\end{align}
for some constant $c=c(n)$.
\end{enumerate}
\end{lem}
\begin{proof}
    Since the proofs of (1)-(3) follow from \cite[Lemma 2.1]{BKc}, we only give the proof of (4). If $x_0=y_0$, (4) follows from (2). We assume $x_0\neq y_0$. Let 
    \begin{align*}
        D=\mathrm{diam}(B_{aR}(x_0),B_{aR}(y_0)).
    \end{align*}
    If $D\leq 4a$, then $\mathcal{Q}_{aR}(x_0,y_0,t_0)\subset \mathcal{Q}_{10aR}(x_0,t_0)$. Thus by (3), we get 
    \begin{align*}
        \mu_{\tau,t}\left(\mathcal{Q}_{aR}(x_0,y_0,t_0)\right)\leq \mu_{\tau,t}\left(\mathcal{Q}_{10aR}(x_0,t_0)\right)\leq \frac{c a^{2n+2s}}{\tau}\mu_{\tau,t}\left(\mathcal{Q}_{R}(x_0,y_0,t_0)\right)
    \end{align*}
    for some constant $c=c(n)$. If $D>4a$, we observe 
    \begin{align*}
        D\leq|x-y|\leq 3D\quad\text{for any }x\in B_{aR}(x_0)\text{ and }y\in B_{aR}(y_0).
    \end{align*}
    Thus we have
    \begin{align*}
        \mu_{\tau,t}\left(\mathcal{Q}_{aR}(x_0,y_0,t_0)\right)\leq \frac{(aR)^{2n+2s}}{D^{n-2\tau}}\leq ca^{2n+2s}\mu_{\tau,t}\left(\mathcal{Q}_{R}(x_0,y_0,t_0)\right)
    \end{align*}
    for some constant $c=c(n)$. This completes the proof.
\end{proof}


We now give some useful estimates to control the parabolic tail.
\begin{lem}
\label{cz non pa : taillem}
Let $h\in L^{p}\left(\Lambda_{2}; L^{1}_{2\beta}(\mathbb{R}^{n})\right)$ and $D^{\tau}d_{\tilde{s}}h\in L^{p}\left(\mathcal{Q}_{2};\dmutt\right)$ where $\beta\in(0,1)$, $\tilde{s}\in[0,\beta]$ and $p\in[1,\infty)$. Let $Q_{\rho}(z_{0})\Subset Q_{2}$, where $z_{0}\in Q_{r_1}$ with $0<r_1<2$. Suppose that there is a natural number $l\geq 1$ such that
\begin{equation*}
    Q_{2^{l}\rho}(z_{0})\Subset Q_{2}.
\end{equation*} 
Then for any integer $k\in[0,l]$, there are constants $c_{p}=c(n,p)$ and $\tilde{c}=\tilde{c}(n)$ independent of $k$ and $l$ such that
{\small\begin{equation}
\label{cz non pa : tailestimate}
\begin{aligned}
    &\Tail_{p,2\beta}\left(\frac{h-(h)_{B_{\rho}(x_{0})}(t)}{\rho^{\tilde{s}+\tau}};Q_{\rho}(z_{0})\right)\\
    &\leq c_{p}\mathbb{A}_{k}\sum_{i=1}^{k}2^{i\left(-2\beta+\tilde{s}+\tau+\frac{2s}{p}\right)}\left(\frac{1}{\tau}\dashint_{\mathcal{Q}_{2^{i}\rho}(z_{0})}|D^{\tau}d_{\tilde{s}}h|^{p}\dmutt\right)^{\frac{1}{p}}+c_{p}\mathbb{A}_{k+1-l}\frac{2^{-2\beta l+\tilde{s}+\tau+\frac{2s}{p}}}{\rho^{\tilde{s}+\tau+\frac{2s}{p}}}\left(\frac{1}{\tau}\dashint_{\mathcal{Q}_{2}}|D^{\tau}d_{\tilde{s}}h|^{p}\dmutt\right)^{\frac{1}{p}}\\
    &\quad+\tilde{c}\mathbb{A}_{l-k}\sum_{j=k+1}^{l}2^{j(-2\beta+\tilde{s}+\tau)}\sup_{t\in \Lambda_{2^{j}\rho}(t_{0})}\dashint_{B_{2^{j}\rho}(x_{0})}\frac{|h-(h)_{B_{2^{j}\rho}(x_{0})}(t)|}{(2^{j}\rho)^{\tilde{s}+\tau}}\dx\\
    &\quad+\tilde{c}\mathbb{A}_{l-k}\frac{2^{-2\beta l+\tilde{s}+\tau}}{\rho^{\tilde{s}+\tau}}\sup_{t\in \Lambda_{2}}\dashint_{B_{2}}\frac{|h-(h)_{B_{2}}(t)|}{2^{s+\tau}}\dx+\frac{\tilde{c}}{(2-r_{1})^{n+2s}}\left(\frac{2}{\rho}\right)^{-2\beta+\tilde{s}+\tau-\frac{2s}{p}}\Tail_{p,2\beta}\left(\frac{h-(h)_{B_{2}}(t)}{2^{\tilde{s}+\tau}};Q_{2}\right),
\end{aligned}
\end{equation}}
where  
\begin{equation*}
    \mathbb{A}_{m}=\begin{cases}
        1&\text{if }m=1,2,\ldots,\\
        0&\text{if }m=0,-1,\ldots.
    \end{cases}
\end{equation*}

\end{lem}

\begin{proof}
Using Minkowski's inequality, we get that
\begin{equation*}
\begin{aligned}
    &\Tail_{p,2\beta}\left(\frac{h-(h)_{B_{\rho}(x_{0})}(t)}{\rho^{\tilde{s}+\tau}};Q_{\rho}(z_{0})\right)\\
    &\leq \sum_{i=1}^{l}\left(\dashint_{\Lambda_{\rho}(t_{0})}\left(\int_{B_{2^{i}\rho}(x_{0})\setminus B_{2^{i-1}\rho}(x_{0})}\rho^{2\beta-\tilde{s}-\tau}\frac{|h-(h)_{B_{\rho}(x_{0})}(t)|}{|y-x_{0}|^{n+2\beta}}\dy\right)^{p}\dt\right)^{\frac{1}{p}}\\
    &\quad+\left(\dashint_{\Lambda_{\rho}(t_{0})}\left(\int_{B_{2}\setminus B_{2^{l}\rho}(x_{0})}\rho^{2\beta-\tilde{s}-\tau}\frac{|h-(h)_{B_{\rho}(x_{0})}(t)|}{|y-x_{0}|^{n+2\beta}}\dy\right)^{p}\dt\right)^{\frac{1}{p}}\\
    &\quad+\left(\dashint_{\Lambda_{\rho}(t_{0})}\left(\int_{\mathbb{R}^{n}\setminus B_{2}}\rho^{2\beta-\tilde{s}-\tau}\frac{|h-(h)_{B_{\rho}(x_{0})}(t)|}{|y-x_{0}|^{n+2\beta}}\dy\right)^{p}\dt\right)^{\frac{1}{p}}\eqqcolon\sum_{i=1}^{l}I_{i}+\mathcal{T}_{1}+\mathcal{T}_{2}.
\end{aligned}
\end{equation*}
From the estimate of $T_{k}^{\frac{1}{p-1}}$ in \cite[Lemma 2.6]{BKc}, we have
\begin{equation*}
\begin{aligned}
    I_{i}\leq \frac{2^{-2\beta i}}{\rho^{\tilde{s}+\tau}}\sum_{j=1}^{i}\left(\dashint_{\Lambda_{\rho}(t_{0})}\left(\dashint_{B_{2^{j}\rho}(x_{0})}|h-(h)_{B_{2^{j}\rho}(x_{0})}(t)|\dy\right)^{p}\dt\right)^{\frac{1}{p}}\eqqcolon 2^{-2\beta i}\sum_{j=1}^{i}I_{i,j}.
\end{aligned} 
\end{equation*}
Using H\"older's inequality and then following the same line as for (2.12) in \cite[Lemma 2.6]{BKc}, we obtain
\begin{equation*}
\begin{aligned}
\sum_{j=1}^{i}I_{i,j}&\leq c_{p}\sum_{j=1}^{k}\mathbb{A}_{k}2^{j\left(\tilde{s}+\tau+\frac{2s}{p}\right)}\left(\frac{1}{\tau}\dashint_{\mathcal{Q}_{2^{j}\rho}(z_{0})}|D^{\tau}d_{\tilde{s}}h|^{p}\dmutt\right)^{\frac{1}{p}}\\
&\quad+\sum_{j=k+1}^{i}\mathbb{A}_{i-k}2^{j(\tilde{s}+\tau)}\sup_{t\in \Lambda_{\rho}(t_{0})}\dashint_{B_{2^{j}\rho}(x_{0})}\frac{|h-(h)_{B_{2^{j}\rho}(x_{0})}(t)|}{(2^{j}\rho)^{\tilde{s}+\tau}}\dx,
\end{aligned}
\end{equation*}
for some constant $c_{p}=2^{n+2p}$, where we have taken supremum for the time variable if $j\geq k+1$.
Similarly, we estimate $\mathcal{T}_{1}$ as
\begin{equation*}
\begin{aligned}
    \mathcal{T}_{1}&\leq 2^{-2\beta l}\sum_{i=1}^{l}I_{l,i}+c_{p}\mathbb{A}_{k+1-l}2^{-2\beta l}\left(\frac{2}{\rho}\right)^{\tilde{s}+\tau+\frac{2s}{p}}\left(\frac{1}{\tau}\dashint_{\mathcal{Q}_{2}}|D^{\tau}d_{\tilde{s}}h|^{p}\dmutt\right)^{\frac{1}{p}}\\
    &\quad+c\mathbb{A}_{l-k}2^{-2\beta l}\left(\frac{2}{\rho}\right)^{\tilde{s}+\tau}\sup_{t\in \Lambda_{2}}\dashint_{B_{2}}\frac{|h-(h)_{B_{2}}(t)|}{2^{\tilde{s}+\tau}}\dx\eqqcolon \mathcal{T}_{1,1},
\end{aligned}
\end{equation*}
where $c=c(n)$.
Lastly, we estimate $\mathcal{T}_{2}$ as
\begin{equation}
\label{cz non pa : taiilestimt2}
\begin{aligned}
    \mathcal{T}_{2}&\leq c\mathcal{T}_{1,1}+c\left(\dashint_{\Lambda_{\rho}}\left(\int_{\mathbb{R}^{n}\setminus B_{2}}\left(\frac{2}{\rho}\right)^{-2\beta+\tilde{s}+\tau}\left(\frac{2}{2-r_{1}}\right)^{n+2\beta}2^{-2\beta+\tilde{s}+\tau}\frac{|h-(h)_{B_{2}}(t)|}{|y|^{n+2\beta}}\dy\right)^{p}\dt\right)^{\frac{1}{p
    }},
\end{aligned}
\end{equation}
where for the last term we have used the fact that
\begin{equation*}
    |y-x_{0}|\geq |y|-|x_{0}|\geq \frac{|y|(2-r_{1})}{2} \quad\text{for any }y\in \mathbb{R}^{n}\setminus B_{2}.
\end{equation*}
We combine all the estimates $I_{i}$, $\mathcal{T}_{1}$ and $\mathcal{T}_{2}$ and use Fubini's theorem as in \cite[Lemma 2.6]{BKc} to get the desired result \eqref{cz non pa : tailestimate}.
\end{proof}
\begin{rmk}
\label{cz non pa : cp}
By tracking the choice of the constant $c_{p}$ appearing in Lemma \ref{cz non pa : taillem}, we find that $c_{p}\leq c_{q}$ if $p\leq q$.
\end{rmk}
We end this section with the following iteration lemma.
\begin{lem}(See \cite[Lemma 6.1]{G})
\label{cz non pa : technicallemma}
Let $\phi:[1,2]\to \mathbb{R}$ be a nonnegative bounded function. For $1\leq r_{1}<r_{2}\leq 2$, we assume that
\begin{equation*}
    \phi(r_{1})\leq \frac{1}{2}\phi(r_{2})+\frac{\lambda_{0}}{\left(r_{2}-r_{1}\right)^{\frac{5n}{s}}},
\end{equation*}
where $\lambda_{0}>0$. Then, 
\begin{equation*}
    \phi(1)\leq c\lambda_{0}
\end{equation*}for some constant $c=c(n,s)$.
\end{lem}

\section{Energy estimates and the Sobolev-Poincar\'e inequalities}
In this section, we give energy estimates and derive Sobolev-Poincar\'e type inequalities from the energy estimates.
We first give an energy inequality of a weak solution $u$ to \eqref{cz non pa : eq1}.
\begin{lem} 
Let $u$ be a local weak solution to \eqref{cz non pa : eq1}. Let $0<\rho<r\leq2\rho$ with $Q_{2\rho}(z_{0})\Subset \Omega_{T}$. Then, there is a constant $c=c(n,s,L)$ such that
\begin{equation}
\label{cz non pa : CE} 
\begin{aligned}
    &\left[\dashint_{\Lambda_{\rho}(t_{0})}\dashint_{B_{\rho}(x_{0})}\int_{B_{\rho}(x_{0})}\frac{|u(x,t)-u(y,t)|^{2}}{|x-y|^{n+2s}}\dx\dy\dt+
    \sup_{t\in\Lambda_{\rho}(t_{0})}\dashint_{B_{\rho}(x_{0})}\frac{|u(x,t)-k|^{2}}{\rho^{2s}}\dx\right]\\
    &\leq c\frac{r^{n+2-2s}}{(r-\rho)^{n+2}}\dashiint_{Q_{r}(z_{0})}|u-k|^{2}\dz+c\left(\frac{r}{r-\rho}\right)^{2(n+2s)}\dashint_{\Lambda_{r}(t_{0})}\dashint_{B_{r}(x_{0})}\int_{B_{r}(x_{0})}\frac{|f(x,t)-f(y,t)|^{2}}{|x-y|^{n}}\dx\dy\dt\\
    &\quad+c\left(\dashint_{Q_{r}(z_{0})}\left(r^{s}|g|\right)^{\gamma}\dz\right)^{\frac{2}{\gamma}}+c\left(\frac{r}{r-\rho}\right)^{2(n+2s)}\Tail_{\gamma,2s}\left(\frac{u-(u)_{B_{r}(x_{0})}(t)}{r^{s}};Q_{r}(z_{0})\right)^{2}\\
    &\quad+c\left(\frac{r}{r-\rho}\right)^{2(n+2s)}\Tail_{2,s}\left(f-(f)_{B_{r}(x_{0})}(t);Q_{r}(z_{0})\right)^{2},
\end{aligned}
\end{equation}
where $k\in\mathbb{R}$ and
\begin{equation}
\label{cz non pa : gamma}
\gamma=\frac{2(n+2s)}{n+4s}.
\end{equation}
\end{lem}

\begin{proof}
Since $u-k$ is also a weak solution to \eqref{cz non pa : eq1}, we may assume that $k=0$. Let us take a cutoff function $\psi\in C_{c}^{\infty}\left(B_{\frac{\rho+r}{2}}\right)$ such that
\begin{equation*}
    \psi\equiv 1 \text{ on }B_{\rho}(x_{0})\quad\text{and}\quad |D\psi|\leq \frac{16}{r-\rho},
\end{equation*}
and a smooth function $\eta\in C^{\infty}(\mathbb{R})$ such that
\begin{equation*}
    \eta\equiv 1 \text{ on } \left(t_{0}-\left(\frac{r^{2s}+\rho^{2s}}{2}\right),\infty\right),\quad \eta\equiv 0 \text{ on } \left(-\infty,t_{0}-\left(\frac{3r^{2s}+\rho^{2s}}{4}\right)\right]\quad\text{and}\quad |\eta'|\leq \frac{16}{r^{2s}-\rho^{2s}}.
\end{equation*}    
With the aid of \cite[Lemma 3.2]{BKKh}, we deduce that
\begin{equation}
\label{energyb}
\begin{aligned}
    I:&=\dashint_{\Lambda_{r}(t_{0})}\dashint_{B_{r}(x_{0})}\int_{B_{r}(x_{0})}\frac{|(u\psi\eta)(x,t)-(u\psi\eta)(y,t)|^{2}}{|x-y|^{n+2s}}\dx\dy\dt+\sup_{t\in \Lambda_{r}}\dashint_{B_{r}(x_{0})}\frac{|(u\psi\eta)(x,t)|^{2}}{r^{2s}}\dx\\
    &\leq c\frac{r^{2(1-s)}}{(r-\rho)^{2}}\dashiint_{Q_{r}(z_{0})}|u|^{2}\dz+c\frac{1}{r^{2s}-\rho^{2s}}\dashiint_{Q_{r}(z_{0})}|u|^{2}\dz+c\dashiint_{Q_{r}(z_{0})}|g||u\psi\eta|\dz\\
    &\quad+c\left(\frac{r}{r-\rho}\right)^{n+2s}\dashint_{\Lambda_{r}(t_{0})}\int_{\mathbb{R}^{n}\setminus B_{r}(x_{0})}\frac{|u(y,t)|}{|y-x_{0}|^{n+2s}}\dy\dashint_{B_{r}(x_{0})}|(u\psi\eta)(x,t)|\dx\dt\\
    &\quad+ c\dashint_{\Lambda_{r}(t_{0})}\dashint_{B_{r}(x_{0})}\int_{B_{r}(x_{0})}\frac{(f(x,t)-f(y,t))\left((u\psi\eta)(x,t)-(u\psi\eta)(y,t)\right)}{|x-y|^{n+s}}\dx\dy\dt\\
    &\quad +c\left(\frac{r}{r-\rho}\right)^{n+s}\dashint_{\Lambda_{r}(t_{0})}\int_{\mathbb{R}^{n}\setminus B_{r}(x_{0})}\frac{|f(x,t)-f(y,t)|}{|y-x_{0}|^{n+s}}\dy\dashint_{B_{r}(x_{0})}|(u\psi\eta)(x,t)|\dx\dt\eqqcolon\sum_{i=1}^{6}I_{i}.
\end{aligned}
\end{equation}
After a few simple algebraic computations along with the fact that $\rho<r\leq2\rho$, we observe that $I_{2}\leq c I_{1}$. Using H\"older's inequality, Lemma \ref{cz non pa : embed lem} and Young's inequality, we have
{\small\begin{equation*}
\begin{aligned}
    I_{3}&\leq \left(\dashiint_{Q_{r}(z_{0})}(r^{s}|g|)^{\gamma}\right)^{\frac{1}{\gamma}}\left(\dashiint_{Q_{r}(z_{0})}\left|\frac{u\psi\eta}{r^{s}}\right|^{2_{\#}}\right)^{\frac{1}{2_{\#}}}\leq c\left(\dashiint_{Q_{r}(z_{0})}(r^{s}|g|)^{\gamma}\right)^{\frac{2}{\gamma}}+\frac{I}{8}+ r^{-2s}\dashiint_{Q_{r}(z_{0})}|u|^{2}\dx\dt
\end{aligned}
\end{equation*}}\\
\noindent
and
{\small\begin{equation*}
\begin{aligned}
    I_{4}&\leq c\left(\frac{r}{r-\rho}\right)^{(n+2s)}\Tail_{\gamma,2s}\left(\frac{u-(u)_{B_{r}(x_{0})}(t)}{r^{s}};Q_{r}(z_{0})\right)\left(\dashiint_{Q_{r}(z_{0})}\left|\frac{u\psi\eta}{r^{s}}\right|^{2_{\#}}\right)^{\frac{1}{2_{\#}}}+c\left(\frac{r}{r-\rho}\right)^{n}I_{1}\\
    &\leq c\left(\frac{r}{r-\rho}\right)^{2(n+2s)}\Tail_{\gamma,2s}\left(
    \frac{u-(u)_{B_{r}(x_{0})}(t)}{r^{s}};Q_{r}(z_{0})\right)^{2}+\frac{I}{8}+ r^{-2s}\dashiint_{Q_{r}(z_{0})}|u|^{2}\dx\dt+c\left(\frac{r}{r-\rho}\right)^{n}I_{1},
\end{aligned}
\end{equation*}}\\
\noindent
where the constant $2_{\#}$ is defined in \eqref{cz non pa : paraconj}.
On the other hand, we obtain 
\begin{equation*}
    I_{5}\leq c\dashint_{\Lambda_{r}(t_{0})}\dashint_{B_{r}(x_{0})}\int_{B_{r}(x_{0})}\frac{|f(x,t)-f(y,t)|^{2}}{|x-y|^{n}}\dx\dy\dt+\frac{I}{8}
\end{equation*}
by H\"older's inequality and Young's inequality. For the last two terms, we first observe that
\begin{equation*}
\begin{aligned}
    I_{5}+I_{6}&\leq c\left(\frac{r}{r-\rho}\right)^{n+s}\dashint_{\Lambda_{r}(t_{0})}\int_{\mathbb{R}^{n}\setminus B_{r}(x_{0})}\dashint_{B_{r}(x_{0})}\frac{|f(x,t)-(f)_{B_{r}(x_{0})}(t)|}{|y-x_{0}|^{n+s}}|(u\psi\eta)(x,t)|\dx\dy\dt\\
    &\quad +c\left(\frac{r}{r-\rho}\right)^{n+s}\dashint_{\Lambda_{r}(t_{0})}\int_{\mathbb{R}^{n}\setminus B_{r}(x_{0})}\frac{|f(y,t)-(f)_{B_{r}(x_{0})}(t)|}{|y-x_{0}|^{n+s}}\dy\dashint_{B_{r}(x_{0})}|(u\psi\eta)(x,t)|\dx\dt.
\end{aligned}
\end{equation*}
By following the same line as in the estimate of $I_{2,2}$ in \cite[Lemma 3.3]{BKc}, we estimate $I_{5}$ as
\begin{equation*}
\begin{aligned}
    I_{5}+I_{6}&\leq \frac{I}{8}+c\left(\frac{r}{r-\rho}\right)^{2(n+2s)}\dashint_{\Lambda_{r}(t_{0})}\dashint_{B_{r}(x_{0})}\int_{B_{r}(x_{0})}\frac{|f(x,t)-f(y,t)|^{2}}{|x-y|^{n}}\dx\dy\dt\\
    &\quad +c\Tail_{2,s}\left(f-(f)_{B_{r}(x_{0})}(t);Q_{r}(z_{0})\right)^{2}.
\end{aligned}
\end{equation*}
Using the definitions of $\psi$ and $\eta$, we estimate $I$ as
\begin{equation*}
\begin{aligned}
    I\geq \left[\dashint_{\Lambda_{\rho}(t_{0})}\dashint_{B_{\rho}(x_{0})}\int_{B_{\rho}(x_{0})}\frac{|u(x,t)-u(y,t)|^{2}}{|x-y|^{n+2s}}\dx\dy\dt+
    \sup_{t\in \Lambda_{\rho}(t_{0})}\dashint_{B_{\rho}(x_{0})}\frac{|u(x,t)|^{2}}{\rho^{2s}}\dx\right].
\end{aligned}
\end{equation*}
We combine the estimates $I$ and $I_{i}$ for each $i=1,2,\ldots,6$ along with \eqref{energyb} to obtain \eqref{cz non pa : CE}.
\end{proof}


Next we give a gluing lemma.
\begin{lem}(See \cite[Lemma 4.5]{BKKh}.)
\label{cz non pa : GL}
Let $Q_{\rho}(z_{0})\Subset\Omega_{T}$. Take $\psi\in C_{c}^{\infty}\left(B_{\frac{3\rho}{4}}(x_{0})\right)$ with $\psi\equiv 1$ in $B_{\frac{\rho}{2}}(x_{0})$. Then we have that
\begin{align*}
    &\sup_{t_{1},t_{2}\in \Lambda_{\rho}(t_{0})} \left|(u)_{B_{\rho}}^{\psi}(t_{1})-(u)_{B_{\rho}}^{\psi}(t_{2})\right|   \\
    &\leq c\rho^{2s-1}\dashint_{Q_{\rho}(z_{0})}\int_{B_{\rho}(x_{0})}\frac{|u(x,t)-u(y,t)|}{|x-y|^{n+2s-1}}\dy\dz+c\rho^{2s}\dashint_{Q_{\rho}(z_{0})}\int_{\mathbb{R}^{n}\setminus B_{\rho}(x_{0})}\frac{|u(x,t)-u(y,t)|}{|y-x_{0}|^{n+2s}}\dy\dz\\
    &\quad+c\rho^{2s-1}\dashint_{Q_{\rho}(z_{0})}\int_{B_{\rho}(x_{0})}\frac{|f(x,t)-f(y,t)|}{|x-y|^{n+s-1}}\dy\dz+c\rho^{2s}\dashint_{Q_{\rho}(z_{0})}\int_{\mathbb{R}^{n}\setminus B_{\rho}(x_{0})}\frac{|f(x,t)-f(y,t)|}{|y-x_{0}|^{n+s}}\dy\dz\\
    &\quad +c\rho^{2s}\dashint_{Q_{\rho}(z_{0})}|g|\dz
\end{align*}
for some constant $c=c(n,s)$, where $(u)_{B_{\rho}}^{\psi}(t)=\frac{1}{\|\psi\|_{L^1}}\int_{B_{\rho}}u(x,t)\psi(x)\dx$.
\end{lem}

\noindent
We now show that $L^{2}$-oscillation integral by the sup-norm is estimated by the sum of $L^{2}$-integral of $d_{s}u$ and $d_{0}f$, $L^{\gamma}$-integral of $g$, and tail terms of $u$ and $f$. Before giving the estimate, we define a function $G:\Omega\times\Omega_{T}\to\mathbb{R}$ by
\begin{equation}
\label{cz non pa :defnG}
G(x,y,t)=g(x,t).
\end{equation}
We directly deduce that
\begin{equation}
\label{cz non pa : equbwGg}
\int_{Q_{r}}|g|^{p}\dz \approx_{n,s,p} \frac{1}{r^{2\tau}}  \int_{\mathcal{Q}_{r}}|G|^{p}\dmutt.
\end{equation}

\begin{lem}
\label{cz non pa : RHI}
Let $u$ be a weak solution to \eqref{cz non pa : eq1} and let $Q_{2\rho}(z_{0})\Subset\Omega_{T}$. Then we have
{\small\begin{equation}
\label{cz non pa : reversehig}
\begin{aligned}
\sup_{t\in \Lambda_{\rho}(t_{0})}\dashint_{B_{\rho}(x_{0})}\frac{|u-(u)_{Q_{\rho}(z_{0})}|^{2}}{\rho^{2s+2\tau}}\dx&\leq   \frac{c_{0}}{\tau}\dashint_{\mathcal{Q}_{2\rho}(z_{0})}|D^{\tau}d_{s}u|^{2}\dmutt+c_{0}\Tail_{2,2s}\left(\frac{u-(u)_{B_{2\rho}(x_{0})}(t)}{(2\rho)^{s+\tau}};Q_{2\rho}(z_{0})\right)^{2}\\
&\quad+\frac{c_{0}}{\tau}\dashint_{\mathcal{Q}_{2\rho}(z_{0})}|D^{\tau}d_{0}f|^{2}\dmutt+c_{0}\Tail_{2,s}\left(\frac{f-(f)_{B_{2\rho}(x_{0})}(t)}{(2\rho)^{\tau}};Q_{2\rho}(z_{0})\right)^{2}\\
&\quad +c_{0}\left(\frac{1}{\tau}\dashint_{\mathcal{Q}_{2\rho}(z_{0})}\left((2\rho)^{s-\tau}|G|\right)^{\gamma}\dmutt\right)^{\frac{2}{\gamma}},
\end{aligned}
\end{equation}}\\
\noindent
where $c_{0}=c_{0}(n,s,L)$ is a constant.
\end{lem}
\begin{proof}
We may assume that $z_{0}=0$.
Using \eqref{cz non pa : CE} with $r=2\rho$ and $k=(u)_{Q_{\rho}}$, we have
{\small\begin{equation*}
\begin{aligned}
    \sup_{t\in\Lambda_{\rho}}\dashint_{B_{\rho}}\frac{|u(x,t)-(u)_{Q_{\rho}}|^{2}}{\rho^{2s+2\tau}}\dx&\leq c\dashiint_{Q_{2\rho}}\frac{|u-(u)_{Q_{2\rho}}|^{2}}{(2\rho)^{2s+2\tau}}\dz+\frac{c}{\tau}\dashint_{\mathcal{Q}_{2\rho}}|D^{\tau}d_{0}f|^{2}\dmutt+c\left(\dashint_{Q_{2\rho}}\left((2\rho)^{s-\tau}|g|\right)^{\gamma}\dz\right)^{\frac{2}{\gamma}}\\
    &\quad+c\left[\Tail_{\gamma,2s}\left(\frac{u-(u)_{B_{2\rho}}(t)}{(2\rho)^{s+\tau}};Q_{2\rho}\right)^{2}+\Tail_{2,s}\left(\frac{f-(f)_{B_{2\rho}}(t)}{(2\rho)^{\tau}};Q_{2\rho}\right)^{2}\right].
\end{aligned}
\end{equation*}}\\
\noindent
Applying \eqref{cz non pa : paraSPI2} and \eqref{cz non pa : size of b mut} to the first term and the third term in the right-hand side of the above inequality, respectively, we obtain the desired estimate \eqref{cz non pa : reversehig}. 
\end{proof}


\section{Comparison estimates}
This section is devoted to establishing comparison estimates. We now assume 
\begin{equation}
\label{cz non pa :taufirstcond}
\tau\in\left(0,\min\{s,1-s\}\right).
\end{equation}
Before proving comparison estimates, we first give two lemmas. The first one is a self-improving property for weak solutions to the corresponding homogeneous problem of \eqref{cz non pa : eq1}.
\begin{lem}
\label{cz non pa : self impro} Let $w\in L^{2}\left(\Lambda_{3};W^{s,2}(B_{3})\right)\cap L^{\infty}\left(\Lambda_{3};L^{1}_{2s}\left(\mathbb{R}^{n}\right)\right)$ be a  weak solution to
\begin{equation*}
    w_t+\mathcal{L}^{\Phi}_{A}w=0\quad\text{in }Q_{3}.
\end{equation*}
Then there are constants $\epsilon_{0}=\epsilon_{0}(n,s,L)\in(0,1)$ and $c=c(n,s,L)$ such that 
{\small\begin{equation*}
\begin{aligned}
    \left(\frac{1}{\tau}\dashint_{\mathcal{Q}_{2}}|D^{\tau}d_{s}w|^{2(1+\epsilon_{0})}\dmutt\right)^{\frac{1}{2(1+\epsilon_{0})}}&\leq c\left(\frac{1}{\tau}\dashint_{\mathcal{Q}_{3}}|D^{\tau}d_{s}w|^{2}\dmutt\right)^{\frac{1}{2}}+\Tail_{\infty,2s}\left(\frac{w-(w)_{B_{3}}(t)}{3^{s+\tau}};Q_{3}\right)\\
    &\quad+c\left(\sup_{t\in \Lambda_{3}}\dashint_{B_{3}}\frac{|w-(w)_{B_{3}}(t)|^{2}}{3^{2(s+\tau)}}\dx\right)^{\frac{1}{2}}.
\end{aligned}
\end{equation*}}
\end{lem}
\begin{proof}
By Theorem \ref{cz non pa : selfprof} below, we have
{\small\begin{equation*}
\begin{aligned}
    \left(\dashint_{Q_{r}}\int_{B_{r}}\left|r^{\epsilon_{1}}{d_{s+\epsilon_{1}}w}\right|^{2+\epsilon_{1}}\frac{\dx\dz}{|x-y|^{n}}\right)^{\frac{1}{2+\epsilon_{1}}}&\leq c\left(\dashint_{Q_{2r}}\int_{B_{2r}}|d_{s}w|^{2}\frac{\dx\dz}{|x-y|^{n}}\right)+c\Tail_{\infty,2s}\left(\frac{w-(w)_{B_{2r}}(t)}{(2r)^{s}};Q_{2r}\right)\\
    &\quad +c\left(\sup_{t\in \Lambda_{2r}}\dashint_{B_{2r}}\frac{|w-(w)_{B_{2r}}(t)|^{2}}{(2r)^{2s}}\dx\right)^{\frac{1}{2}}
\end{aligned}
\end{equation*}}\\
\noindent
for some constant $\epsilon_{1}=\epsilon_{1}(n,s,L)\in(0,1)$, where $Q_{2r}\subset Q_{3}$. By taking $\epsilon_{0}=\frac{\epsilon_{1}}{2}$ and using \eqref{cz non pa : size of b mut}, we get
{\small\begin{equation*}
\begin{aligned}
    \left(\frac{1}{\tau}\dashint_{\mathcal{Q}_{r}}|D^{\tau}d_{s}w|^{2(1+\epsilon_{0})}\dmutt\right)^{\frac{1}{2(1+\epsilon_{0})}}&\leq c\left(\frac{1}{\tau}\dashint_{\mathcal{Q}_{2r}}|D^{\tau}d_{s}w|^{2}\dmutt\right)^{\frac{1}{2}}+\Tail_{\infty,2s}\left(\frac{w-(w)_{B_{2r}}(t)}{(2r)^{s+\tau}};Q_{2r}\right)\\
    &\quad+c\left(\sup_{t\in \Lambda_{2r}}\dashint_{B_{2r}}\frac{|w-(w)_{B_{2r}}(t)|^{2}}{(2r)^{2(s+\tau)}}\dx\right)^{\frac{1}{2}}.
\end{aligned}
\end{equation*}}\\
\noindent
The standard covering argument along with Lemma \ref{cz non pa : off rever lem} gives the desired result (see \cite[Lemma 3.1]{BKc} for more details).
\end{proof}

The second one is a higher H\"older regularity of weak solutions to fractional parabolic equations with locally constant coefficients with respect to the spatial variables.
\begin{lem}(See \cite[Theorem 1.2]{BKKh})
\label{cz non pa : higher hol} Let $v\in L^{2}\left(\Lambda_{2};W^{s,2}(B_{2})\right)\cap L^{\infty}\left(\Lambda_{2};L^{1}_{2s}\left(\mathbb{R}^{n}\right)\right)$ be a weak solution to
\begin{equation*}
    v_t+\mathcal{L}^{\Phi}_{A_{2}(t)}v=0\quad\text{in }Q_{2},
\end{equation*}
where we denote $A_{2}(t)=A_{2,0}(t)$ which is defined in \eqref{cz non pa : defn ava ker}.
Then for any $\alpha\in (0,\min\left\{2s,1\right\})$, there is a constant $c=c(n,s,L,\alpha)$ such that
\begin{equation*}
    [v]_{C^{\alpha,\frac{\alpha}{2s}}(Q_{1})}\leq c\left(\frac{1}{\tau}\dashint_{\mathcal{Q}_{2}}|D^{\tau}d_{s}v|^{2}\dmutt\right)^{\frac{1}{2}}+c\Tail_{\infty,2s}\left(v-(v)_{B_{2}}(t);Q_{2}\right).
\end{equation*}
\end{lem}

We are now in position to prove the following comparison lemma.
\begin{lem}
\label{cz non pa : comp}
Let $Q_{4}\Subset\Omega_{T}$. For any $\epsilon>0$, there is a constant $\delta=\delta(n,s,L,\epsilon)$ such that for any weak solution $u$ to \eqref{cz non pa : eq1} satisfying 
\begin{equation}
\label{cz non pa : compu}
    \frac{1}{\tau}\dashint_{\mathcal{Q}_{4}}|D^{\tau}d_{s}u|^{2}\dmutt+\Tail_{\infty,2s}\left(\frac{u-(u)_{B_{4}}(t)}{4^{s+\tau}}\,;\,Q_{4}\right)^{2}\leq 1
\end{equation}
and
{\small\begin{equation}\label{cz non pa : comp.9}
\begin{aligned}
  &\left(\frac{1}{\tau}\dashint_{\mathcal{Q}_{4}}\left(4^{s-\tau}|G|\right)^{\gamma}\dmutt\right)^{\frac{2}{\gamma}}+\frac{1}{\tau}\dashint_{\mathcal{Q}_{4}}|D^{\tau}d_{0}f|^{2}\dmutt +\Tail_{2,s}\left(\frac{f-(f)_{B_{4}}(t)}{4^{\tau}}\,;\,Q_{4}\right)^{2} \\
  &\quad+\left(\dashint_{\mathcal{Q}_{2}}|A-(A)_{2}(t)|\,dx\,dy\,dt\right)^{2}\leq \delta^{2},
  \end{aligned}
\end{equation}}\\
\noindent
there is a solution $v$ to 
\begin{equation*}
   v_t+\mathcal{L}^{\Phi}_{A_{{2}}(t)}v=0\quad\text{in }Q_{2}
\end{equation*}
such that
\begin{equation}
\label{cz non pa : compres}
    \frac{1}{\tau}\dashint_{\mathcal{Q}_{1}}|D^{\tau}d_{s}(u-v)|^{2}\dmutt\leq \epsilon^{2}\quad\text{and}\quad \left\|D^{\tau}d_{s}v\right\|_{L^{\infty}\left(\mathcal{Q}_{1}\right)}\leq c,
\end{equation}
where $c=c(n,s,L,\tau)$.
\end{lem}

\begin{proof}
The proof is divided into several steps for the ease of readability.\\
\textbf{Step 1}: \textit{First comparison estimates}.  For a fixed weak solution $u$ to problem \eqref{cz non pa : eq1}, we consider the following problem:
\begin{equation}\label{eq:1st.comp}
\left\{\begin{array}{rlll}
     w_t + \mathcal{L}^{\Phi}_A w &=0 \quad\mbox{in }Q_3, \\
       w&=u \quad\mbox{in }(\mathbb{R}^n\setminus B_3)\times \Lambda_3 \cup B_3\times\{-3^{2s}\}.
 \end{array}
 \right.
 \end{equation}
We intend to apply Lemma \ref{cz non pa : exst-unq} for the existence and uniqueness of $w$. To this end, it remains to show that $u_t \in L^2(\Lambda_3; W^{s,2}(B_4))^*$. Indeed, using the fact that $u$ is a solution to problem \eqref{cz non pa : eq1}, we find that for all $\phi\in L^2(\Lambda_3; W^{s,2}(B_4))\cap C^1_0(\Lambda_3;L^2(B_4))$, there holds
\begin{align*}
  \bigg|\int_{\Lambda_3}\langle u,\phi_t\rangle dt\bigg|\leq c \int_{\Lambda_3}\|\phi(\cdot,t)\|_{W^{s,2}(B_4)}dt,  
\end{align*}
for some $c$ depending only on $n,s,L,u,f$ and $g$.
Thus, we have the existence of  $w\in L^{2}\left(\Lambda_3;W^{s,2}(B_4)\right)\cap 
C\left(\Lambda_3;L^{2}(B_3)\right)\cap L^{\infty}\left(\Lambda_3;L^{1}_{2s}(\mathbb{R}^{n})\right)$ satisfying \eqref{eq:1st.comp}. Then $\varphi:=u-w$ solves 
\begin{equation*}
    \varphi_{t}+\mathcal{L}_{A}^{\Phi}u-\mathcal{L}_{A}^{\Phi}w=(-\Delta)^{\frac{s}{2}}f+g\quad\text{in }Q_{3}.
\end{equation*}
With the help of an approximation argument, we take $\varphi$ as a test function to the above equation to see that for every $\tilde{T}\in (-3^{2s},3^{2s}]$, setting $\tilde\Lambda_3:=(-3^{2s},\tilde{T}]$, there holds
\begin{align*}
    \frac{1}{2}&\int_{B_3}(\varphi(x,\tilde{T}))^2\dx+\frac{1}{L}\int_{\tilde\Lambda_3}\int_{\mathbb{R}^{n}}\int_{\mathbb{R}^{n}}|\varphi(x,t)-\varphi(y,t)|^2\frac{A(x,y,t)}{|x-y|^{n+2s}}\dx\dy\dt\\
    &\leq\int_{\tilde\Lambda_3}\int_{\mathbb{R}^{n}}\int_{\mathbb{R}^{n}}|(f(x,t)-f(y,t))(\varphi(x,t)-\varphi(y,t))|\frac{\dx\dy\dt}{|x-y|^{n+s}}+\int_{\tilde\Lambda_3}\int_{B_3}|g\varphi|\dz,
\end{align*}
where we have used the fact that $\varphi(\cdot,-3^{2s})=0$ in $B_3$ and the first condition in \eqref{cz non pa : Phicond}.
Noting the bounds on $A$, the above expression yields
\begin{align}\label{cz non pa : comp.2}
\frac{1}{2}&\int_{B_3}(\varphi(x,\tilde{T}))^2\dx+\int_{\tilde\Lambda_3}\int_{\mathbb{R}^{n}}\int_{\mathbb{R}^{n}}\frac{|\varphi(x,t)-\varphi(y,t)|^2}{|x-y|^{n+2s}}\dx\dy\dt  \nonumber\\
&\leq L^{2}\underbrace{\int_{\tilde\Lambda_3}\int_{B_4}\int_{B_4}|f(x,t)-f(y,t)|\, |\varphi(x,t)-\varphi(y,t)|\frac{\dx\dy\dt}{|x-y|^{n+s}}}_{=:J_1} \nonumber\\
&\quad+ 2L^{2}\underbrace{\int_{\tilde\Lambda_3}\int_{B_3}\int_{\mathbb{R}^{n}\setminus B_4}|f(x,t)-f(y,t)|\,|\varphi(x,t)|\frac{\dx\dy\dt}{|x-y|^{n+s}}}_{=:J_2}+ L^{2}\underbrace{\int_{\tilde\Lambda_3}\int_{B_3}|g\varphi|\dz}_{=:J_3}.
\end{align}
 For $J_1$, applying H\"older's inequality and Young's inequality, we observe that
\begin{align*}
 J_1&\leq \left(\int_{\Lambda_3}\int_{B_4}\int_{B_4}|f(x,t)-f(y,t)|^2 \frac{\dx\dy\dt}{|x-y|^{n}}\right)^\frac{1}{2} \left(\int_{\tilde\Lambda_3}\int_{B_4}\int_{B_4} |\varphi(x,t)-\varphi(y,t)|^2\frac{\dx\dy\dt}{|x-y|^{n+2s}}\right)^\frac{1}{2} \nonumber\\
 &\leq \frac{c}{\tau}\dashint_{\mathcal{Q}_4}|D^\tau d_0f|^2 d\mu_{\tau,t}  + \frac{1}{4L^{2}} \int_{\tilde\Lambda_3}\int_{B_4}\int_{B_4}|\varphi(x,t)-\varphi(y,t)|^2 \frac{\dx\dy\dt}{|x-y|^{n+2s}},
\end{align*}
where $c=c(n,s,L)$. On a similar account, we deduce that
\begin{align*}
J_2 &\leq \int_{\tilde\Lambda_3}\int_{B_3}\int_{\mathbb{R}^{n}\setminus B_4}\big(|f(x,t)-(f)_{B_4}(t)|+|f(y,t)-(f)_{B_4}(t)|\big)|\varphi(x,t)|\frac{\dx\dy\dt}{|x-y|^{n+s}} \nonumber\\
&\leq c \left(\int_{\Lambda_4}\int_{B_4} |f(x,t)-(f)_{B_4}(t)|^2 dx\,dt\right)^\frac{1}{2} \left(\int_{\tilde\Lambda_3}\int_{B_3} \varphi(x,t)^2 dx\,dt\right)^\frac{1}{2} \nonumber\\
&\quad+ c \left(\int_{\tilde\Lambda_3}\int_{B_3}|\varphi|^2 \dz\right)^\frac{1}{2} \left(\int_{\Lambda_3}\left(\int_{\mathbb{R}^n\setminus B_4}\frac{|f(y,t)-(f)_{B_4}(t)|}{|y|^{n+s}}dy\right)^2dt\right)^\frac{1}{2}.
\end{align*} 
Noticing the fact that $\varphi(\cdot,t)=0$ in $\mathbb{R}^n\setminus B_3$ for a.e. $t\in\Lambda_3$, 
 we employ the Sobolev-Poincar\'e inequality to get 
{\small\begin{align*}
 J_2\leq \frac{1}{8L^{2}} \int_{\tilde\Lambda_3}\int_{B_4}\int_{B_4}|\varphi(x,t)-\varphi(y,t)|^2 \frac{\dx\dy\dt}{|x-y|^{n+2s}} + \frac{c}{\tau} \dashint_{\mathcal{Q}_4}|D^\tau d_0f|^2 d\mu_{\tau,t}+ c\, {\rm Tail}_{2,s}\left(\frac{f-(f)_{B_4}(t)}{4^{\tau}}; Q_4\right)^2,
\end{align*}}\\
\noindent
where $c=c(n,s,L)$. Now for $J_3$, employing H\"older's inequality, Lemma \ref{cz non pa : embed lem} and Young's inequality, we get
\begin{align*}
    J_3\leq c \left(\int_{Q_3}|g(x,t)|^\gamma dx\,dt\right)^\frac{2}{\gamma}+ \frac{1}{4L^{2}}\int_{\tilde\Lambda_3}\int_{B_3}\int_{B_3}\frac{|\varphi(x,t)-\varphi(y,t)|^2}{|x-y|^{n+2s}}\dx\dy\dt +\frac{1}{4L^{2}} \sup_{t\in \Lambda_3} \int_{B_3} |\varphi(x,t)|^2 dx.  
\end{align*}
Therefore, using the estimates of $J_1$, $J_2$ and $J_3$ in \eqref{cz non pa : comp.2}, taking supremum over $\tilde{T}\in (-3^{2s},3^{2s}]$ and recalling the definition of $G$ from \eqref{cz non pa :defnG} along with \eqref{cz non pa : equbwGg}, we obtain 
\begin{align}\label{cz non pa : comp.3}
\sup_{t\in\Lambda_3}&\dashint_{B_3}(u-w)^2(x,t)\dx+\frac{1}{\tau}\dashint_{\mathcal{Q}_3}|D^\tau d_s(u-w)|^2\dmutt \nonumber\\
&\leq \left(\frac{1}{\tau}\dashint_{\mathcal{Q}_{4}}\left(4^{s-\tau}|G|\right)^{\gamma}\dmutt\right)^{\frac{2}{\gamma}}+\frac{c}{\tau} \dashint_{\mathcal{Q}_4}|D^\tau d_0f|^2 d\mu_{\tau,t}+ c\, {\rm Tail}_{2,s}\Big(\frac{f-(f)_{B_4}(t)}{4^\tau}; Q_4\Big)^2 \leq c \,\delta^2,
\end{align}
where we have also used \eqref{cz non pa : size of b mut} and \eqref{cz non pa : comp.9}. \\
\textbf{Step 2}:  \textit{Uniform self-improving inequality for $w$}.
We first observe from Lemma \ref{cz non pa : self impro} that 
{\small\begin{equation*}
\begin{aligned}
    \left(\frac{1}{\tau}\dashint_{\mathcal{Q}_{2}}|D^{\tau}d_{s}w|^{2(1+\epsilon_{0})}\dmutt\right)^{\frac{1}{2(1+\epsilon_{0})}}&\leq c\left(\frac{1}{\tau}\dashint_{\mathcal{Q}_{3}}|D^{\tau}d_{s}w|^{2}\dmutt\right)^{\frac{1}{2}}+c\Tail_{\infty,2s}\left(\frac{w-(w)_{B_{3}}(t)}{3^{s+\tau}};Q_{3}\right)\\
    &\quad+c\left(\sup_{t\in \Lambda_{3}}\dashint_{B_{3}}\frac{|w-(w)_{B_{3}}(t)|^{2}}{3^{2(s+\tau)}}\dx\right)^{\frac{1}{2}},
\end{aligned}
\end{equation*}}\\
\noindent
where $\epsilon_{0}=\epsilon_{0}(n,s,L)\in(0,1)$ and $c=c(n,s,L)$. Furthermore, by using \eqref{cz non pa : comp.3}, we have 
\begin{align*}
 \left(\frac{1}{\tau}\dashint_{\mathcal{Q}_{3}}|D^{\tau}d_{s}w|^{2}\dmutt\right)^{\frac{1}{2}}\leq \left(\frac{1}{\tau}\dashint_{\mathcal{Q}_{3}}|D^{\tau}d_{s}(u-w)|^{2}\dmutt\right)^{\frac{1}{2}}+ \left(\frac{1}{\tau}\dashint_{\mathcal{Q}_{3}}|D^{\tau}d_{s}u|^{2}\dmutt\right)^{\frac{1}{2}} \leq c.
\end{align*}
Using \eqref{cz non pa : reversehig} with a slight modification, \eqref{cz non pa : comp.3} and \eqref{cz non pa : compu}, we next have
\begin{equation*}
\begin{aligned}
    \left(\sup_{t\in \Lambda_{3}}\dashint_{B_{3}}\frac{|w-(w)_{B_{3}}(t)|^{2}}{3^{2(s+\tau)}}\dx\right)^{\frac{1}{2}}\leq c\left(\sup_{t\in \Lambda_{3}}\dashint_{B_{3}}\frac{|u-(u)_{B_{3}}(t)|^{2}}{3^{2(s+\tau)}}\dx\right)^{\frac{1}{2}}+c\left(\sup_{t\in \Lambda_{3}}\dashint_{B_{3}}\frac{|u-w|^{2}}{3^{2(s+\tau)}}\dx\right)^{\frac{1}{2}}\leq c.
\end{aligned}
\end{equation*}
For the tail term, a simple computation together with \eqref{cz non pa : comp.3} and \eqref{cz non pa : compu} yields
\begin{align*}
 \Tail_{\infty,2s}\left(\frac{w-(w)_{B_{3}}(t)}{3^{s+\tau}};Q_{3}\right) \leq \sup_{t\in\Lambda_3}\dashint_{B_3}\frac{|w-u|^2}{3^{2(\tau+s)}} \dx+ \Tail_{\infty,2s}\left(\frac{u-(u)_{B_{3}}(t)}{3^{s+\tau}};Q_{3}\right)\leq c.
\end{align*}
Consequently, 
\begin{align}\label{cz non pa : comp.4}
    \left(\frac{1}{\tau}\dashint_{\mathcal{Q}_{2}}|D^{\tau}d_{s}w|^{2(1+\epsilon_{0})}\dmutt\right)^{\frac{1}{2(1+\epsilon_{0})}}\leq c.
\end{align}

\noindent
\textbf{Step 3}: \textit{Second comparison estimates}. For $w$ as in Step 1, we consider the following problem:
\begin{equation}\label{eq:2nd.comp}
\left\{\begin{array}{rlll}
     v_t + \mathcal{L}^{\Phi}_{A_{2}(t)} v &=0 \quad\mbox{in }Q_2, \\
       v&=w \quad\mbox{in }(\mathbb{R}^n\setminus B_2)\times \Lambda_2 \cup B_2\times\{-2^{2s}\}.
 \end{array}
 \right.
 \end{equation}
 Similar to Step 1, we have the existence of a unique solution $v\in L^{2}\left(\Lambda_2;W^{s,2}(B_3)\right)\cap 
C\left(\Lambda_2;L^{2}(B_2)\right)\cap L^{\infty}\left(\Lambda_2;L^{1}_{2s}(\mathbb{R}^{n})\right)$ to the problem \eqref{eq:2nd.comp}. Taking $\tilde{\varphi}:=v-w$ as a test function (upon approximation) to 
\begin{equation*}
    \tilde{\varphi}_{t}+\mathcal{L}^{\Phi}_{A_{2}(t)} v-\mathcal{L}^{\Phi}_{A_{2}(t)} w=\mathcal{L}^{\Phi}_{A}w-\mathcal{L}^{\Phi}_{A_{2}(t)} w\quad\text{in }Q_{2}
\end{equation*} and then using \eqref{cz non pa : Phicond} and H\"older's inequality, we obtain
\begin{align*}
&\frac{1}{\tau}\int_{\mathcal{Q}_{2}}|D^{\tau}d_{s}\tilde{\varphi}|^{2}\dmutt +\sup_{t\in\Lambda_2}\int_{B_2}|\tilde{\varphi}(x,t)|^2 dx \nonumber\\
&\leq c \int_{\Lambda_2}\int_{B_2}\int_{B_2} |A(x,t)-(A)_2(t)|\frac{|w(x,t)-w(y,t)||\tilde{\varphi}(x,t)-\tilde{\varphi}(y,t)|}{|x-y|^{n+2s}}dx\,dy\,dt \nonumber\\
&\leq c \left(\frac{1}{\tau}\int_{\mathcal{Q}_{2}}|D^{\tau}d_{s}\tilde{\varphi}|^{2}\dmutt \right)^\frac{1}{2} \left(\frac{1}{\tau}\int_{\mathcal{Q}_{2}}|D^{\tau}d_{s}w|^{2(1+\epsilon_0)}\dmutt \right)^\frac{1}{2(1+\epsilon_0)} \nonumber \\
&\qquad\times \left(\int_{\Lambda_{2}}\int_{B_2}\int_{B_2} |A-(A)_2(t)|^{2\frac{1+\epsilon_0}{\epsilon_0}}dx\,dy\,dt \right)^\frac{\epsilon_0}{2(1+\epsilon_0)}.
\end{align*}
Finally, up on using the vanishing condition on $A$ and \eqref{cz non pa : comp.4}, the above expression yields
\begin{align}\label{cz non pa : comp.5}
   \frac{1}{\tau}\int_{\mathcal{Q}_{2}}|D^{\tau}d_{s}(v-w)|^{2}\dmutt &\leq c \,\delta^\frac{\epsilon_0}{2(1+\epsilon_0)},
\end{align}
for some $c=c(n,s,L)$. Coupling \eqref{cz non pa : comp.3} with \eqref{cz non pa : comp.5} and using triangle inequality, we get the first part of \eqref{cz non pa : compres} by taking $\delta$ sufficiently small depending on $n,s,L$ and $\epsilon$.\\
\textbf{Step 4}: \textit{Uniform bound on $|D^\tau d_s v|$}. From Lemma \ref{cz non pa : higher hol} along with \eqref{cz non pa :taufirstcond}, we observe that 
{\small\begin{align*}
 \|D^\tau d_s v\|_{L^\infty(\mathcal{Q}_1)} = \sup_{(x,y,t)\in\mathcal{Q}_1}\frac{|v(x,t)-v(y,t)|}{|x-y|^{s+\tau}} \leq c\left(\frac{1}{\tau}\dashint_{\mathcal{Q}_{2}}|D^{\tau}d_{s}v|^{2}\dmutt\right)^{\frac{1}{2}}+c\Tail_{\infty,2s}\left(v-(v)_{B_{2}}(t);Q_{2}\right),  
\end{align*}}\\
\noindent
where $c=c(n,s,L,\tau)$. Proceeding as in Step 2 and  \cite[Lemma 3.3]{BKc}, it can be shown that the right-hand side quantity of the above expression is bounded by a uniform constant $c=c(n,s,L,\tau)$. This completes the proof of the lemma.
\end{proof}
We finish this section by giving a non-scaled version of the above lemma and this directly follows from Lemma \ref{cz non pa : comp} along with a scaling argument (see Lemma \ref{cz non pa : scalinglem}).
\begin{lem}
\label{cz non pa : nsccomp}
Let $Q_{20\rho_{i}}(z_{i})\Subset\Omega_{T}$. For any $\epsilon>0$, there is a constant $\delta=\delta(n,s,L,\epsilon)$ such that for any weak solution $u$ to \eqref{cz non pa : eq1} satisfying 
\begin{equation*}
    \frac{1}{\tau}\dashint_{\mathcal{Q}_{20\rho_{i}}(z_i)}|D^{\tau}d_{s}u|^{2}\dmutt+\Tail_{\infty,2s}\left(\frac{u-(u)_{B_{20\rho_{i}(x_{i})}}(t)}{(20\rho_{i})^{s+\tau}}\,;\,Q_{20\rho_{i}}(z_i)\right)^{2}\leq 1
\end{equation*}
and
{\small\begin{equation*}
\begin{aligned}
  &\left(\frac{1}{\tau}\dashint_{\mathcal{Q}_{20\rho_{i}}(z_{i})}\left((20\rho_{i})^{s-\tau}|G|\right)^{\gamma}\dmutt\right)^{\frac{2}{\gamma}}+\frac{1}{\tau}\dashint_{\mathcal{Q}_{20\rho_{i}}(z_{i})}|D^{\tau}d_{0}f|^{2}\dmutt +\Tail_{2,s}\left(\frac{f-(f)_{B_{20\rho_{i}}(x_{i})}(t)}{(20\rho_{i})^{\tau}}\,;\,Q_{20\rho_{i}}(z_{i})\right)^{2} \\
  &\quad +\left(\dashint_{\mathcal{Q}_{10\rho_{i}}(z_{i})}|A-(A)_{10\rho_{i},x_{i}}(t)|\,dx\,dy\,dt\right)^{2}\leq \delta^{2},
  \end{aligned}
\end{equation*}}\\
\noindent
then there exists  a solution $v$ to 
\begin{equation*}
   v_t+\mathcal{L}^{\Phi}_{A_{{10\rho_{i},x_{i}}}(t)}v=0\quad\text{in }Q_{10\rho_{i}}(z_{i})
\end{equation*}
such that
\begin{equation}
\label{cz non pa : nsccompres}
    \frac{1}{\tau}\dashint_{\mathcal{Q}_{5\rho_{i}}(z_{i})}|D^{\tau}d_{s}(u-v)|^{2}\dmutt\leq \epsilon^{2}\quad\text{and}\quad \left\|D^{\tau}d_{s}v\right\|_{L^{\infty}\left(\mathcal{Q}_{5\rho_{i}}(z_{i})\right)}\leq c,
\end{equation}
where $c=c(n,s,L,\tau)$.
\end{lem}


\section{Coverings of upper level sets}
In this section, we construct parabolic cylinders covering the upper level set of $d_{s}u$, where
\[u\in L^{2}\left( \Lambda_{2}\,;\, W^{s,2}(B_{2})\right)\cap C\left( \Lambda_{2};L^{2}(B_{2})\right) \cap L^{\infty}\left( \Lambda_{2}\,;\, L^{1}_{2s}(\mathbb{R}^{n})\right)\]
is a weak solution to the localized problem: 
\begin{equation}
\label{cz non pa : localized pb} 
    u_{t}+\mathcal{L}_{A}^{\Phi}u=(-\Delta)^{\frac{s}{2}}f+g \quad\text{in }Q_{2}.
\end{equation}
In addition, we assume that $f\in L^{q}\left( \Lambda_{2}; L^{1}_{s}\left(\mathbb{R}^{n}\right)\right)$ and
\begin{equation*}
    \int_{\mathcal{Q}_{2}}|D^{\tau}d_{s}u|^{p}+|D^{\tau}d_{0}f|^{q}+|G|^{\gamma}\dmutt<\infty,
\end{equation*}
where $p\in[2,q]$,
\begin{equation}
\label{cz non pa : tautail}
\tau\in\left(0,s-\frac{2s}{q}\right)
\end{equation}and
the constant $\gamma$ is defined in \eqref{cz non pa : gamma}.
Let us denote
\begin{equation}
\label{cz non pa : tildeq}
\tilde{q}=\frac{1}{2}\left(q+\frac{2s}{s-\tau}\right)
\end{equation}
to see that
\begin{equation}
\label{cz non pa : tildeqprop}
     s>\tau+\frac{2s}{\tilde{q}}
\end{equation}
and 
\begin{equation}
\label{cz non pa : tildeqprop2}
\tilde{q}<q,
\end{equation}
which follow from the choice of $\tau$ given in \eqref{cz non pa : tautail}.
We point out that \eqref{cz non pa : tildeqprop} and \eqref{cz non pa : tildeqprop2} are needed to handle the tail induced by the right-hand side $f$ and to employ Fubini's theorem, respectively (see \eqref{cz non pa : oneretildeqnec} and \eqref{cz non pa : tildeqfirstcond}).
We now present the main proposition of this section. 
\begin{prop}
\label{cz non pa : coveringL}
Let $1\leq r_{1}<r_{2}\leq 2$, $\delta>0$ and $u$ be a weak solution to \eqref{cz non pa : localized pb}. Then, there are two families of countable disjoint cylinders $\left\{\mathcal{Q}_{\rho_{{i}}}(z_{i})\right\}_{i\geq0}$ and $\left\{\mathcal{Q}_{\widetilde{r}_{{j}}}\left(x_{1,j},x_{2,j},t_{0,j}\right)\right\}_{j\geq0}$ such that
\begin{equation}
\label{cz non pa : levelset}
U_{\lambda}:=\left\{(x,y,t)\in \mathcal{Q}_{r_{1}}\,:\, |D^{\tau}d_{s}u(x,y,t)|\geq\lambda \right\}\subset\left(\bigcup\limits_{i\geq0}\mathcal{Q}_{5^{\frac{2}{s}}\rho_{{i}}}\left(z_{i}\right)\right)\bigcup \left(\bigcup\limits_{j\geq0}\mathcal{Q}_{5^{\frac{1}{s}}\widetilde{r}_{{j}}}\left(x_{1,j},x_{2,j},t_{0,j}\right)\right)
\end{equation}
whenever $\lambda\geq\lambda_{0}$, where
{\small\begin{equation}
\label{cz non pa : lambda0}
\begin{aligned}
    \lambda_{0}&\coloneqq \frac{c\tau^{\frac{1}{q}-\frac{1}{\gamma}}}{(r_{2}-r_{1})^{\frac{5n}{s}}}\left(\left(\frac{1}{\tau}\dashint_{\mathcal{Q}_{2}}|D^{\tau}d_{s}u|^{p}\dmutt\right)^{\frac{1}{p}}+\Tail_{\infty,2s}\left(\frac{u-(u)_{B_{2}}(t)}{2^{s+\tau}};Q_{2}\right)+\left(\sup_{t\in \Lambda_{2}}\dashint_{B_{2}}\frac{|u-(u)_{Q_{2}}|^{2}}{2^{2s+2\tau}}\dx\right)^{\frac{1}{2}}\right)\\
    &\quad+\frac{c\tau^{\frac{1}{q}-\frac{1}{\gamma}}}{(r_{2}-r_{1})^{\frac{5n}{s}}}\frac{1}{\delta}\left(\left(\frac{1}{\tau}\dashint_{\mathcal{Q}_{2}}|D^{\tau}d_{0}f|^{q}\dmutt\right)^{\frac{1}{q}}+\Tail_{q,s}\left(\frac{f-(f)_{B_{2}}(t)}{2^{\tau}};Q_{2}\right)+\left(\frac{1}{\tau}\dashint_{\mathcal{Q}_{2}}(2^{s-\tau}|G|)^{\gamma}\dmutt\right)^{\frac{1}{\gamma}}\right)
\end{aligned}
\end{equation}} \\
\noindent
for some constant $c=c(n,s,L,q,\tau)$. In particular, there exist constants $a_{u}=a_{u}(n,s,L,q,\tau)\in(0,1]$, $a_{f}=a_{f}(n,s,L,q,\tau)\in(0,1]$ and $a_{g}=a_{g}(n,s,L,q,\tau)\in(0,1]$ such that
\begin{equation}
\label{cz non pa : diagonal estimate}
\begin{aligned}
&\sum_{i\geq0}\mu_{\tau,t}\left(\mathcal{Q}_{\rho_{{i}}}(z_{i})\right)+\sum_{j\geq0}\mu_{\tau,t}\left(\mathcal{Q}_{\widetilde{r}_{{j}}}\left(x_{1,j},x_{2,j},t_{0,j}\right)\right)\\
&\leq\frac{c}{\lambda^{p}}\int_{\mathcal{Q}_{r_{2}}\cap\{|D^{\tau}d_{s}u|>a_{u}\lambda\}}|D^{\tau}d_{s}u|^{p}\dmutt+\frac{c}{(\delta\lambda)^{\tilde{q}}}\int_{\mathcal{Q}_{r_{2}}\cap\{|D^{\tau}d_{0}f|>a_{f}\delta\lambda\}}|D^{\tau}d_{0}f|^{\tilde{q}}\dmutt\\
&\quad+\frac{c}{(\delta\lambda)^{b_{\tau}}}\int_{\mathcal{Q}_{r_{2}}\cap\{|G|^{\gamma}>(a_{g}\delta\lambda)^{b_{\tau}}G_{0}^{-1}\}}|G|^{\gamma}G_{0}\dmutt,
\end{aligned}
\end{equation}
where  we denote
\begin{equation}
\label{cz non pa : constb}
b_{\tau}=\frac{2n+4s}{n+2s+2\tau}\quad\text{and}\quad
G_{0}=\left(\int_{Q_{2}}|g|^{\gamma}\dz\right)^{\frac{b_{\tau}-\gamma}{\gamma}}.
\end{equation}
In addition, we get that
\begin{equation}
\label{cz non pa : pstra norm}
\left(\dashint_{\mathcal{Q}_{5^{\frac{1}{s}}\widetilde{r}_{{j}}}\left(x_{1,j},x_{2,j},t_{0,j}\right)}|D^{\tau}d_{s}u|^{p_{\#}}\dmutt\right)^{\frac{1}{{p_{\#}}}}\leq c_{od}\lambda\quad\text{for any }j
\end{equation}
and for some constant $c_{od}=c_{od}(n,s,p)$, where the constant $p_{\#}$ is defined in \eqref{cz non pa : paraconj},

\end{prop}

\begin{rmk}
\label{cz non pa : taurmk}
As we pointed out earlier, \eqref{cz non pa : tautail} is only employed to control the tail term of $f$. Thus if $f=0$, we can remove the condition \eqref{cz non pa : tautail}.
\end{rmk}

\begin{proof}
We first define the functional
\begin{equation}
\label{cz non pa : defnTheta}
\begin{aligned}
    \Theta_{D}\left(z_{0},r\right)&=\left(\dashint_{\mathcal{Q}_{r}(z_{0})}|D^{\tau}d_{s}u|^{p}\dmutt\right)^{\frac{1}{p}}+\tau^{\frac{1}{\gamma}}\left(\sup_{t\in \Lambda_{r}(t_{0})}\dashint_{B_{r}(x_{0})}\frac{|u-(u)_{Q_{r}(z_{0})}|^{2}}{r^{2s+2\tau}}\dx\right)^{\frac{1}{2}}\\
    &\quad+\frac{1}{\delta}\left(\dashint_{\mathcal{Q}_{r}(z_{0})}|D^{\tau}d_{0}f|^{\tilde{q}}\dmutt\right)^{\frac{1}{\tilde{q}}}+\frac{1}{\delta}\left(\dashint_{\mathcal{Q}_{r}(z_{0})}(r^{s-\tau}|G|)^{\gamma}\dmutt\right)^{\frac{1}{\gamma}}
\end{aligned}
\end{equation}
for any $z_{0}\in Q_{r_{1}}$ and $r>0$ with $Q_{r}(z_{0})\subset Q_{2}$. 
The rest of the proof is divided into 8 steps.\\
\noindent
\textbf{Step 1. Coverings for the diagonal part.}
Let us take 
{\small\begin{equation}
\label{cz non pa : lambdacond}
\begin{aligned}
    \lambda_{0}&=  \frac{M\kappa^{-1}\tau^{\frac{1}{q}}}{(r_{2}-r_{1})^{\frac{5n}{s}}}\left(\left(\frac{1}{\tau}\dashint_{\mathcal{Q}_{2}}|D^{\tau}d_{s}u|^{p}\dmutt\right)^{\frac{1}{p}}+\Tail_{\infty,2s}\left(\frac{u-(u)_{B_{2}}(t)}{2^{s+\tau}};Q_{2}\right)+\left(\sup_{t\in \Lambda_{2}}\dashint_{B_{2}}\frac{|u-(u)_{Q_{2}}|^{2}}{2^{2s+2\tau}}\dx\right)^{\frac{1}{2}}\right)\\
    &\quad+\frac{M\kappa^{-1}\tau^{\frac{1}{q}}}{(r_{2}-r_{1})^{\frac{5n}{s}}}\frac{1}{\delta}\left(\left(\frac{1}{\tau}\dashint_{\mathcal{Q}_{2}}|D^{\tau}d_{0}f|^{q}\dmutt\right)^{\frac{1}{q}}+\Tail_{q,s}\left(\frac{f-(f)_{B_{2}}(t)}{2^{\tau}};Q_{2}\right)+\left(\frac{1}{\tau}\dashint_{\mathcal{Q}_{2}}(2^{s-\tau}|G|)^{\gamma}\dmutt\right)^{\frac{1}{\gamma}}\right),
\end{aligned}
\end{equation}}\\
\noindent
where $M\geq1$ and $\kappa\in(0,1]$ are free parameters which we will determine later (see \eqref{cz non pa : condM1} and \eqref{cz non pa : rangeofkappa1}). More precisely, the parameter $M$ will be used to handle the diagonal part and the parameter $\kappa$ will be used to handle the non-diagonal part. We next take a positive integer $j_{0}\geq 5$ such that
\begin{equation}
\label{cz non pa : j0cond}
\frac{16(c_{0}+\tilde{c}+2c_{q})^{2}}{1-2^{-s+\tau+\frac{2s}{\tilde{q}}}}\leq 2^{j_{0}\left(s-\tau-\frac{2s}{\tilde{q}}\right)},
\end{equation}
where $c_{0}$ is the constant determined in Lemma \ref{cz non pa : RHI}, and $\tilde{c}$ and $c_{q}$ are the constants determined in \eqref{cz non pa : tailestimate}. Using \eqref{cz non pa : tildeq}, we observe that the number $j_{0}$ depends only on $n,s,L,q$ and $\tilde{q}$.
We then note that for any $z_{0}\in Q_{r_{1}}$,
\begin{equation*}
    Q_{5^{\frac{2}{s}}\times 2^{j_{0}+3}\mathcal{R}_{1,2}}(z_{0})\subset Q_{r_{2}},
\end{equation*}
where we denote
\begin{equation}
\label{cz non pa : r1r2defn}
\mathcal{R}_{1,2}=2^{-j_{0}-3}\times 5^{-\frac{2}{s}}\times (s(r_{2}-r_{1}))^{\frac{1}{s}}.
\end{equation}
Let us now define for $\lambda\geq\lambda_{0}$,
\begin{equation*}
    D_{\kappa\lambda}=\left\{z_{0}\in Q_{r_{1}}\,:\,\sup_{0<\rho\leq \mathcal{R}_{1,2}}\Theta_{D}\left(z_{0},\rho\right)>\kappa\lambda\right\}.
\end{equation*}
Since $\tau<1$, we observe that for any $z_{0}\in Q_{r_{1}}$ and $r\in \left[\mathcal{R}_{1,2},5^{\frac{2}{s}}\times 2^{j_{0}+3}\mathcal{R}_{1,2}\right]$,
\begin{equation*}
\begin{aligned}
\left(\dashint_{\mathcal{Q}_{r}(z_{0})}|D^{\tau}d_{s}u|^{p}\dmutt\right)^{\frac{1}{p}}+\frac{1}{\delta}\left(\dashint_{\mathcal{Q}_{r}(z_{0})}|D^{\tau}d_{0}f|^{\tilde{q}}\dmutt\right)^{\frac{1}{\tilde{q}}}+\frac{1}{\delta}\left(\dashint_{\mathcal{Q}_{r}(z_{0})}(r^{s-\tau}|G|)^{\gamma}\dmutt\right)^{\frac{1}{\gamma}}\leq \frac{\kappa\lambda}{4}
\end{aligned}
\end{equation*}
holds by assuming
\begin{equation}
\label{cz non pa : condM0}
M\geq 2^{10n(j_{0}+4+5s^{-1})}s^{-\frac{5n}{s}}.
\end{equation}
In addition, using Lemma \ref{cz non pa : RHI} and the fact that $\tau<1$ with $\gamma<2$, we have
{\small\begin{equation}
\label{cz non pa : timel2osc}
\begin{aligned}
    &\tau^{\frac{1}{\gamma}}\left(\sup_{t\in \Lambda_{r}(t_{0})}\dashint_{B_{r}(x_{0})}\frac{|u-(u)_{Q_{r}(z_{0})}|^{2}}{r^{2s+2\tau}}\dx\right)^{\frac{1}{2}}\\
    &\leq c_{0}\left(\dashint_{\mathcal{Q}_{2r}(z_{0})}|D^{\tau}d_{s}u|^{2}\dmutt\right)^{\frac{1}{2}}+c_{0}\left(\dashint_{\mathcal{Q}_{2r}(z_{0})}|D^{\tau}d_{0}f|^{2}\dmutt\right)^{\frac{1}{2}}+c_{0}\left(\dashint_{\mathcal{Q}_{2r}(z_{0})}\left((2r)^{s-\tau}|G|\right)^{\gamma}\dmutt\right)^{\frac{1}{\gamma}}\\
&\quad+\underbrace{c_{0}\tau^{\frac{1}{2}}\Tail_{2,2s}\left(\frac{u-(u)_{B_{2r}(x_{0})}(t)}{(2r)^{s+\tau}};Q_{2r}(z_{0})\right)}_{T_{1}}+\underbrace{c_{0}\tau^{\frac{1}{2}}\Tail_{2,s}\left(\frac{f-(f)_{B_{2r}(x_{0})}(t)}{(2r)^{\tau}};Q_{2r}(z_{0})\right)}_{T_{2}},
\end{aligned}
\end{equation}}\\
\noindent
where the constant $c_{0}$ is determined in Lemma \ref{cz non pa : RHI}.
By H\"older's inequality and \eqref{cz non pa : tailestimate}, we further estimate $T_{1}$ and $T_{2}$ as
\begin{equation}
\label{cz non pa t1t2tail}
\begin{aligned}
    T_{1}+T_{2}&\leq c_{0}\tau^{\frac{1}{2}}\Tail_{p,2s}\left(\frac{|u-(u)_{B_{2r}(x_{0})}(t)|}{(2r)^{s+\tau}};Q_{2r}(z_{0})\right)+c_{0}\tau^{\frac{1}{2}}\Tail_{\tilde{q},s}\left(\frac{|f-(f)_{B_{2r}(x_{0})}(t)|}{(2r)^{\tau}};Q_{2r}(z_{0})\right)\\
    &\leq c_{0}c_{p}\left(\frac{2}{\mathcal{R}_{1,2}}\right)^{s+\tau+\frac{2s}{p}}\left(\dashint_{\mathcal{Q}_{2}}|D^{\tau}d_{s}u|^{p}\right)^{\frac{1}{p}}+c_{0}c_{q}\left(\frac{2}{\mathcal{R}_{1,2}}\right)^{\tau+\frac{2s}{q}}\left(\dashint_{\mathcal{Q}_{2}}|D^{\tau}d_{0}f|^{q}\right)^{\frac{1}{q}}\\
    &\quad +c_{0}\tilde{c}\left(\frac{2}{\mathcal{R}_{1,2}}\right)^{5n}\Tail_{p,2s}\left(\frac{u-(u)_{B_{2}}(t)}{2^{s+\tau}};Q_{4}\right)+c_{0}\tilde{c}\left(\frac{2}{\mathcal{R}_{1,2}}\right)^{5n}\Tail_{q,s}\left(\frac{f-(f)_{B_{2}}(t)}{2^{\tau}};Q_{2}\right),
\end{aligned}
\end{equation}
where the constants $c_{p}, c_{q}$ and $\tilde{c}$ are determined in Lemma \ref{cz non pa : taillem}.
Using Remark \ref{cz non pa : cp} and H\"older's inequality to the third term on the right-hand side of \eqref{cz non pa t1t2tail},
we deduce from \eqref{cz non pa : timel2osc} that
\begin{equation*}
    \tau^{\frac{1}{\gamma}}\left(\sup_{t\in \Lambda_{r}(t_{0})}\dashint_{B_{r}(x_{0})}\frac{|u-(u)_{Q_{r}(z_{0})}|^{2}}{r^{2s+2\tau}}\dx\right)^{\frac{1}{2}}\leq \frac{\kappa\lambda}{4}
\end{equation*}
holds by taking
\begin{equation}
\label{cz non pa : condM1}
M=(c_{0}c_{q}+c_{0}\tilde{c})2^{10n(j_{0}+4+5s^{-1})}s^{-\frac{5n}{s}}
\end{equation}
which clearly satisfies \eqref{cz non pa : condM0}.
As a result, we observe that for any $z_{0}\in Q_{r_{1}}$ and $r\in\left[\mathcal{R}_{1,2},5^{\frac{2}{s}}\times 2^{j_{0}+3}\mathcal{R}_{1,2}\right]$,
\begin{equation}
\label{cz non pa : thetales}
    \Theta_{D}(z_{0},r)\leq \kappa\lambda.
\end{equation}
Therefore, for each $z\in D_{\kappa\lambda}$, there is an exit radius $\rho_{z}\leq \mathcal{R}_{1,2}$  such that
\begin{equation}
\label{cz non pa : exitradius}
    \Theta_{D}\left(z,\rho_{z}\right)\geq\kappa\lambda\quad\text{and}\quad\Theta_{D}\left(z,\rho\right)\leq\kappa\lambda\quad\text{ if }\rho_{z}\leq\rho\leq 5^{\frac{2}{s}}\times 2^{j_{0}+3}\mathcal{R}_{1,2}.
\end{equation}
We first observe that if $Q_{\rho}(z_{1})\cap Q_{r}(z_{2})\neq\emptyset$ with $\frac{\rho}{2}\leq r\leq2\rho$, then we have $Q_{r}(z_{2})\subset Q_{5^{\frac{1}{s}}\rho}(z_{1})$.
Thus we apply Vitali's covering lemma to the collection $\{Q_{2^{j_{0}}\rho_{z}}(z)\}_{z\in D_{\kappa\lambda}}$, in order to 
find a family of mutually disjoint countable cylinders
\begin{equation}
\label{cz non pa : disjointcydi}
     \left\{Q_{2^{j_{0}}\rho_{z_{i}}}(z_{i})\right\}_{i\geq0}\text{ such that } D_{\kappa\lambda}\subset\bigcup_{z_{0}\in D_{\kappa\lambda}}Q_{\rho_{z_{0}}}(z_{0})\subset\bigcup\limits_{i=0}^{\infty}Q_{5^{\frac{1}{s}}\times 2^{j_{0}}\rho_{z_{i}}}(z_{i}).
\end{equation}
In addition, by the proof of Vitali's covering lemma, we get that for any $z\in D_{\kappa\lambda}$, there is $i$ such that 
\begin{equation}
\label{cz non pa : lem.cov.step1.vidist}
    \frac{2^{j_{0}}\rho_{z_{i}}}{2}\leq 2^{j_{0}}\rho_{z}\leq 2^{j_{0}+1}\rho_{z_{i}}\quad\text{and}\quad Q_{2^{j_{0}}\rho_{z}}(z)\subset Q_{5^{\frac{1}{s}}\times \rho_{{i}}}(z_{i}),
\end{equation}
where we denote 
\begin{equation}
\label{cz non pa : rhoidefn}
    \rho_{i}=2^{j_{0}}\rho_{z_{i}}\quad\text{for each }i.
\end{equation}
From \eqref{cz non pa : exitradius}, we have
\begin{equation}
\label{cz non pa : kalam}
\begin{aligned}
    \kappa\lambda&\leq \left(\dashint_{\mathcal{Q}_{\rho_{z_{i}}}(z_{i})}|D^{\tau}d_{s}u|^{p}\dmutt\right)^{\frac{1}{p}}+\tau^{\frac{1}{\gamma}}\left(\sup_{t\in \Lambda_{\rho_{z_{i}}}(t_{i})}\dashint_{B_{\rho_{z_{i}}}(x_{i})}\frac{|u-(u)_{Q_{\rho_{z_{i}}}(z_{i})}|^{2}}{\rho_{z_{i}}^{2s+2\tau}}\dx\right)^{\frac{1}{2}}\\
    &\quad+\frac{1}{\delta}\left(\dashint_{\mathcal{Q}_{\rho_{z_{i}}}(z_{i})}|D^{\tau}d_{0}f|^{\tilde{q}}\dmutt\right)^{\frac{1}{\tilde{q}}}+\frac{1}{\delta}\left(\dashint_{\mathcal{Q}_{\rho_{z_{i}}}(z_{i})}\left((\rho_{z_{i}})^{s-\tau}|G|\right)^{\gamma}\dmutt\right)^{\frac{1}{\gamma}}.
\end{aligned}
\end{equation}
We first note from \eqref{cz non pa : tailestimate} with $\beta=s$, $\tilde{s}=s$ and $k=j_{0}$  that
{\small\begin{equation}
\label{cz non pa : findj0u}
\begin{aligned}
\tau^{\frac{1}{\gamma}}\Tail_{p,2s}\left(\frac{u-(u)_{B_{2\rho}(x_{i})}(t)}{(2\rho)^{s+\tau}};Q_{2\rho}(z_{i})\right)&\leq c_{p}\sum_{j=2}^{j_{0}}2^{i\left(-s+\tau+\frac{2s}{p}\right)}\left(\dashint_{\mathcal{Q}_{2^{j}\rho}(z_{i})}|D^{\tau}d_{s}u|^{p}\dmutt\right)^{\frac{1}{p}}\\
&\quad+\tilde{c}\tau^{\frac{1}{\gamma}}\sum_{j=j_{0}+1}^{l}2^{j(-s+\tau)}\left(\sup_{t\in \Lambda_{2^{j}\rho}(t_{i})}\dashint_{B_{2^{j
}\rho}(x_{i})}\frac{|u-(u)_{B_{2^{j}\rho}(x_{i})}(t)|^{2}}{(2^{j}\rho)^{2{s}+2\tau}}\dx\right)^{\frac{1}{2}}\\
    &\quad+\tilde{c}\tau^{\frac{1}{\gamma}}2^{-2s l}\left(\frac{2}{\rho}\right)^{s+\tau}\left(\sup_{t\in \Lambda_{2}}\dashint_{B_{2}}\frac{|u-(u)_{B_{2}}(t)|^{2}}{2^{2s+2\tau}}\dx\right)^{\frac{1}{2}}\\
    &\quad+\frac{\tilde{c}}{(2-r_{1})^{n+2s}}\left(\frac{2}{\rho}\right)^{-s+\tau}\Tail_{\infty,2s}\left(\frac{u-(u)_{B_{2}}(t)}{2^{\tilde{s}+\tau}};Q_{2}\right),
\end{aligned}
\end{equation}}\\
\noindent
where $\rho=\rho_{z_{i}}$ and $l$ is the positive integer such that  $2^{j_{0}+1}\mathcal{R}_{1,2}\leq 2^{l}\rho_{z_{i}}<2^{j_{0}+2}\mathcal{R}_{1,2}$.
We point out that for the last tail term in \eqref{cz non pa : findj0u}, we have taken supremum in the time variable for the expression in \eqref{cz non pa : taiilestimt2}.
Similarly, we observe 
{\small\begin{equation}
\label{cz non pa : findj0f}
\begin{aligned}
&\tau^{\frac{1}{\gamma}}\Tail_{\tilde{q},s}\left(\frac{f-(f)_{B_{2\rho}(x_{i})}(t)}{(2\rho)^{\tau}};Q_{2\rho}(z_{i})\right)\\
&\leq c_{\tilde{q}}\sum_{j=2}^{j_{0}}2^{i\left(-s+\tau+\frac{2s}{\tilde{q}}\right)}\left(\dashint_{\mathcal{Q}_{2^{j}\rho}(z_{i})}|D^{\tau}d_{0}f|^{\tilde{q}}\dmutt\right)^{\frac{1}{\tilde{q}}}+c_{\tilde{q}}\sum_{j=j_{0}+1}^{l}2^{i\left(-s+\tau+\frac{2s}{\tilde{q}}\right)}\left(\dashint_{\mathcal{Q}_{2^{j}\rho}(z_{i})}|D^{\tau}d_{0}f|^{\tilde{q}}\dmutt\right)^{\frac{1}{\tilde{q}}}\\
    &\quad+c_{\tilde{q}}2^{-s l}\left(\frac{2}{\rho}\right)^{\tau+\frac{2s}{\tilde{q}}}\left(\dashint_{\mathcal{Q}_{2}}|D^{\tau}d_{0}f|^{q} \dmutt\right)^{\frac{1}{q}}+\frac{\tilde{c}}{(2-r_{1})^{n+2s}}\left(\frac{2}{\rho}\right)^{-s+\tau+\frac{2s}{q}}\Tail_{q,s}\left(\frac{f-(f)_{B_{2}}(t)}{2^{\tau}};Q_{2}\right),
\end{aligned}
\end{equation}}\\
\noindent
where we have used H\"older's inequality for the third and fourth terms in the right-hand side of \eqref{cz non pa : findj0f}. We combine \eqref{cz non pa : timel2osc} with $r=\rho_{z_{i}}$, \eqref{cz non pa : findj0u} and \eqref{cz non pa : findj0f} together with \eqref{cz non pa : tildeqprop}, \eqref{cz non pa : rhoidefn} and Remark \ref{cz non pa : cp}, in order to get 
{\small\begin{equation}
\label{cz non pa : oneretildeqnec}
\begin{aligned}
    &\tau^{\frac{1}{\gamma}}\sup_{t\in \Lambda_{\rho_{z_{i}}}(t_{i})}\left(\dashint_{B_{\rho_{z_{i}}}(x_{i})}\frac{|u-(u)_{Q_{\rho_{z_{i}}}(z_{i})}|^{2}}{\rho_{z_{i}}^{2s+2\tau}}\dx\right)^{\frac{1}{2}}\\
    &\leq c\left(\dashint_{\mathcal{Q}_{\rho_{i}}(z_{i})}|D^{\tau}d_{s}u|^{p}\dmutt\right)^{\frac{1}{p}}+c\left(\dashint_{\mathcal{Q}_{\rho_{i}}(z_{i})}|D^{\tau}d_{0}f|^{\tilde{q}}\dmutt\right)^{\frac{1}{\tilde{q}}}+c\left(\dashint_{\mathcal{Q}_{\rho_{{i}}}(z_{i})}\left((\rho_{{i}})^{s-\tau}|G|\right)^{\gamma}\dmutt\right)^{\frac{1}{\gamma}}\\
    &\quad+ c_{0}\tilde{c}\left[\sum_{j=j_{0}+1}^{l}2^{i\left(-s+\tau\right)}\Theta_{D}\left(z_{j},2^{j}\rho_{z_{i}}\right)+\tau^{\frac{1}{\gamma}}\left(\frac{2}{2^{j_{0}}\mathcal{R}_{1,2}}\right)^{s+\tau}\left(\sup_{t\in \Lambda_{2}}\dashint_{B_{2}}\frac{|u-(u)_{B_{2}}(t)|^{2}}{2^{2s+2\tau}}\dx\right)^{\frac{1}{2}}\right]\\
    &\quad+ c_{0}c_{q}\left[\sum_{j=j_{0}+1}^{l}2^{i\left(-s+\tau+\frac{2s}{\tilde{q}}\right)}\Theta_{D}\left(z_{j},2^{j}\rho_{z_{i}}\right)+\left(\frac{2}{2^{j_{0}}\mathcal{R}_{1,2}}\right)^{\tau+\frac{2s}{\tilde{q}}}\left(\dashint_{\mathcal{Q}_{2}}|D^{\tau}d_{0}f|^{q}\dmutt\right)^{\frac{1}{q}}\right]\\
    &\quad+ \frac{c_{0}\tilde{c}}{(r_{2}-r_{1})^{5n}}\left(\Tail_{\infty,2s}\left(\frac{u-(u)_{B_{2}}(t)}{2^{s+\tau}};Q_{2}\right)+\Tail_{q,s}\left(\frac{f-(f)_{B_{2}}(t)}{2^{\tau}};Q_{2}\right)\right),
\end{aligned}
\end{equation}}\\
\noindent
where $c=c(n,s,L,q,\tau)$. We here highlight that \eqref{cz non pa : tildeqprop} is necessary to handle the sixth term in the right-hand side of \eqref{cz non pa : oneretildeqnec}, as $\sum\limits_{j=j_{0}+1}^{\infty}2^{i\left(-s+\tau+\frac{2s}{\tilde{q}}\right)}<\frac{2^{j_{0}\left(-s+\tau+\frac{2s}{\tilde{q}}\right)}}{1-2^{-s+\tau+\frac{2s}{\tilde{q}}}}$.
We further estimate the right-hand side of \eqref{cz non pa : oneretildeqnec} using \eqref{cz non pa : tildeqprop}, \eqref{cz non pa : j0cond}, \eqref{cz non pa : r1r2defn}, \eqref{cz non pa : condM1} and \eqref{cz non pa : exitradius} as 
{\small\begin{equation*}
\begin{aligned}
    &\tau^{\frac{1}{\gamma}}\sup_{t\in \Lambda_{\rho_{z_{i}}}(t_{i})}\left(\dashint_{B_{\rho_{z_{i}}}(x_{i})}\frac{|u-(u)_{Q_{\rho_{z_{i}}}(z_{i})}|^{2}}{\rho_{z_{i}}^{2s+2\tau}}\dx\right)^{\frac{1}{2}}\\
    &\leq c\left(\dashint_{\mathcal{Q}_{\rho_{i}}(z_{i})}|D^{\tau}d_{s}u|^{p}\dmutt\right)^{\frac{1}{p}}+c\left(\dashint_{\mathcal{Q}_{\rho_{i}}(z_{i})}|D^{\tau}d_{0}f|^{\tilde{q}}\dmutt\right)^{\frac{1}{\tilde{q}}}+c\left(\dashint_{\mathcal{Q}_{\rho_{{i}}}(z_{i})}\left((\rho_{{i}})^{s-\tau}|G|\right)^{\gamma}\dmutt\right)^{\frac{1}{\gamma}}+\frac{\kappa\lambda}{4}.
\end{aligned}
\end{equation*}}\\
\noindent
Plugging the above estimate into \eqref{cz non pa : kalam} along with \eqref{cz non pa : rhoidefn}, we find that
\begin{equation*}
    \kappa\lambda
    \leq c\left(\dashint_{\mathcal{Q}_{\rho_{i}}(z_{i})}|D^{\tau}d_{s}u|^{p}\dmutt\right)^{\frac{1}{p}}+\frac{c}{\delta}\left(\dashint_{\mathcal{Q}_{\rho_{i}}(z_{i})}|D^{\tau}d_{0}f|^{\tilde{q}}\dmutt\right)^{\frac{1}{\tilde{q}}}+\frac{c}{\delta}\left(\dashint_{\mathcal{Q}_{\rho_{{i}}}(z_{i})}\left((\rho_{{i}})^{s-\tau}|G|\right)^{\gamma}\dmutt\right)^{\frac{1}{\gamma}}
\end{equation*}
for some constant $c=c(n,s,L,q,\tau)$.
Therefore we deduce that one of the following must hold:
\begin{equation}
\label{cz non pa : diaalter3}
\begin{aligned}
    &\frac{\kappa\lambda}{3}\leq c\left(\dashint_{\mathcal{Q}_{\rho_{i}}(z_{i})}|D^{\tau}d_{s}u|^{p}\dmutt\right)^{\frac{1}{p}},\quad \frac{\kappa\lambda}{3}\leq \frac{c}{\delta}\left(\dashint_{\mathcal{Q}_{\rho_{i}}(z_{i})}|D^{\tau}d_{0}f|^{\tilde{q}}\dmutt\right)^{\frac{1}{\tilde{q}}},\\
    &\frac{\kappa\lambda}{3}\leq \frac{c}{\delta}\left(\dashint_{\mathcal{Q}_{\rho_{{i}}}(z_{i})}\left((\rho_{{i}})^{s-\tau}|G|\right)^{\gamma}\dmutt\right)^{\frac{1}{\gamma}}.
\end{aligned}
\end{equation}
If the first inequality or the second inequality in \eqref{cz non pa : diaalter3} holds, then we get 
\begin{equation}
\label{cz non pa : measurealter12}
    \mu_{\tau,t}\left(\mathcal{Q}_{\rho_{{i}}}(z_{i})\right)\leq\frac{c}{(\kappa\lambda)^{p}}\int_{\mathcal{Q}_{\rho_{{i}}}(z_{i})}|D^{\tau}d_{s}u|^{p}\dmutt\quad\text{or}\quad \mu_{\tau,t}\left(\mathcal{Q}_{\rho_{{i}}}(z_{i})\right)\leq\frac{c}{(\kappa\delta\lambda)^{\tilde{q}}}\int_{\mathcal{Q}_{\rho_{{i}}}(z_{i})}|D^{\tau}d_{0}f|^{\tilde{q}}\dmutt.
\end{equation}
On the other hand, if the third inequality in \eqref{cz non pa : diaalter3} holds, we observe that
\begin{equation*}
\begin{aligned}
    \mu_{\tau,t}\left(\mathcal{Q}_{\rho_{{i}}}(z_{i})\right)\leq \frac{c}{(\kappa\delta\lambda)^{\gamma}}\int_{\mathcal{Q}_{\rho_{{i}}}(z_{i})}\left((\rho_{{i}})^{s-\tau}|G|\right)^{\gamma}\dmutt= \frac{c}{(\kappa\delta\lambda)^{\gamma}}\int_{\mathcal{Q}_{\rho_{{i}}}(z_{i})}\left((\rho_{{i}})^{s-\tau}|G|\right)^{\gamma}G_{0}G_{0}^{-1}\dmutt.
\end{aligned}
\end{equation*}
We note from \eqref{cz non pa : constb} that
{\small\begin{equation*}
\begin{aligned}
    \rho_{i}^{\gamma(s-\tau)}G_{0}^{-1}\leq \rho_{i}^{\gamma(s-\tau)}\left(\int_{{Q}_{\rho_{{i}}}(z_{i})}|g|^{\gamma}\right)^{\frac{\gamma-b_{\tau}}{\gamma}}&\leq c\rho_{i}^{\gamma(s-\tau)}\left(\dashint_{\mathcal{Q}_{\rho_{{i}}}(z_{i})}\rho_{i}^{n+2s}|G|^{\gamma}\dmutt\right)^{\frac{\gamma-b_{\tau}}{\gamma}}\\
    &\leq c\left(\dashint_{\mathcal{Q}_{\rho_{{i}}}(z_{i})}\left((\rho_{{i}})^{s-\tau}|G|\right)^{\gamma}\dmutt\right)^{\frac{\gamma-b_{\tau}}{\gamma}}.
\end{aligned}
\end{equation*}}\\
\noindent
Using the above two estimates along with the third inequality in \eqref{cz non pa : diaalter3}, we have
\begin{equation}
\label{cz non pa : measurealter3}
    \mu_{\tau,t}\left(\mathcal{Q}_{\rho_{{i}}}(z_{i})\right)\leq \frac{c}{(\delta\kappa\lambda)^{b_{\tau}}}\int_{\mathcal{Q}_{\rho_{{i}}}(z_{i})}|G|^{\gamma}G_{0}\dmutt.
\end{equation}
We combine \eqref{cz non pa : measurealter12} and \eqref{cz non pa : measurealter3} to see that
{\small\begin{equation*}
    \mu_{\tau,t}\left(\mathcal{Q}_{\rho_{{i}}}(z_{i})\right)\leq\frac{c}{(\kappa\lambda)^{p}}\int_{\mathcal{Q}_{\rho_{{i}}}(z_{i})}|D^{\tau}d_{s}u|^{p}\dmutt+\frac{c}{(\kappa\delta\lambda)^{\tilde{q}}}\int_{\mathcal{Q}_{\rho_{{i}}}(z_{i})}|D^{\tau}d_{0}f|^{\tilde{q}}\dmutt+\frac{c}{(\delta\kappa\lambda)^{b_{\tau}}}\int_{\mathcal{Q}_{\rho_{{i}}}(z_{i})}|G|^{\gamma}G_{0}\dmutt.
\end{equation*}}\\
\noindent
A suitable choice of the constants $\tilde{a}_{u}=\tilde{a}_{u}(n,s,L,q,\tau)\in\left(0,\frac{1}{8}\right]$, $\tilde{a}_{f}=\tilde{a}_{f}(n,s,L,q,\tau)\in(0,1]$ and $\tilde{a}_{g}=\tilde{a}_{g}(n,s,L,q,\tau)\in(0,1]$ yields 
{\small\begin{equation}
\label{cz non pa : diab}
\begin{aligned}
\mu_{\tau,t}\left(\mathcal{Q}_{\rho_{i}}(z_{i})\right)&\leq \frac{c}{(\kappa\lambda)^{p}}\int_{\mathcal{Q}_{\rho_{i}}(z_{i})\cap\{|D^{\tau}d_{s}u|>\tilde{a}_{u}\kappa\lambda\}}|D^{\tau}d_{s}u|^{p}\dmutt+\frac{c}{(\kappa\delta\lambda)^{\tilde{q}}}\int_{\mathcal{Q}_{\rho_{i}}(z_{i})\cap\{|D^{\tau}d_{0}f|>\tilde{a}_{f}\kappa\delta\lambda\}}|D^{\tau}d_{0}f|^{\tilde{q}}\dmutt\\
&\quad+\frac{c}{(\kappa\delta\lambda)^{b_{\tau}}}\int_{\mathcal{Q}_{\rho_{i}}(z_{i})\cap\{|G|^{\gamma}>(\tilde{a}_{g}\kappa\delta\lambda)^{b_{\tau}}G_{0}^{-1}\}}|G|^{\gamma}G_{0}\dmutt.
\end{aligned}
\end{equation}}
\noindent
\begin{rmk}
\label{cz non pa : imprmk1}
    We here remark on the second term appearing on the right-hand side of \eqref{cz non pa : defnTheta}. We first note that this term is used to handle parabolic tail terms. 
    Since $u\in C\left(-2^{2s},2^{2s};L^{2}(B_{2})\right)$, there may exist some points $z_{0}\in Q_{r_{1}}$ such that $\Theta_{D}(z_{0},\rho_{z_{0}})\geq \kappa\lambda$ and
    \begin{equation*}
        \left(\dashint_{\mathcal{Q}_{\rho_{z_{0}}}(z_{0})}|D^{\tau}d_{s}u|^{p}\dmutt\right)^{\frac{1}{p}}\leq\frac{\kappa\lambda}{2}\quad\text{with}\quad\tau^{\frac{1}{\gamma}}\left(\sup_{t\in \Lambda_{\rho_{z_{0}}}(t_{0})}\dashint_{B_{\rho_{z_{0}}}(x_{0})}\frac{|u-(u)_{Q_{\rho_{z_{0}}}(z_{0})}|^{2}}{\rho_{z_{0}}^{2s+2\tau}}\dx\right)^{\frac{1}{2}}\geq \frac{\kappa\lambda}{2},
    \end{equation*}
    where $\rho_{z_{0}}$ is the exit-radius of the point $z_{0}$.
    However, from energy estimates and rigorous tail estimates, we still have a sufficiently good bound on the measure of such cylinders as in \eqref{cz non pa : diab}, which is an essential ingredient to obtain $L^{q}$-regularity of $D^{\tau}d_{s}u$.
\end{rmk}

\noindent
\textbf{Step 2. Coverings for off-diagonal parts.} We first note that for any $(x,y,t)\in \mathcal{Q}_{r_{1}}$ and $r\in\left(0,5^{\frac{2}{s}}\times 2^{j_{0}+3}{\mathcal{R}_{1,2}}\right]$, 
we have 
\begin{equation*}
    \mathcal{Q}_{r}(x,y,t)\subset\mathcal{Q}_{r_{2}},
\end{equation*}
where {the parameter} $\mathcal{R}_{1,2}$ is defined in \eqref{cz non pa : r1r2defn}. 
Let us define 
\begin{equation}
\label{cz non pa : peneftnlofu}
\begin{aligned}
    E_{p,\tau}\left(u\,;\,Q_{r}(z_{0})\right)&=\left(\dashint_{\mathcal{Q}_{r}(z_{0})}|D^{\tau}d_{s}u|^{p}\dmutt\right)^{\frac{1}{p}}+\tau^{\frac1{\gamma}}\left(\sup_{t\in \Lambda_{r}(t_{0})}\dashint_{B_{r}(x_{0})}\frac{|u-(u)_{Q_{r}(z_{0})}|^{2}}{r^{2s+2\tau}}\dx\right)^{\frac{1}{2}}\\
    &\quad+\frac{1}{\delta}\left(\dashint_{\mathcal{Q}_r(z_0)}|D^\tau d_0f|^{\tilde{q}}\dmutt\right)^{\frac1{\tilde{q}}},
\end{aligned}
\end{equation}
for any $Q_{r}(z_{0})\Subset Q_{2}$. Then we first note that for any $\rho<r$ ,
\begin{equation}
\label{cz non pa : timesize}
\sup_{t\in \Lambda_{\rho}}\left(\dashint_{B_{\rho}}\frac{|u-(u)_{Q_{\rho}}|^{2}}{\rho^{2s+2\tau}}\dx\right)^{\frac{1}{2}}\leq 2\left(\frac{r}{\rho}\right)^{\frac{n}{2}+s+\tau}\sup_{t\in \Lambda_{r}}\left(\dashint_{B_{r}}\frac{|u-(u)_{Q_{r}}|^{2}}{r^{2s+2\tau}}\dx\right)^{\frac{1}{2}}.
\end{equation}
We now intend to employ exit-time arguments, to this end, we first introduce a set function which is defined by
\begin{equation}
\label{cz non pa : cubecoeff}
\begin{aligned}
    \mathfrak{A}\left(\mathcal{B}_{r}\right)=\begin{cases}
    \left(\frac{r}{\mathrm{dist}\left({B^{1}_{r}},{B^{2}_{r}}\right)}\right)^{s+\tau}&\text{if }\mathrm{dist}\left(B_{r}^{1},B_{r}^{2}\right)\geq r,\\
        1&\text{if }\mathrm{dist}\left(B_{r}^{1},B_{r}^{2}\right)< r,
    \end{cases}
\end{aligned}
\end{equation}
where $\mathcal{B}_{r}=B_{r}^{1}\times B_{r}^{2}$. We next define a functional 
{\small\begin{equation}\label{cz non pa : func-C}
\begin{aligned}
    \Theta_{OD}\left(u;\mathcal{Q}_{r}(x_{1},x_{2},t_{0})\right)&=\left(\dashint_{\mathcal{Q}_{r}(x_{1},x_{2},t_{0})}|D^{\tau}d_{s}u|^{p}\dmutt\right)^{\frac{1}{p}}+ {\mathfrak{A}\left(\mathcal{B}_{r}(x_{1},x_{2})\right) } \sum_{d=1}^{2}E_{p,\tau}\left(u\,;\,Q_{r}(x_{d},t_{0})\right).
\end{aligned}
\end{equation}}
We then observe from \eqref{cz non pa : inclusion measure} and \eqref{cz non pa : thetales} that for any $r\in \left[5^{-\frac{2}{s}}{\mathcal{R}_{1,2}},5^{\frac{2}{s}}\times 2^{j_{0}+3}\mathcal{R}_{1,2}\right]$,
\begin{equation*}
    \Theta_{OD}\left(u;\mathcal{Q}_{r}(x_{1},x_{2},t_{0})\right)\leq \lambda
\end{equation*}
holds by taking 
\begin{equation}
\label{cz non pa : rangeofkappa1}
\kappa=\frac{\tau^{\frac{1}{\gamma}}}{2^{4n j_0}5^{\frac{20n}{s}}(C_{n}+1)},
\end{equation}
where the constants $C_{n}$ and $j_0$ are determined in \eqref{cz non pa : inclusion measure} and \eqref{cz non pa : j0cond}, respectively.
Let us define 
\begin{equation*}
    OD_{\lambda}=\left\{(x_{1},x_{2},t_{0})\in \mathcal{Q}_{r_{1}}\,:\, \sup_{0<\rho\leq 5^{-\frac{2}{s}}{\mathcal{R}_{1,2}}}\Theta_{OD}\left(u;\mathcal{Q}_{\rho}(x_{1},x_{2},t_{0})\right)\geq\lambda\right\}.
\end{equation*}
For each $(x_{1},x_{2},t_{0})\in OD_{\lambda}$, there is an exit-time radius $\overline{r}\leq 5^{-\frac{2}{s}}{\mathcal{R}_{1,2}}$ such that 
\begin{equation}
\label{cz non pa : czdereof}
    \Theta_{OD}\left(u;\mathcal{Q}_{\rho}(x_{1},x_{2},t_{0})\right)\leq \lambda\quad\text{if }\rho\geq \overline{r}\quad\text{and}\quad  \Theta_{OD}\left(u;\mathcal{Q}_{\overline{r}}(x_{1},x_{2},t_{0})\right)\geq\lambda.
\end{equation}
Using Vitali's covering lemma, we find a collection $\widetilde{\mathcal{A}}=\left\{\mathcal{Q}_{2^{j_0}\overline{r}_j}(x_{1,j},x_{2,j},t_{0,j})\right\}_{j\geq0}$ whose elements are mutually disjoint and satisfy
\begin{equation*}
    OD_{\lambda}\subset \bigcup_{j\geq0}\mathcal{Q}_{5^{\frac{1}{s}}2^{j_0}\overline{r}_{j}}(x_{1,j},x_{2,j},t_{0,j}),
\end{equation*}
where we denote by $\overline{r}_j$ the exit-time radius of the point $(x_{1,j},x_{2,j},t_{0,j})$.
Therefore, we have 
\begin{equation}
\label{upperlevelsetoffb}
\left\{(x,y,t)\in \mathcal{Q}_{r_{1}}\,:\, |D^{\tau}d_{s}u(x,y,t)|\geq\lambda \right\}\subset \bigcup_{j\geq0} \mathcal{Q}_{5^{\frac{1}{s}}\widetilde{r}_{j}}(x_{1,j},x_{2,j},t_{0,j}),
\end{equation}
where we denote 
\begin{align}\label{defn.tilder}
    \widetilde{r}_j\equiv 2^{j_0}\overline{r}_j.
\end{align}
\begin{rmk}
\label{cz non pa : imprmk2}

We give some remarks on the functional defined in \eqref{cz non pa : func-C}.  We first note that the second term in the right-hand side of \eqref{cz non pa : func-C} is needed to obtain a bound on the $L^{p_{\#}}$-norm of $D^{\tau}d_{s}u$ (see \eqref{cz non pa : reverseestimate}, below). We now explain how to obtain a good upper bound on the measure of $\mathcal{Q}_{\widetilde{r}_{j}}(x_{1,j},x_{2,j},t_{0,j})\in \widetilde{\mathcal{A}}$ which is an essential ingredient to get $L^{q}$-regularity of $D^{\tau}d_{s}u$. Indeed, by \eqref{cz non pa : meascom}, \eqref{cz non pa : j0cond} and \eqref{defn.tilder}, we observe 
\begin{align*}
    \mu_{\tau,t}\left(\mathcal{Q}_{\widetilde{r}_{j}}(x_{1,j},x_{2,j},t_{0,j})\right)\leq c\mu_{\tau,t}\left(\mathcal{Q}_{\overline{r}_{j}}(x_{1,j},x_{2,j},t_{0,j})\right)
\end{align*}
for some constant $c=c(n,s,\Lambda,q,\tau)$, which implies that  
it suffices to investigate a good upper bound on the measure of $\mathcal{Q}_{\overline{r}_{j}}(x_{1,j},x_{2,j},t_{0,j})$.
Suppose that the selected cylinder $\mathcal{Q}_{\widetilde{r}_{j}}(x_{1,j},x_{2,j},t_{0,j})$ is close to the diagonal. In the next step, we will see that such a cylinder is indeed contained in a diagonal cylinder which we did choose in step 1. Therefore, the measure of $\mathcal{Q}_{\widetilde{r}_{j}}(x_{1,j},x_{2,j},t_{0,j})$ has a good upper bound as in \eqref{cz non pa : diab}.  

On the other hand, if the selected cylinder $\mathcal{Q}_{\widetilde{r}_{j}}(x_{1,j},x_{2,j},t_{0,j})\in {\widetilde{\mathcal{A}}}$ is not close to the diagonal, then 
from \eqref{cz non pa : czdereof}, we  observe that there holds
{\small\begin{equation}
\label{cz non pa : meaoffest}
\begin{aligned}
    \frac{3\lambda}{4}\leq&\left(\frac{1}{\mu_{\tau,t}\left(\mathcal{Q}_{\overline{r}_{j}}(x_{1,j},x_{2,j},t_{0,j})\right)}\int_{\mathcal{Q}_{\overline{r}_{j}}(x_{1,j},x_{2,j},t_{0,j})\cap \left\{|D^{\tau}d_{s}u|>\frac{\lambda}{16}\right\}}|D^{\tau}d_{s}u|^{p}\dmutt\right)^{\frac{1}{p}}\\
    &\quad+\mathfrak{A}\left(\mathcal{B}_{\overline{r}_{j}}(x_{1,j},x_{2,j})\right)^{s+\tau}\left[\sum_{d=1}^{2}E_{p,\tau}\left(u\,;\,Q_{\overline{r}_{j}}(x_{d,j},t_{0,j})\right)\right].
\end{aligned}
\end{equation}}\\
\noindent
Our approach to obtaining a good upper bound on the measure of such cylinders depends on the size of $E_{p,\tau}\left(u\,;\,Q_{\overline{r}_{j}}(x_{d,j},t_{0,j})\right)$. Indeed, if $E_{p,\tau}\left(u\,;\,Q_{\overline{r}_{j}}(x_{d,j},t_{0,j})\right)<\frac{\lambda}{16}$, then the second term on the right-hand side of \eqref{cz non pa : meaoffest} can be absorbed to the left-hand side.  Thus we have a good upper bound on the measure of $\mathcal{Q}_{\widetilde{r}_{j}}(x_{1,j},x_{2,j},t_{0,j})\in \widetilde{\mathcal{A}}$ (see \eqref{cz non pa : measqad} in step 6 below). We next suppose $E_{p,\tau}\left(u\,;\,Q_{\overline{r}_{j}}(x_{d,j},t_{0,j})\right)\geq\frac{\lambda}{16}$. Then we first note that if 
\begin{equation}\label{simple.cal}
    \lambda<a+b+c\Longrightarrow \mbox{$\left(\frac{\lambda}{3}\right)^{p}<a^{p}$, $\left(\frac{\lambda}{3}\right)^{\widetilde{q}}<b^{\widetilde{q}}$ or $\left(\frac{\lambda}{3}\right)^{\gamma}<c^{\gamma}$},
\end{equation} where $a,b$ and $c$ are nonnegative constants. Applying this simple observation to \eqref{cz non pa : meaoffest}, we get \eqref{cz non pa : altres}. Thus it remains to obtain a suitable upper bound on the second term in the right-hand side of \eqref{cz non pa : altres}. To this end, we find a suitable diagonal cylinder which was chosen in step 1 and contains $\mathcal{Q}_{5^{\frac1s}\widetilde{r}_{j}}(x_{d,j},t_{0,j})$, and then we use some combinatorial arguments by taking advantage of the factor $\mathfrak{A}\left(\mathcal{B}_{\overline{r}_{j}}(x_{1,j},x_{2,j})\right)$ (see step 7 for more details). As a result, we obtain a good upper bound on the measure of $\mathcal{Q}_{\widetilde{r}_{j}}(x_{1,j},x_{2,j},t_{0,j})\in \widetilde{\mathcal{A}}$.
\end{rmk}

\noindent
\textbf{Step 3. First elimination of off-diagonal cylinders.}
We now prove that if $ \mathcal{Q}_{\widetilde{r}_{j}}(x_{1,j},x_{2,j},t_{0,j})\in\widetilde{\mathcal{A}}$ satisfies
\begin{equation}
\label{cz non pa : lem.cov.step3.dist}
    \mathrm{dist}\left(B_{5^{\frac{1}{s}}\widetilde{r}_{j}}(x_{1,j}),B_{5^{\frac{1}{s}}\widetilde{r}_{j}}(x_{2,j})\right)<5^{\frac{1}{s}}\widetilde{r}_{j},
\end{equation}
then 
\begin{equation}
\label{cz non pa : distrel}
    \mathcal{Q}_{5^{\frac{1}{s}}\widetilde{r}_{j}}(x_{1,j},x_{2,j},t_{0,j})\subset \bigcup\limits_{i}\mathcal{Q}_{5^{\frac{1}{s}}
\rho_{{i}}}\left(z_{i}\right).
\end{equation}
By \eqref{cz non pa : czdereof}, one of the followings must hold:
{\small\begin{equation}
\label{cz non pa : alter}
\begin{aligned}
    &\left(\dashint_{\mathcal{Q}_{\overline{r}_{j}}(x_{1,j},x_{2,j},t_{0,j})}|D^{\tau}d_{s}u|^{p}\dmutt\right)^{\frac{1}{p}}>\frac{\lambda}{3},\quad E_{p,\tau}\left(u\,;\,{Q}_{\overline{r}_{j}}(x_{1,i},t_{0,i})\right)>\frac{\lambda}{3},\\
    &E_{p,\tau}\left(u\,;\,{Q}_{\overline{r}_{j}}(x_{2,i},t_{0,i})\right)>\frac{\lambda}{3}.
\end{aligned}
\end{equation}}\\
\noindent
Suppose that the first inequality in \eqref{cz non pa : alter} holds. We now observe 
\begin{equation*}
   \mathcal{B}_{5^{\frac{1}{s}}\widetilde{r}_{j}}(x_{1,j},x_{2,j}) \subset \mathcal{B}_{5^{\frac{2}{s}}\widetilde{r}_{j}}(x_{1,j}).
\end{equation*}
Therefore, using \eqref{cz non pa : inclusion measure}, \eqref{cz non pa : lambdacond} and \eqref{cz non pa : rangeofkappa1}, we obtain
{\small\begin{equation*}
\begin{aligned}
    \left(\dashint_{\mathcal{Q}_{5^{\frac{2}{s}}\widetilde{r}_{j}}(x_{1,j},t_{0,j})}|D^{\tau}d_{s}u|^{p}\dmutt\right)^{\frac{1}{p}}\geq \left(\frac{\mu_{\tau,t}\left(\mathcal{Q}_{\overline{r}_{j}}(x_{1,j},x_{2,j},t_{0,j})\right)}{\mu_{\tau,t}\left({\mathcal{Q}_{5^{\frac{2}{s}}\widetilde{r}_{j}}(x_{1,j},t_{0,j})}\right)}\dashint_{\mathcal{Q}_{\overline{r}_{j}}(x_{1,j},x_{2,j},t_{0,j})}|D^{\tau}d_{s}u|^{p}\dmutt\right)^{\frac{1}{p}}\geq \kappa\lambda,
\end{aligned}
\end{equation*}}\\
\noindent
which implies that $\Theta_{D}\left((x_{1,j},t_{0,j}),5^{\frac{2}{s}}\widetilde{r}_{j}\right)>\kappa\lambda$. By the fact that $5^{\frac2s}\widetilde{r}_j\leq \mathcal{R}_{1,2}$ and \eqref{cz non pa : exitradius}, \eqref{cz non pa : disjointcydi} yields 
\[\mathcal{Q}_{5^{\frac{2}{s}}\widetilde{r}_{j}}(x_{1,j},t_{0,j})\subset  \bigcup\limits_{i}\mathcal{Q}_{5^{\frac{1}{s}}
\rho_{{i}}}\left(z_{i}\right).\]
We next assume that the second inequality in \eqref{cz non pa : alter} is true. By \eqref{cz non pa : rangeofkappa1}, we have $\Theta_{D}\left((x_{1,j},t_{0,j}),5^{\frac{2}{s}}\widetilde{r}_{j}\right)>\kappa\lambda$. Since $\mathcal{Q}_{5^{\frac{1}{s}}\widetilde{r}_{j}}(x_{1,j},x_{2,j},t_{0,j})\subset \mathcal{Q}_{5^{\frac{2}{s}}\widetilde{r}_{j}}(x_{1,j},t_{0,j})$ which follows by \eqref{cz non pa : lem.cov.step3.dist}, we have  
\[\mathcal{Q}_{5^{\frac{1}{s}}\widetilde{r}_{j}}(x_{1,j},x_{2,j},t_{0,j})\subset\mathcal{Q}_{5^{\frac{2}{s}}\widetilde{r}_{j}}(x_{1,j},t_{0,j})\subset  \bigcup\limits_{i\geq0}\mathcal{Q}_{5^{\frac{1}{s}}
\rho_{{i}}}\left(z_{i}\right).\]
Similarly, we get \eqref{cz non pa : distrel} if the third inequality in \eqref{cz non pa : alter} holds. 
Thus, we now focus on the following subfamily of $\widetilde{\mathcal{A}}$:
\begin{equation}
\label{cz non pa : measure of tildea}
    \mathcal{A}=
\left\{ \mathcal{Q}_{\overline{r}_{j}}(x_{1,j},x_{2,j},t_{0,j})\ \middle\vert \begin{array}{l}
\mathcal{Q}_{\widetilde{r}_{j}}(x_{1,j},x_{2,j},t_{0,j})\in\widetilde{\mathcal{A}}\text{ and } \\ \mathcal{Q}_{5^{\frac{1}{s}}\widetilde{r}_{j}}(x_{1,j},x_{2,j},t_{0,j})\not\subset \bigcup\limits_{i\geq0}\mathcal{Q}_{5^{\frac{2}{s}}
\rho_{{i}}}\left(z_{i}\right)\text{ for each }i
  \end{array}\right\}
\end{equation}
Indeed, we take cylinders $\mathcal{Q}_{5^{\frac{2}{s}}
\rho_{{i}}}\left(z_{i}\right)$ instead of the cylinder $\mathcal{Q}_{5^{\frac{1}{s}}
\rho_{{i}}}\left(z_{i}\right)$, in order to eliminate other types of nearly diagonal cylinders (see \eqref{cz non pa : lem.cov.step5.dist} below).


\noindent
\textbf{Step 4. Off-diagonal estimates.} We now obtain a bound on $L^{p_{\#}}$-norm of $D^{\tau}d_{s}u$ and a reverse H\"older's inequality on cylinders which are far from the diagonal.
\begin{lem}
\label{cz non pa : off rever lem}
Let $\mathcal{Q}=\mathcal{Q}_{r}(x_{1},x_{2},t_{0})\subset \mathcal{Q}_{2}$ be such that
\begin{equation}
\label{cz non pa : lem.cov.step4.rhsdist}
\mathrm{dist}(B_{r}(x_{1}),B_{r}(x_{2}))\geq r.
\end{equation}
Then we have
\begin{equation*}
\begin{aligned}
    \left(\dashint_{ \mathcal{Q}}|D^{\tau}d_{s}u|^{p_{\#}}\dmutt\right)^{\frac{1}{p_{\#}}}&\leq c\,\Theta_{OD}\left(u;\mathcal{Q}\right)
\end{aligned}
\end{equation*}
for some constant $c=c(n,s,p,\tau)$.
\end{lem}


\begin{proof}
We observe from \eqref{cz non pa : lem.cov.step4.rhsdist} that
\begin{equation}
\label{cz non pa : xydist}
{\dist\left(B_{r}(x_{1}),B_{r}(x_{2})\right)}\leq |x-y|\leq 5\dist\left(B_{r}(x_{1}),B_{r}(x_{2})\right)
\end{equation}
whenever $x\in B_{r}(x_{1})$ and $y\in B_{r}(x_{2})$. The above inequality and Jensen's inequality yield that 
\begin{equation}
\label{cz non pa : measure of qintau}
    \frac{{r}^{2n}}{c(n)\dist\left(B_{r}(x_{1}),B_{r}(x_{2})\right)^{n-2\tau}}\leq\mu_{\tau}\left(\mathcal{B}_{r}(x_{1},x_{2})\right)\leq {\frac{c(n){r}^{2n}}{\dist\left(B_{r}(x_{1}),B_{r}(x_{2})\right)^{n-2\tau}}}
\end{equation}
and 
{\small\begin{equation*}
\begin{aligned}
&\dashint_{\mathcal{Q}_{{r}}(x_{1},x_{2},t_{0})}|D^{\tau}d_{s}u|^{p_{\#}}\dmutt\\
&\leq \frac{c}{\dist\left(B_{r}(x_{1}),B_{r}(x_{2})\right)^{p_{\#}\left(s+\tau\right)}}\dashint_{Q_{{r}}(x_{1},t_{0})}\dashint_{B_{r}(x_{1})}|u(x,t)-u(y,t)|^{p_{\#}}\dz\dt\\
&\leq \frac{c{r}^{p_{\#}(s+\tau)}}{\dist\left(B_{r}(x_{1}),B_{r}(x_{2})\right)^{p_{\#}\left(s+\tau\right)}}\sum_{d=1}^{2}\underbrace{\frac{1}{{r}^{p_{\#}(s+\tau)}}\dashint_{Q_{{r}}(x_{d},t_{0})}|u(x,t)-(u)_{B_{{r}}(x_{d})}(t)|^{p_{\#}}\dz}_{=I_{d}}\\
&\quad+\frac{c{r}^{p_{\#}(s+\tau)}}{\dist\left(B_{r}(x_{1}),B_{r}(x_{2})\right)^{p_{\#}\left(s+\tau\right)}}\underbrace{\dashint_{\Lambda_{{r}}(t_{0})}\frac{|(u)_{B_{r}(x_{1})}(t)-(u)_{B_{r}(x_{2})}(t)|^{p_{\#}}}{{r}^{p_{\#}(s+\tau)}}\dt}_{= J}
\end{aligned}
\end{equation*}}\\
\noindent
for some constant $c=c(n,s,p)$. We now further estimate $I_{1}$, $I_{2}$ and $J$ as below.

\noindent
\textbf{Estimates of $I_{1}$ and $I_{2}$.}
Using \eqref{cz non pa : lem.cov.step4.rhsdist} and Lemma \ref{cz non pa : embed lem}, we estimate $I_{d}$ as
{\small\begin{equation*}
\begin{aligned}
    I_{d}&\leq  c\left(\frac{1}{\tau}\dashint_{\mathcal{Q}_{{r}}(x_{d},t_{0})}|D^{\tau}d_{s}u|^{p}\dmutt\right)  \left(\sup_{t\in \Lambda_{{r}}(t_{0})}\dashint_{B_{{r}}(x_{d})}\frac{\left|u-(u)_{Q_{{r}}(x_{d},t_{0})}\right|^{2}}{{r}^{2s+2\tau}}\dx\right)^{\frac{sp}{n}},
\end{aligned}
\end{equation*}}\\
\noindent
where we have used the fact that
\begin{equation*}
    \sup_{t\in \Lambda_{{r}}(t_{0})}\dashint_{B_{{r}}(x_{d})}\frac{\left|u-(u)_{B_{{r}}(x_{d})}(t)\right|^{2}}{{r}^{2s+2\tau}}\dx\leq \sup_{t\in \Lambda_{{r}}(t_{0})}\dashint_{B_{{r}}(x_{d})}\frac{\left|u-(u)_{Q_{{r}}(x_{d},t_{0})}\right|^{2}}{{r}^{2s+2\tau}}\dx.
\end{equation*}
Applying Young's inequality to the above inequality, we have
\begin{equation*}
\frac{{r}^{p_{\#}(s+\tau)}}{\dist\left(B_{r}(x_{1}),B_{r}(x_{2})\right)^{p_{\#}\left(s+\tau\right)}}\sum_{d=1}^{2}I_{d}\leq \frac{c{r}^{p_{\#}(s+\tau)}}{\dist\left(B_{r}(x_{1}),B_{r}(x_{2})\right)^{p_{\#}\left(s+\tau\right)}}\left[\sum_{d=1}^{2}E_{p,\tau}\left(u\,;\,Q_{r}(x_{d},t_{0})\right)\right]^{p_{\#}}
\end{equation*}
for some constant $c=c(\tau)$.

\textbf{Estimate of $J$.}
We first note that
\begin{equation*}
\begin{aligned}
J\leq c\sum_{d=1}^{2}\sup_{t\in \Lambda_{{r}}(t_{0})}\left(\dashint_{B_{{r}}(x_{d})}\frac{|u-(u)_{Q_{{r}}(x_{d},t_{0})}|^{2}}{{r}^{2s+2\tau}}\dx\right)^{\frac{p_{\#}}{2}}+\frac{|(u)_{Q_{{r}}(x_{1},t_{0})}-(u)_{Q_{{r}}(x_{2},t_{0})}|^{p_{\#}}}{{r}^{p_{\#}(s+\tau)}},
\end{aligned}
\end{equation*}
where we denote the last term by $\hat{J}$.
In light of Jensen's inequality, \eqref{cz non pa : xydist} and \eqref{cz non pa : measure of qintau}, we estimate $\hat{J}$ as
{\small\begin{equation*}
\begin{aligned}
\hat{J}\leq \left(\dashint_{Q_{r}(x_{1},t_{0})}\dashint_{B_{r}(x_{2})}|u(x,t)-u(y,t)|^{p}\dy\dz\right)^{\frac{p_{\#}}{p}}\leq \frac{c\dist\left(B_{r}(x_{1}),B_{r}(x_{2})\right)^{p_{\#}(s+\tau)}}{{r}^{p_{\#}(s+\tau)}}\left(\dashint_{\mathcal{Q}}|D^{\tau}d_{s}u|^{p}\dmutt\right)^{\frac{p_{\#}}{p}}.
\end{aligned}
\end{equation*}}\\
\noindent
We finally combine all the estimates $I_{1}$, $I_{2}$ and $J$ to get the desired result \eqref{cz non pa : reverseestimate}.
\end{proof}
\begin{rmk}
Let $ \mathcal{Q}_{\overline{r}_{j}}(x_{1,j},x_{2,j},t_{0,j})\in\mathcal{A}$. By \eqref{cz non pa : lem.cov.step3.dist} and \eqref{cz non pa : measure of tildea}, we deduce $\mathcal{Q}_{5^{\frac{1}{s}}\widetilde{r}_{j}}(x_{1,j},x_{2,j},t_{0,j})$ satisfies \eqref{cz non pa : lem.cov.step4.rhsdist}. 
In light of \eqref{cz non pa : czdereof} and Lemma \ref{cz non pa : off rever lem}, there is a constant $c_{od}=c_{od}(n,s,p,\tau)$ such that
\begin{equation}
\label{cz non pa : reverseestimate}
\begin{aligned}
    \left(\dashint_{ \mathcal{Q}_{5^{\frac{1}{s}}\widetilde{r}_{j}}(x_{1,j},x_{2,j},t_{0,j})}|D^{\tau}d_{s}u|^{p_{\#}}\dmutt\right)^{\frac{1}{p_{\#}}}
    &\leq c_{od}\lambda.
\end{aligned}
\end{equation}
\end{rmk}



\noindent
\textbf{Step 5. Decomposition of the family $\mathcal{A}$.}
We now decompose the family $\mathcal{A}$ into subfamilies $AD_{\lambda}=\bigcap\limits_{d=1}^{2}AD_{\lambda}^{d}$ and $ND_{\lambda}=\bigcup\limits_{d=1}^{2}ND_{\lambda}^{d}$ for
\begin{equation}
\label{cz non pa : almost diagonal cubes2}
    AD_{\lambda}^{d}=\left\{\mathcal{Q}=B^{1}\times B^{2}\times I\in \mathcal{A}\,: \quad E_{p,\tau}\left(u\,;\,{B}^{d}\times I\right)\leq\frac{\lambda}{16}\right\}
\end{equation}
and
\begin{equation*}
    ND_{\lambda}^{d}=\left\{\mathcal{Q}=B^{1}\times B^{2}\times I\in \mathcal{A}\,: \quad E_{p,\tau}\left(u\,;\,{B}^{d}\times I\right)>\frac{\lambda}{16}\right\}.
\end{equation*}
Recalling \eqref{cz non pa : rangeofkappa1}, \eqref{cz non pa : peneftnlofu} and \eqref{cz non pa : exitradius}, we observe that for any $\mathcal{Q}_{\overline{r}_{j}}(x_{1,j},x_{2,j},t_{0,j})\in ND_{\lambda}^{d}$, there is an exit-radius $\rho_{(x_{d,j},t_{0,j})}\geq \overline{r}_{j}$ such that \eqref{cz non pa : exitradius} holds with $z_{0}=(x_{d,j},t_{0,j})$.
Thus, there is a cylinder ${Q}_{\rho_{i}}(z_{i})$ which is selected in \eqref{cz non pa : disjointcydi} such that
\begin{equation}
\label{cz non pa : exitoffim}
    {Q}_{2^{j_{0}}\overline{r}_{j}}(x_{d,j},t_{0,j})={Q}_{\widetilde{r}_{j}}(x_{d,j},t_{0,j})\subset {Q}_{2^{j_{0}}\rho_{(x_{d},t_{0})}}(x_{d,j},t_{0,j})\subset{Q}_{5^{\frac{1}{s}}\rho_{i}}(z_{i})\quad\text{and}\quad \frac{\rho_{i}}{2}\leq 2^{j_{0}}\rho_{(x_{d},t_{0})}\leq2\rho_{i}.
\end{equation} 
We have used \eqref{cz non pa : lem.cov.step1.vidist} to obtain the second observation in \eqref{cz non pa : exitoffim}. Therefore,  the set
{\small\begin{equation}
\label{cz non pa : lem.cov.step4.aijld}
\mathcal{A}^{d}_{i,j,l}\coloneqq
\left\{ \mathcal{Q}_{\overline{r}_{j}}(x_{1,j},x_{2,j},t_{0,j})\in ND_{\lambda}^{d}\ : \begin{array}{l}
{Q}_{\widetilde{r}_{j}}({x}_{d,j},t_{0,j})\subset {Q}_{5^{\frac{1}{s}}\rho_{i}}(z_{i})\text{, where $i$ is the smallest integer}\\ \text{satisfying }\eqref{cz non pa : exitoffim}\text{ with }\rho_{i}\leq 2^{l} \overline{r}_{j}<2\rho_{i}
  \end{array}\right\}
\end{equation}}\\
\noindent
is either a singleton or an empty set for any $l\geq0$. Then we will verify 
\begin{equation}
\label{cz non pa : lem.cov.step5.dist}
    \rho_{i}<\mathrm{dist}(B_{\overline{r}_{j}}(x_{1,j}),B_{\overline{r}_{j}}(x_{2,j})).
\end{equation}
Suppose not, then we have 
\begin{equation*}
    |x_{d',j}-x_{i}|\leq |x_{d',j}-x_{d,j}|+|x_{d,j}-x_{i}|\leq 5\rho_{i}+5^{\frac{1}{s}}\rho_{i}<2\times 5^{\frac{1}{s}}\rho_{i},
\end{equation*}
where $d'\in\{1,2\}\setminus\{d\}$. Consequently, we have
\begin{equation*}
    \mathcal{Q}_{\overline{r}_{j}}(x_{1,j},x_{2,j},t_{0,j})\subset \mathcal{Q}_{5^{\frac{2}{s}}\rho_{i}}(z_{i}),
\end{equation*}
which contradicts the definition of $\mathcal{A}$ defined in \eqref{cz non pa : measure of tildea}, and the claim follows.

We next define for any $k\geq0$,
\begin{equation}
\label{cz non pa : lem.cov.step4.aijlkd}
    \mathcal{A}^{d}_{i,j,l,k}=\left\{\mathcal{Q}_{\overline{r}_{j}}(x_{1,j},x_{2,j},t_{0,j})\in \mathcal{A}^{d}_{i,j,l}\, :\, 2^{k}\rho_{i}\leq\mathrm{dist}(B_{\overline{r}_{j}}(x_{1,j}),B_{\overline{r}_{j}}(x_{2,j}))<2^{k+1}\rho_{i}\right\}
\end{equation}
to see that
\begin{equation}
\label{cz non pa : lem.cov.step5.unindl}
    ND^{d}_{\lambda}=\bigcup_{i}\bigcup_{j}\bigcup_{l\geq0,k\geq0}\mathcal{A}^{d}_{i,j,l,k}.
\end{equation} 
We first observe that if $\mathcal{Q}_{\overline{r}_{j}}(x_{1,j},x_{2,j},t_{0,j})\in ND_{\lambda}^d$, then we have 
\begin{equation}\label{cz non pa : impe}
\begin{aligned}
&\mathfrak{A}\left(\mathcal{B}_{\overline{r}_{j}}(x_{1,j},x_{2,j})\right)^{s+\tau}E_p(u;\mathcal{Q}_{\overline{r}_{j}}(x_{d,j},t_{0,j})\\
&\leq c{\mathfrak{A}\left(\mathcal{B}_{\overline{r}_{j}}(x_{1,j},x_{2,j})\right) }\left(\dashint_{\mathcal{Q}_{\widetilde{r}_{j}}(x_{d,j},t_{0,j})}|D^{\tau}d_{s}u|^{p}\dmutt\right)^{\frac{1}{p}}\\
&\quad+\frac{c\mathfrak{A}\left(\mathcal{B}_{\overline{r}_{j}}(x_{1,j},x_{2,j})\right)}{\delta}\left(\dashint_{\mathcal{Q}_{\widetilde{r}_{j}}(x_{d,j},t_{0,j})}|D^{\tau}d_{0}f|^{\tilde{q}}\dmutt\right)^{\frac{1}{\tilde{q}}}\\
&\quad+\frac{c\mathfrak{A}\left(\mathcal{B}_{\overline{r}_{j}}(x_{1,j},x_{2,j})\right)}{\delta}\left(\dashint_{\mathcal{Q}_{\widetilde{r}_{j}}(x_{d,j},t_{0,j})}(\widetilde{r}_j^{s-\tau}|G|)^{\gamma}\dmutt\right)^{\frac{1}{\gamma}}+\frac{\lambda}{100}
\end{aligned}
\end{equation}
for some constant $c=c(n,s,L,q,\tau)$, where the constant $\widetilde{r}_j$ is determined in \eqref{defn.tilder}.  By following the same lines as in the proof of \eqref{cz non pa : oneretildeqnec}, \eqref{cz non pa : j0cond} along with \eqref{cz non pa : czdereof} and the fact that 
\begin{align*}
\mathfrak{A}\left(\mathcal{B}_{\overline{r}_{j}}(x_{1,j},x_{2,j})\right)\leq \mathfrak{A}\left(\mathcal{B}_{c\overline{r}_{j}}(x_{1,j},x_{2,j})\right)
\end{align*} 
for any $c\geq1$, 
we get 
\begin{align*}
&\mathfrak{A}\left(\mathcal{B}_{\overline{r}_{j}}(x_{1,j},x_{2,j}\right)\tau^{\frac1{\gamma}}\left(\sup_{t\in \Lambda_{r}(t_{0})}\dashint_{B_{r}(x_{0})}\frac{|u-(u)_{Q_{r}(z_{0})}|^{2}}{r^{2s+2\tau}}\dx\right)^{\frac{1}{2}} \\
&\leq c{\mathfrak{A}\left(\mathcal{B}_{\overline{r}_{j}}(x_{1,j},x_{2,j})\right) }\left(\dashint_{\mathcal{Q}_{2^{j_0}\overline{r}_{j}}(x_{d,j},t_{0,j})}|D^{\tau}d_{s}u|^{p}\dmutt\right)^{\frac{1}{p}}\\
&\quad+\frac{c\mathfrak{A}\left(\mathcal{B}_{\overline{r}_{j}}(x_{1,j},x_{2,j})\right)}{\delta}\left(\dashint_{\mathcal{Q}_{2^{j_0}\overline{r}_{j}}(x_{d,j},t_{0,j})}|D^{\tau}d_{0}f|^{\tilde{q}}\dmutt\right)^{\frac{1}{\tilde{q}}}\\
&\quad+\frac{c\mathfrak{A}\left(\mathcal{B}_{\overline{r}_{j}}(x_{1,j},x_{2,j})\right)}{\delta}\left(\dashint_{\mathcal{Q}_{2\overline{r}_{j}}(x_{d,j},t_{0,j})}(r_j^{s-\tau}|G|)^{\gamma}\dmutt\right)^{\frac{1}{\gamma}}+\frac{\lambda}{100}
\end{align*}
for some constant $c=c(n,s,L,q,\tau)$. Thus, using \eqref{cz non pa : peneftnlofu}, \eqref{cz non pa : j0cond} and \eqref{defn.tilder}, we get the desired estimate.

We end this step with the following lemma which is an essential ingredient for the next step. 
\begin{lem}
\label{cz non pa : dist of ndl}
Let us fix $i,l,k\geq0$ and $d\in\{1,2\}$. Then there is a constant $c=c(n,s,L,q,\tau)$ such that
\begin{equation}
\label{cz non pa : rmk.mainest}
    \sum_{\mathcal{Q}_{\overline{r}_{j}}(x_{1,j},x_{2,j},t_{0,j})\in \bigcup\limits_{\tilde{j}\geq0}\mathcal{A}^{1}_{i,\tilde{j},l,k}}\int_{\mathcal{Q}_{\widetilde{r}_{j}}(x_{1,j},t_{0,j})}|D^{\tau}d_{s}u|^{p}\dmutt\leq c2^{n(l+k)}\int_{\mathcal{Q}_{5^{\frac{1}{s}}\rho_{i}}(z_{i})}|D^{\tau}d_{s}u|^{p}\dmutt,
\end{equation}
\begin{equation}
\label{cz non pa : rmk.mainest1}
    \sum_{\mathcal{Q}_{\overline{r}_{j}}(x_{1,j},x_{2,j},t_{0,j})\in \bigcup\limits_{\tilde{j}\geq0}\mathcal{A}^{1}_{i,\tilde{j},l,k}}\int_{\mathcal{Q}_{\widetilde{r}_{j}}(x_{1,j},t_{0,j})}|D^{\tau}d_{0}f|^{\tilde{q}}\dmutt\leq c2^{n(l+k)}\int_{\mathcal{Q}_{5^{\frac{1}{s}}\rho_{i}}(z_{i})}|D^{\tau}d_{0}f|^{\tilde{q}}\dmutt
\end{equation}
and
\begin{equation}
\label{cz non pa : rmk.mainest2}
    \sum_{\mathcal{Q}_{\overline{r}_{j}}(x_{1,j},x_{2,j},t_{0,j})\in \bigcup\limits_{\tilde{j}\geq0}\mathcal{A}^{1}_{i,\tilde{j},l,k}}\int_{\mathcal{Q}_{\widetilde{r}_{j}}(x_{1,j},t_{0,j})}(a\overline{r}_j^{s-\tau}|G|)^{\gamma}\dmutt\leq c2^{n(l+k)}\int_{\mathcal{Q}_{5^{\frac{1}{s}}\rho_{i}}(z_{i})}(\rho_i^{s-\tau}|G|)^{\gamma}\dmutt
\end{equation}
\end{lem}
\begin{proof}
We first observe that if 
\begin{equation}
\label{cz non pa : rmk.nonempty}
    \bigcap_{\substack{j\in J\\\mathcal{Q}_{\overline{r}_{j}}(x_{1,j},x_{2,j},t_{0,j})\in \bigcup\limits_{\tilde{j}\geq0}\mathcal{A}^{1}_{i,\tilde{j},l,k}}}\mathcal{Q}_{\widetilde{r}_{j}}(x_{1,j},t_{0,j})\neq\emptyset
\end{equation}
holds for some index set $J$, then $|J|\leq c2^{n(l+k)}$ for some constant $c$ depending only on $n$ and $a$, where $|J|$ denotes the number of elements in the set $J$. Next, suppose that $(x_{1},t_{0})\in \bigcap_{j\in J}{Q}_{\widetilde{r}_{j}}(x_{1,j},t_{0,j})$. Then, by the definition of the set $\mathcal{A}^{1}_{i,j,l,k}$ given in \eqref{cz non pa : lem.cov.step4.aijlkd}, we get 
\begin{equation}
\label{cz non pa : rmk.combi}
    \mathrm{dist}(x_{1},B_{\widetilde{r}_{j}}(x_{2,j}))<2^{k+1}\rho_{i}+2\widetilde{r}_{j}<c2^{k}\rho_{i}
\end{equation}
for some constant $c=c(n,s,L,q,\tau)$, where we have used the fact that $\overline{r}_{j}\leq 2^{1-l}\rho_{i}$ by \eqref{cz non pa : lem.cov.step4.aijld}, \eqref{defn.tilder} and \eqref{cz non pa : j0cond}. Since the set $\left\{\mathcal{Q}_{\widetilde{r}_{j}}(x_{1,j},x_{2,j},t_{0,j})\right\}_{j\in J}$ is a mutually disjoint, we note from \eqref{cz non pa : rmk.nonempty} that $\left\{B_{\widetilde{r}_{j}}(x_{2,j})\right\}_{j\in J}$ is also a mutually disjoint set. This along with \eqref{cz non pa : rmk.combi} implies
\begin{equation*}
    |J|\leq \frac{\left|B_{c2^{k}\rho_{i}}\right|}{\left|B_{\widetilde{r}_{j}}\right|}\leq c2^{n(k+l)}
\end{equation*}
for some constant $c=c(n,s,L,q,\tau)$, where we have used \eqref{defn.tilder}, \eqref{cz non pa : j0cond} and the relation between $\overline{r}_{j}$ and $\rho_{i}$ given in \eqref{cz non pa : lem.cov.step4.aijld}. This proves \eqref{cz non pa : rmk.nonempty}.
 We are now ready to prove \eqref{cz non pa : rmk.mainest} using an inductive argument. We first note that the number of elements in $\bigcup\limits_{j}\mathcal{A}^{1}_{i,j,l,k}$ is finite, as each cylinder in $\bigcup\limits_{j}\mathcal{A}^{1}_{i,j,l,k}$ is of the form $\mathcal{Q}=\mathcal{Q}_{r}$, where $2^{-l}\rho_{i}<r\leq 2^{-l
+1}\rho_{i}$, is mutually disjoint and is contained in $\mathcal{Q}_{r_{2}}$. Let us denote $j_{i}$ as the number of elements in $\bigcup\limits_{j}\mathcal{A}^{1}_{i,j,l,k}$. For a clear notation, we assume $\mathcal{A}^{1}_{i,j,l,k}\neq\emptyset$ if $j\leq j_{i}$ and $\mathcal{A}^{1}_{i,j,l,k}=\emptyset$ if $j\geq j_{i}+1$. By \eqref{cz non pa : exitoffim}, we note $\mathcal{Q}_{\widetilde{r}_{j}}(x_{1,j},t_{0,j})\subset Q_{5^{\frac{1}{s}}\rho_{i}}(z_{i})$. We first define 
\begin{equation*}
    \mathcal{D}_{1}=\left\{\mathcal{Q}_{5^{\frac{1}{s}}\rho_{i}}(z_{i})\cap\mathcal{Q}_{\widetilde{r}_{j}}(x_{1,j},t_{0,1}),\,\mathcal{Q}_{5^{\frac{1}{s}}\rho_{i}}(z_{i})\setminus \mathcal{Q}_{\widetilde{r}_{j}}(x_{1,1},t_{0,1}),\right\}.
\end{equation*}
Suppose $\mathcal{D}_{k}$ is determined, where $k\geq1$, then we define
\begin{equation*}
    \mathcal{D}_{k+1}=\bigcup_{\mathcal{F}\in\mathcal{D}_{k}}\left\{\mathcal{F}\cap\mathcal{Q}_{\widetilde{r}_{k+1}}(x_{1,k+1},t_{0,k+1}),\, \mathcal{F}\setminus \mathcal{Q}_{\widetilde{r}_{k+1}}(x_{1,k+1},t_{0,k+1})\right\}.
\end{equation*}
In this way, we obtain a collection $\mathcal{D}_{j_{i}}$. Indeed, for any choice of two elements $\mathcal{F}$ and $\mathcal{F}'$ in $\mathcal{D}_{j_{i}}$, we see that either $\mathcal{F}=\mathcal{F}'$ or $\mathcal{F}\cap\mathcal{F}'=\emptyset$. Thus we can write 
\begin{equation}
\label{cz non pa : lem.cov.step5.fm}
    \mathcal{Q}_{5^{\frac{1}{s}}\rho_{i}}(z_{i})=\bigcup_{\mathcal{F}_{m}\in\mathcal{D}_{j_{i}}}{\mathcal{F}_{m}},
\end{equation}
where $\mathcal{F}_{m}$'s are mutually disjoint, and for any $m$ and $j\in[1,j_{i}]$, either 
\begin{equation}
\label{cz non pa : lem.cov.step5.fmuni}
    \mathcal{F}_{m}\subset \mathcal{Q}_{\widetilde{r}_{j}}(x_{1,j},t_{0,j})\quad\text{or}\quad \mathcal{F}_{m}\cap \mathcal{Q}_{\widetilde{r}_{j}}(x_{1,j},t_{0,j})=\emptyset.
\end{equation}
Moreover, there are mutually disjoint elements $\mathcal{F}_{m}\in \mathcal{D}_{j_{i}}$ such that
\begin{equation}
\label{cz non pa : lem.cov.step5.fmuni2}
    \mathcal{Q}_{\widetilde{r}_{j}}(x_{1,j},t_{0,j})=\bigcup_{\mathcal{F}_{m}\subset \mathcal{Q}_{\widetilde{r}_{j}}(x_{1,j},t_{0,j}) }\mathcal{F}_{m}.
\end{equation}
We note that for each $\mathcal{F}_{m}$, the number of elements in $J\coloneqq\{j\in [1,j_{i}]\,:\,\mathcal{F}_{m}\subset \mathcal{Q}_{\widetilde{r}_{j}}(x_{1,j},t_{0,j})\}$  is at most $c2^{n(k+l)}$, due to \eqref{cz non pa : rmk.nonempty}.
As a result, we have 
{\small\begin{equation*}
\begin{aligned}
    \sum_{\mathcal{Q}_{\overline{r}_{j}}(x_{1,j},x_{2,j},t_{0,j})\in \bigcup\limits_{\tilde{j}}\mathcal{A}^{1}_{i,\tilde{j},l,k}}\int_{\mathcal{Q}_{\widetilde{r}_{j}}(x_{1,j},t_{0,j})}|D^{\tau}d_{s}u|^{p}\dmutt&=\sum_{j=1}^{j_{i}}\sum_{\substack{m\\ \mathcal{F}_{m}\subset \mathcal{Q}_{\widetilde{r}_{j}}(x_{1,j},t_{0,j})}}\int_{\mathcal{F}_{m}}|D^{\tau}d_{s}u|^{p}\dmutt\\
    &=\sum_{m}\sum_{\substack{j=1\\\mathcal{F}_{m}\subset \mathcal{Q}_{\widetilde{r}_{j}}(x_{1,j},t_{0,j})}}^{j_{i}}\int_{\mathcal{F}_{m}}|D^{\tau}d_{s}u|^{p}\dmutt\\
    &\leq c2^{n(l+k)}\sum_{m}\int_{\mathcal{F}_{m}}|D^{\tau}d_{s}u|^{p}\dmutt\\
    &\leq c2^{n(l+k)}\int_{\mathcal{Q}_{5^{\frac{1}{s}}\rho_{i}}(z_{i})}|D^{\tau}d_{s}u|^{p}\dmutt
\end{aligned}
\end{equation*}}\\
\noindent
for some constant $c=c(n,s,L,q,\tau)$, which implies the desired result \eqref{cz non pa : rmk.mainest}. By following the same lines as in the proof of \eqref{cz non pa : rmk.mainest} with $(D^\tau d_s u,p)$ replaced by $(D^\tau d_0 f,\tilde{q})$ and $((\widetilde{r}_j^{s-\tau}|G|)^\gamma ,1)$, respectively and by using the fact that $\widetilde{r}_j\leq 2\rho_i$ which follows from \eqref{cz non pa : exitoffim}, we prove \eqref{cz non pa : rmk.mainest1} and \eqref{cz non pa : rmk.mainest2}.
\end{proof}


\noindent
\textbf{Step 6. Measure estimate of $\mathcal{Q}\in AD_{\lambda}$.} We claim that for every $\mathcal{Q}\in AD_{\lambda}$, there holds
\begin{equation}
\label{cz non pa : measqad}
\mu_{\tau,t}(\mathcal{Q})\leq \frac{2^{q}}{\lambda^{p}}\int_{\mathcal{Q}\cap\left\{|D^{\tau}d_{s}u|>\frac{\lambda}{16}\right\}}|D^{\tau}d_{s}u|^{p}\dmutt.
\end{equation}
Indeed, on account of \eqref{cz non pa : czdereof}, we have 
\begin{equation}
\label{cz non pa : nd app}
\begin{aligned}
    \lambda^{p}
    &\leq 2^{p}\left(\dashint_{\mathcal{Q}}|D^{\tau}d_{s}u|^{p}\dmutt\right)+2^{p}\left(\mathfrak{A}(\mathcal{K})\right)^{p}\left[\sum_{d=1}^{2}E_{p,\tau}\left(u\,;\,{B}^{d}\times I\right)\right]^{p}.
\end{aligned}
\end{equation}
Using \eqref{cz non pa : nd app}, \eqref{cz non pa : almost diagonal cubes2} and the fact that
\begin{equation*}
    \left(\dashint_{\mathcal{Q}}|D^{\tau}d_{s}u|^{p}\dmutt\right)\leq \frac{\lambda^{p}}{8^{p}}+\frac{1}{\mu_{\tau,t}(\mathcal{Q})}\int_{\mathcal{Q}\cap\left\{|D^{\tau}d_{s}u|>\frac{\lambda}{4}\right\}}|D^{\tau}d_{s}u|^{p}\dmutt,
\end{equation*}
we deduce that
\begin{equation*}
    \frac{\lambda^{p}}{2^{p}}\leq \frac{1}{\mu_{\tau,t}(\mathcal{Q})}\int_{\mathcal{Q}\cap\left\{|D^{\tau}d_{s}u|>\frac{\lambda}{16}\right\}}|D^{\tau}d_{s}u|^{p}\dmutt.
\end{equation*}
This proves the claim of \eqref{cz non pa : measqad} as $p\leq q$.

\noindent
\textbf{Step 7. Measure estimate of $\mathcal{Q}\in ND_{\lambda}$.}
Our next aim is to estimate $\mu_{\tau,t}(\mathcal{Q})$ for $\mathcal{Q}\in ND_{\lambda}$.

\noindent
With the aid of \eqref{cz non pa : measqad}, \eqref{cz non pa : lem.cov.step5.unindl} and Lemma \ref{cz non pa : dist of ndl}, we prove the following result.
\begin{lem}
There exists a constant $c=c(n,s,L,q,\tau)$ such that
\begin{equation}
\label{cz non pa : sumofmeasqnd}
\begin{aligned}
    \sum_{\mathcal{Q}\in ND_{\lambda}}\mu_{\tau,t}(\mathcal{Q})&\leq\frac{c}{\lambda^{p}}\int_{\mathcal{Q}_{r_{2}}\cap\{|D^{\tau}d_{s}u|>\frac{\lambda}{16}\}}|D^{\tau}d_{s}u|^{p}\dmutt+c\sum_{i}\mu_{\tau,t}\left(\mathcal{Q}_{\rho_{i}}(z_{i})\right).
\end{aligned}
\end{equation} 
\end{lem}
\begin{proof}
Let $\mathcal{Q}\equiv\mathcal{Q}_{r}(x_{1},x_{2},t_{0})\in ND_{\lambda}$. 
We note from \eqref{cz non pa : meaoffest}, \eqref{cz non pa : almost diagonal cubes2}, \eqref{defn.tilder} and \eqref{cz non pa : impe} that there is a constant $c=c(n,s,L,q,\tau)$ such that
\begin{align*}
    \frac{\lambda}{4}\leq \left(\dashint_{\mathcal{Q}}|D^{\tau}d_{s}u|^{p}\dmutt\right)^{\frac1p}+c\sum_{d=1}^2\left[\left(E^d_u(\mathcal{Q})\right)^{\frac{1}{p}}+\left(E^d_f(\mathcal{Q})\right)^{\frac{1}{\tilde{q}}}+\left(E^d_g(\mathcal{Q})\right)^{\frac{1}{\gamma}}\right],
\end{align*}
if $\mathcal{Q}\in ND^{1}_{\lambda}\cap ND^{2}_{\lambda}$ and 
\begin{align*}
    \frac{\lambda}4\leq \left(\dashint_{\mathcal{Q}}|D^{\tau}d_{s}u|^{p}\dmutt\right)^{\frac1p}+c\left[\left(E^d_u(\mathcal{Q})\right)^{\frac{1}{p}}+\left(E^d_f(\mathcal{Q})\right)^{\frac{1}{\tilde{q}}}+\left(E^d_g(\mathcal{Q})\right)^{\frac{1}{\gamma}}\right],
\end{align*}
if $\mathcal{Q}\in ND^{d}_{\lambda}\cap AD^{d'}_{\lambda}$,
where we denote $d'\in \{1,2\}\setminus \{d\}$.
where we denote 
\begin{equation}\label{cz non pa : Kdt1}
\begin{aligned}
    \mathbb{E}^{d}_{u}\left(\mathcal{Q}\right)=\mathfrak{A}\left(\mathcal{B}_{r}(x_{1},x_{2})\right)^{p\left(s+\tau\right)}\left(\dashint_{\mathcal{Q}_{2^{j_0}r}(x_{d},t_{0})}|D^{\tau}d_{s}u|^{p}\dmutt\right),
\end{aligned}
\end{equation}
\begin{equation}
\label{cz non pa : Kdt12}
\begin{aligned}
    \mathbb{E}^{d}_{f}\left(\mathcal{Q}\right)= \mathfrak{A}\left(\mathcal{B}_{r}(x_{1},x_{2})\right)^{
    \tilde{q}\left(s+\tau\right)} \left(\dashint_{2^{j_0}\mathcal{Q}_{r}(x_{d},t_{0})}|D^{\tau}d_{0}f|^{\tilde{q}}\dmutt\right),
\end{aligned}
\end{equation}
and
\begin{equation}
\label{cz non pa : Kdt123}
\begin{aligned}
    \mathbb{E}^{d}_{g}\left(\mathcal{Q}\right)= \mathfrak{A}\left(\mathcal{B}_{r}(x_{1},x_{2})\right)^{
    \gamma\left(s+\tau\right)} \left(\dashint_{\mathcal{Q}_{2^{j_0}r}(x_{d},t_{0})}((2^{j_0}r)^{s-\tau}|G|)^{\gamma}\dmutt\right)
\end{aligned}
\end{equation}
We note that the set function $\mathfrak{A}(\cdot)$ is defined in \eqref{cz non pa : cubecoeff}. Using \eqref{simple.cal} and then multiplying $\mu_{\tau,t}\left(\mathcal{Q}\right)$ along with a few simple calculations, we obtain 
{\begin{equation}
\label{cz non pa : altres}
\begin{aligned}
    \mu_{\tau,t}\left(\mathcal{Q}\right)&\leq \frac{c}{\lambda^{p}}\int_{{\mathcal{Q}}\cap \left\{|D^{\tau}d_{s}u|>\frac{\lambda}{16}\right\}}|D^{\tau}d_{s}u|^{p}\dmutt\\
    &\quad+  \frac{\mu_{\tau,t}(\mathcal{Q})}{\mu_{\tau,t}\left(\mathcal{Q}_{r}(x_{d},t_{0})\right)}\left[\frac{c}{\lambda^{p}}\sum_{d=1}^{2}\mathbb{E}^{d}_{u}\left(\mathcal{Q}\right)+\frac{c}{(\delta\lambda)^{\tilde{q}}}\sum_{d=1}^{2}\mathbb{E}^{d}_{f}\left(\mathcal{Q}\right)+\frac{c}{(\delta\lambda)^{b_\tau}}\sum_{d=1}^{2}\mathbb{E}^{d}_{g}\left(\mathcal{Q}\right)\right]
\end{aligned}
\end{equation}}
if $\mathcal{Q}\in ND^{1}_{\lambda}\cap ND^{2}_{\lambda}$ and
{\begin{equation}
\label{cz non pa : altres1}
\begin{aligned}
    \mu_{\tau,t}\left(\mathcal{Q}\right)&\leq \frac{c}{\lambda^{p}}\int_{\mathcal{Q}\cap \left\{|D^{\tau}d_{s}u|>\frac{\lambda}{16}\right\}}|D^{\tau}d_{s}u|^{p}\dmutt\\
    &\quad+  \frac{\mu_{\tau,t}(\mathcal{Q})}{\mu_{\tau,t}\left(\mathcal{Q}_{r}(x_{d},t_{0})\right)}\left[\frac{c}{\lambda^{p}}\mathbb{E}^{d}_{u}\left(\mathcal{Q}\right)+\frac{c}{(\delta\lambda)^{\tilde{q}}}\mathbb{E}^{d}_{f}\left(\mathcal{Q}\right)+\frac{c}{(\delta\lambda)^{b_\tau}}\mathbb{E}^{d}_{g}\left(\mathcal{Q}\right)\right]
\end{aligned}
\end{equation}}
if $\mathcal{Q}\in ND^{d}_{\lambda}\cap AD^{d'}_{\lambda}$. We next observe from \eqref{cz non pa : inclusion measure} and \eqref{cz non pa : measure of qintau} that 
\begin{equation}
\label{cz non pa : qqh}
   \frac{\mu_{\tau,t}(\mathcal{Q})}{\mu_{\tau,t}\left(\mathcal{Q}_{r}(x_{d},t_{0})\right)}\leq c\tau \left(\frac{r}{\dist(B_{r}(x_{1}),B_{r}(x_{2}))}\right)^{n-2\tau}
\end{equation}
for some constant $c=c(n)$.
On account of \eqref{cz non pa : lem.cov.step4.aijld}, \eqref{cz non pa : lem.cov.step4.aijlkd}, \eqref{cz non pa : lem.cov.step5.unindl}, \eqref{cz non pa : qqh} and \eqref{cz non pa : rmk.mainest}, we find that
{\small\begin{equation}
\label{cz non pa : higherdifft0}
\begin{aligned}
\sum_{\mathcal{Q}\in ND_{\lambda}^{d}}\frac{\mu_{\tau,t}(\mathcal{Q})}{\mu_{\tau,t}\left(\mathcal{Q}_{r}(x_{d},t_{0})\right)}\mathbb{E}^{d}_{u}(\mathcal{Q})&\leq c\sum_{i,j,k,l}\sum_{\mathcal{Q}\in\mathcal{A}^{d}_{i,j,l,k}} \left(\frac{r}{\dist(B_{r}(x_{1}),B_{r}(x_{2}))}\right)^{n-2\tau+ps+p\tau}\left(\int_{2^{j_0}P^{d}\mathcal{Q}}|D^{\tau}d_{s}u|^{p}\dmutt\right)\\
&\leq c\sum_{i,l,k\geq0} \left(2^{-(l+k)}\right)^{ps} \left(\int_{\mathcal{Q}_{5^{\frac{1}{s}}\rho_{i}}(z_{i})}|D^{\tau}d_{s}u|^{p}\dmutt\right)\\
&\leq c\sum_{i\geq0}\left(\int_{\mathcal{Q}_{5^{\frac{1}{s}}\rho_{i}}(z_{i})}|D^{\tau}d_{s}u|^{p}\dmutt\right),
\end{aligned}
\end{equation}}\\
\noindent
where we denote $P^{d}\mathcal{Q}=\mathcal{Q}_{r}(x_{d},t_{0})$ for a given cylinder $\mathcal{Q}=\mathcal{Q}_{r}(x_{1},x_{2},t_{0})$. For the last inequality, we have used the fact that
\begin{equation*}
    \sum_{i,j\geq0} \left(2^{-(i+j)}\right)^{a}\leq \frac{2^{a}}{a\mathrm{ln}2}\quad(a>0).
\end{equation*}
Therefore, using \eqref{cz non pa : exitradius}, we obtain 
\begin{equation*}
    \sum_{\mathcal{Q}\in ND_{\lambda}^{d}}\frac{\mu_{\tau,t}(\mathcal{Q})}{\mu_{\tau,t}\left(\mathcal{Q}_{r}(x_{d},t_{0})\right)}\mathbb{E}^{d}_u(\mathcal{Q})\leq c\lambda^{p}\sum_{i}\mu_{\tau,t}\left(\mathcal{Q}_{\rho_{i}}(z_{i})\right)
\end{equation*}
As in \eqref{cz non pa : higherdifft0} along with \eqref{cz non pa : rmk.mainest1}, \eqref{cz non pa : rmk.mainest2} and \eqref{cz non pa : exitradius}, we have
\begin{align*}
    \sum_{\mathcal{Q}\in ND_{\lambda}^{d}}\frac{\mu_{\tau,t}(\mathcal{Q})}{\mu_{\tau,t}\left(\mathcal{Q}_{r}(x_{d},t_{0})\right)}\mathbb{E}^{d}_{f}(\mathcal{Q})&\leq c\sum_{i,l,k\geq0}\left(2^{-(l+k)}\right)^{\tilde{q}s}\int_{\mathcal{Q}_{5^{\frac1{s}}\rho_i}(x_{d},t_{0})}|D^{\tau}d_{0}f|^{\tilde{q}}\dmutt\\
    &\leq c(\delta\lambda)^{\tilde{q}}\sum_{i}\mu_{\tau,t}\left(\mathcal{Q}_{\rho_{i}}(z_{i})\right)
\end{align*}
and
\begin{align*}
    \sum_{\mathcal{Q}\in ND_{\lambda}^{d}}\frac{\mu_{\tau,t}(\mathcal{Q})}{\mu_{\tau,t}\left(\mathcal{Q}_{r}(x_{d},t_{0})\right)}\mathbb{E}^{d}_{f}(\mathcal{Q})&\leq c\sum_{i,l,k\geq0}\left(2^{-(l+k)}\right)^{\gamma s}\int_{\mathcal{Q}_{5^{\frac1{s}}\rho_i}(x_{d},t_{0})}\left((\rho_i)^{s-\tau}|G|\right)^{\gamma}\dmutt\\
    &\leq c(\delta\lambda)^{\gamma}\sum_{i}\mu_{\tau,t}\left(\mathcal{Q}_{\rho_{i}}(z_{i})\right)
\end{align*}
for some constant $c=c(n,s,L,q,\tau)$.
Combining the above two inequalities, we get
\begin{equation}
\label{cz non pa : higherdifft}
\begin{aligned}
   \sum_{\mathcal{Q}\in ND_{\lambda}^{d}}\frac{\mu_{\tau,t}(\mathcal{Q})}{\mu_{\tau,t}\left(\mathcal{Q}_{r}(x_{d},t_{0})\right)}\left[\frac{1}{(\delta\lambda)^{\tilde{q}}}\mathbb{E}^{d}_{f}(\mathcal{Q}) +\frac{1}{(\delta\lambda)^{\gamma}}\mathbb{E}^{d}_{g}(\mathcal{Q})\right]\leq c\sum_{i}\mu_{\tau,t}\left(\mathcal{Q}_{\rho_{i}}(z_{i})\right),
\end{aligned}
\end{equation}
where we denote $\mathcal{Q}=B^{1}\times B^{2}\times \Lambda$.
Plugging \eqref{cz non pa : higherdifft0} and \eqref{cz non pa : higherdifft} into \eqref{cz non pa : altres} and \eqref{cz non pa : altres1}, we have
\begin{equation*}
\begin{aligned}
    \sum_{\mathcal{Q}\in ND_{\lambda}}\mu_{\tau,t}\left(\mathcal{Q}\right)&\leq \frac{c}{\lambda^{p}}\int_{{\mathcal{Q}_{r_{2}}}\cap \left\{|D^{\tau}d_{s}u|>\frac{\lambda}{16}\right\}}|D^{\tau}d_{s}u|^{p}\dmutt+c\sum_{i}\mu_{\tau,t}\left(\mathcal{Q}_{\rho_{i}}(z_{i})\right)
\end{aligned}
\end{equation*}
\end{proof}
\textbf{Step 8. Completion of the proof.}
Considering \eqref{cz non pa : condM1} and \eqref{cz non pa : rangeofkappa1}, the constant $\frac{M}{\kappa}$ depends only on $n,s,L,q$ and $\tau$.
Therefore, if $\lambda\geq\lambda_{0}$ which is determined in \eqref{cz non pa : lambdacond}, then we find two families of countable disjoint cylinders 
\begin{equation*}
    \left\{\mathcal{Q}_{\rho_{i}}(z_{i})\right\}_{i}\quad\text{and}\quad\left\{\mathcal{Q}_{\widetilde{r}_{j}}(x_{1,j},x_{2,j},t_{0,j}\right\}_{\mathcal{Q}_{\overline{r}_{j}}(x_{1,j},x_{2,j},t_{0,j})\in\mathcal{A}}
\end{equation*} 
so that 
\begin{align*}
    \{(x,y,t)\in\mathcal{Q}_{r_1}\,:\,|D^\tau d_su(x,y,t)|\geq\lambda\}\subset \bigcup_{i}\mathcal{Q}_{5^{\frac2s}\rho_{i}}(z_{i})\cup \bigcup_{\substack{j \\ \mathcal{Q}_{\overline{r}_{j}}(x_{1,j},x_{2,j},t_{0,j})\in\mathcal{A}} }\mathcal{Q}_{5^{\frac1s}\widetilde{r}_{j}}(x_{1,j},x_{2,j},t_{0,j}),
\end{align*} which follows from the steps 2-3 along with \eqref{cz non pa : lem.cov.step1.vidist},  \eqref{upperlevelsetoffb} and \eqref{cz non pa : measure of tildea}. In addition, using \eqref{cz non pa : reverseestimate}, \eqref{cz non pa : measqad}, \eqref{cz non pa : sumofmeasqnd} and \eqref{cz non pa : diab} along with the choice of the constants $a_{u}=\tilde{a}_{u}\kappa$, $a_{f}=\tilde{a}_{f}\kappa$ and $a_{g}=\tilde{a}_{g}\kappa$ given in \eqref{cz non pa : diab}, we get \eqref{cz non pa : pstra norm} and \eqref{cz non pa : diagonal estimate}. This completes the proof.
\end{proof}

\noindent
Before ending this section, we give some estimates which are useful in the context of the comparison Lemma \ref{cz non pa : nsccomp}.
\begin{rmk}
\label{cz non pa : diagonalpartlambda}
Let $\mathcal{Q}_{\rho_{i}}(z_{i})$ be the cylinder chosen in Lemma \ref{cz non pa : coveringL}. Then we want to show that
{\small\begin{equation}
\label{cz non pa : bcompassm1}
    \frac{1}{\tau}\dashint_{Q_{4\times 5^{\frac{2}{s}}\rho_{i}}(z_{i})}|D^{\tau}d_{s}u|^{2}\dmutt+\Tail_{\infty,2s}\left(\frac{u-(u)_{B_{4\times 5^{\frac{2}{s}}\rho_{i}}(x_{i})}(t)}{(4\times 5^{\frac{2}{s}}\rho_{i})^{s+\tau}}\,;\,Q_{4\times 5^{\frac{2}{s}}\rho_{i}}(z_{i})\right)^{2}\leq (c\lambda)^{2}
\end{equation}}\\
\noindent
and
{\small\begin{equation}
\label{cz non pa : bcompassm2}
\begin{aligned}
  \left(\frac{1}{\tau}\dashint_{Q_{4\times 5^{\frac{2}{s}}\rho_{i}}(z_{i})}\left((4\times 5^{\frac{2}{s}}\rho_{i})^{s-\tau}|G|\right)^{\gamma}\dmutt\right)^{\frac{2}{\gamma}}&+\frac{1}{\tau}\dashint_{Q_{4\times 5^{\frac{2}{s}}\rho_{i}}(z_{i})}|D^{\tau}d_{0}f|^{2}\dmutt\\
  &\quad+\Tail_{2,s}\left(\frac{f-(f)_{B_{4\times 5^{\frac{2}{s}}\rho_{i}}(x_{i})}(t)}{(4\times 5^{\frac{2}{s}}\rho_{i})^{\tau}}\,;\,Q_{4\times 5^{\frac{2}{s}}\rho_{i}}(z_{i})\right)^{2}\leq (c\lambda\delta)^{2}
\end{aligned}
\end{equation}}
\\
\noindent
for some constant $c=c(n,s,L,q,\tau)$. We first note from \eqref{cz non pa : r1r2defn}, \eqref{cz non pa : exitradius} and \eqref{cz non pa : rhoidefn} that there is a natural number $l$ such that 
\begin{equation*}
    5^{\frac{2}{s}}\times 2^{j_{0}+2}\mathcal{R}_{1,2}<2^{l}\times \overline{\rho}\leq5^{\frac{2}{s}}\times 2^{j_{0}+3}\mathcal{R}_{1,2},
\end{equation*}
where we denote $\overline{\rho}=4\times 5^{\frac{2}{s}}\rho_{i}$ for a clear notation.
After a few modifications of the proof for \eqref{cz non pa : tailestimate} with $p=\infty$, we deduce that
{\small\begin{equation*}
\begin{aligned}
    &\Tail_{\infty,2s}\left(\frac{u-(u)_{B_{\overline{\rho}}(x_{i})}(t)}{\overline{\rho}^{s+\tau}};Q_{\overline{\rho}}(z_{i})\right)\\
    &\leq c\left[\sum_{j=1}^{l}2^{j(-s+\tau)}\sup_{t\in \Lambda_{2^{j}\overline{\rho}}(t_{i})}\dashint_{B_{2^{j}\overline{\rho}}(x_{i})}\frac{|u-(u)_{B_{2^{j}\overline{\rho}}(x_{i})}(t)|}{(2^{j}\overline{\rho})^{s+\tau}}\dx+\frac{2^{-2sl+s+\tau}}{\overline{\rho}^{s+\tau}}\sup_{t\in \Lambda_{2}}\dashint_{B_{2}}\frac{|u-(u)_{B_{2}}(t)|}{2^{s+\tau}}\dx\right]\\
    &\quad+\frac{c}{(2-r_{1})^{n+2s}}\left(\frac{2}{\overline{\rho}}\right)^{-s+\tau}\Tail_{\infty,2s}\left(\frac{u-(u)_{B_{2}}(t)}{2^{s+\tau}};Q_{2}\right),
\end{aligned}
\end{equation*}}
where $c=c(n,s)$. Then using \eqref{cz non pa : lambda0}, \eqref{cz non pa : exitradius} and the fact that 
\begin{equation*}
2^{-2s l}\left(\frac{2}{\overline{\rho}}\right)^{s+\tau}\leq \frac{c}{\mathcal{R}_{1,2}^{\frac{5n}{s}}}  
  \quad\text{and}\quad \sum_{i=1}^{\infty}2^{i\left(-s+\tau\right)}\leq c(n,s,\tau),
\end{equation*}
we estimate the above term as
\begin{equation}
\label{cz non pa :bcompu1}
    \Tail_{\infty,2s}\left(\frac{u-(u)_{B_{\overline{\rho}}(x_{i})}(t)}{\overline{\rho}^{s+\tau}};Q_{\overline{\rho}}(z_{i})\right)\leq c\lambda,
\end{equation}
where $c=c(n,s,L,q,\tau)$. In addition, H\"older's inequality and \eqref{cz non pa : exitradius} imply
\begin{equation}
\label{cz non pa :bcompu2}
     \left(\frac{1}{\tau}\dashint_{Q_{\overline{\rho}}(z_{i})}|D^{\tau}d_{s}u|^{2}\dmutt\right)^{\frac{1}{2}}\leq c\left(\dashint_{Q_{\overline{\rho}}(z_{i})}|D^{\tau}d_{s}u|^{p}\dmutt\right)^{\frac{1}{p}}\leq c\lambda.
\end{equation}
We combine the above two estimates to obtain \eqref{cz non pa : bcompassm1}. Similarly, with the aid of \eqref{cz non pa : tailho} and \eqref{cz non pa : tailestimate} along with \eqref{cz non pa : tildeqprop}, we get that
{\small\begin{equation*}
\begin{aligned}
    &\Tail_{2,s}\left(\frac{f-(f)_{B_{\overline{\rho}}(x_{i})}(t)}{(\overline{\rho})^{\tau}}\,;\,Q_{\overline{\rho}}(z_{i})\right)\\
    &\leq \Tail_{\tilde{q},s}\left(\frac{f-(f)_{B_{\overline{\rho}}(x_{i})}(t)}{(\overline{\rho})^{\tau}}\,;\,Q_{\overline{\rho}}(z_{i})\right)\\
    &\leq c_{q}\sum_{j=1}^{l}2^{i\left(-s+\tau+\frac{2s}{\tilde{q}}\right)}\left(\frac{1}{\tau}\dashint_{\mathcal{Q}_{2^{j}\overline{\rho}}(z_{i})}|D^{\tau}d_{0}f|^{\tilde{q}}\dmutt\right)^{\frac{1}{\tilde{q}}}+c_{q}2^{-s l}\left(\frac{2}{\overline{\rho}}\right)^{\tau+s}\left(\frac{1}{\tau}\dashint_{\mathcal{Q}_{2}}|D^{\tau}d_{0}f|^{q}\dmutt\right)^{\frac{1}{q}}\\
    &\quad+\frac{\tilde{c}}{(2-r_{1})^{n+2s}}\left(\frac{2}{\overline{\rho}}\right)^{-s+\tau+\frac{2s}{\tilde{q}}}\Tail_{q,s}\left(\frac{f-(f)_{B_{2}}(t)}{2^{\tau}};Q_{2}\right).
\end{aligned}
\end{equation*}}
Hence, as in \eqref{cz non pa :bcompu1} and \eqref{cz non pa :bcompu2},  there is a constant $c=c(n,s,L,q,\tau)$ such that
\begin{equation*}
    \left(\frac{1}{\tau}\dashint_{Q_{\overline{\rho}}(z_{i})}|D^{\tau}d_{0}f|^{2}\dmutt\right)^{\frac{1}{2}}+\Tail_{2,s}\left(\frac{f-(f)_{B_{\overline{\rho}}(x_{i})}(t)}{(\overline{\rho})^{\tau}}\,;\,Q_{\overline{\rho}}(z_{i})\right)\leq c\delta\lambda.
\end{equation*}
Additionally, \eqref{cz non pa : exitradius} and \eqref{cz non pa : rangeofkappa1} yield that
\begin{equation*}
    \left(\frac{1}{\tau}\dashint_{Q_{4\times 5^{\frac{2}{s}}\rho_{i}}(z_{i})}\left((4\times 5^{\frac{2}{s}}\rho_{i}\rho_{i})^{s-\tau}|G|\right)^{\gamma}\dmutt\right)^{\frac{2}{\gamma}}\leq c\delta\lambda.
\end{equation*}
We combine the above two inequality to show that \eqref{cz non pa : bcompassm2} holds.
\end{rmk}


\section{\texorpdfstring{$L^{q}$}{Lq}-estimate of \texorpdfstring{$d_{s}u$}{dtaudsu}}
\label{cz non pa :lqsection}
In this section, we prove our main theorem. Since we have established comparison estimates and constructed coverings of upper level sets, we are able to obtain $L^{q}$-estimate of $d_{s}u$ with the estimate \eqref{cz non pa : main est} via a bootstrap argument as in \cite[Theorem 1.2]{BKc}.
Let us define
\begin{equation}
\label{cz non pa : ph}
p_{h}=2\left(1+\frac{s}{n}\right)^{h},\quad\mbox{ for }h=0,1,2,\ldots.
\end{equation}
Then there is a positive integer $h_{q}$ such that
\begin{equation}
\label{cz non pa : hq}
p_{h_{q}-1}< q \leq  p_{h_{q}}.
\end{equation}
We now prove the following lemma which is an essential ingredient to use a boot strap argument.
\begin{lem}
\label{cz non pa : ind prof}
Let $h\in\{0,1,\ldots,h_{q}-1\}$. Suppose that $u$ is a weak solution to \eqref{cz non pa : localized pb} with 
\begin{equation*}
    f\in L^{q}\left(\Lambda_{2}; L^{1}_{s}(\mathbb{R}^{n})\right)\quad\text{and}\quad\int_{\mathcal{Q}_{2}}|D^{\tau}d_{s}u|^{p_{h}}+|D^{\tau}d_{0}f|^{q}+|G|^{\frac{q\gamma}{b_{\tau}}}\dmutt<\infty.
\end{equation*}
Then there is a sufficiently small $\delta=\delta(n,s,L,q,\tau)\in(0,1]$ independent of $h$ such that if $A$ is $(\delta,2)$-vanishing in $Q_{2}$, then we have that
{\small\begin{equation}
\label{cz non pa : ind est}
\begin{aligned}
 \left(\dashint_{\mathcal{Q}_{1}}|D^{\tau}d_{s}u|^{\tilde{p}}\dmutt\right)^{\frac{1}{\tilde{p}}}
 &\leq c\left(\left(\dashint_{\mathcal{Q}_{2}}|D^{\tau}d_{s}u|^{p_{h}}\dmutt\right)^{\frac{1}{p_{h}}}+\Tail_{\infty,2s}\left(\frac{u-(u)_{B_{2}}(t)}{2^{s+\tau}}\,;\,Q_{2}\right)\right)\\
 &\quad+c\left(\sup_{t\in \Lambda_{2}}\dashint_{B_{2}}\frac{|u-(u)_{Q_{2}}|^{2}}{2^{2s+2\tau}}\dx\right)^{\frac{1}{2}}+c\left(\dashint_{\mathcal{Q}_{2}}\left(2^{{s-\tau}}|G|\right)^{\frac{\tilde{p}\gamma}{b_{\tau}}}\dmutt\right)^{\frac{b_{\tau}}{\tilde{p}\gamma}}\\
    &\quad+c\left(\left(\dashint_{\mathcal{Q}_{2}}|D^{\tau}d_{0}f|^{q}\dmutt\right)^{\frac{1}{q}}+\Tail_{q,s}\left(\frac{f-(f)_{B_{2}}(t)}{2^{\tau}}\,;\,Q_{2}\right)\right)
\end{aligned}
\end{equation}}\\
\noindent
for some constant $c=c(n,s,L,q,\tau)$, where the constant $b_{\tau}$ is defined in \eqref{cz non pa : constb} and
\begin{equation}
\label{cz non pa : tildep}
\tilde{p}=\begin{cases}
    p_{h+1}&\text{if }h<h_{q}-1,\\
    q&\text{if }h=h_{q}-1.
\end{cases}
\end{equation}
\end{lem}
\begin{proof}
Let us first fix $1\leq r_{1}<r_{2}\leq 2$ and $\epsilon>0$. Then we select $\lambda_{0}$ as given in Lemma \ref{cz non pa : coveringL} with $p=p_{h}$ and select $\delta=\delta(n,s,L,\epsilon)$ determined in Lemma \ref{cz non pa : nsccomp}. For any $N\geq \lambda_{0}$, we define a function $\phi:[1,2]\to\mathbb{R}$ by
\begin{equation}
\label{cz non pa : phifunt}
\phi_{N}(r)=\left(\dashint_{\mathcal{Q}_{r}}|D^{\tau}d_{s}u|_{N}^{\tilde{p}}\dmutt\right)^{\frac{1}{\tilde{p}}},
\end{equation}
where
$|D^{\tau}d_{s}u|_{N}=\min\left\{|D^{\tau}d_{s}u|,N\right\}$. We now claim that for $N\geq\lambda_{0}$, the following holds
\begin{equation}
\label{cz non pa : inddesried}
    \phi_{N}(r_{1})\leq \frac{\phi_{N}(r_{2})}{2}+c\lambda_{0}+c\left(\dashint_{\mathcal{Q}_{2}}\left(2^{{s-\tau}}|G|\right)^{\frac{\tilde{p}\gamma}{b_{\tau}}}\dmutt\right)^{\frac{b_{\tau}}{\tilde{p}\gamma}}
\end{equation}
for some constant $c=c(n,s,L,q,\tau)$. Using Fubini's theorem, we observe that
\begin{equation*}
\begin{aligned}
    \int_{\mathcal{Q}_{r_{1}}}|D^{\tau}d_{s}u|_{N}^{\tilde{p}}\dmutt&=\int_{0}^{\infty}\tilde{p}\lambda^{\tilde{p}-1}\mu_{\tau,t}\left(\left\{(x,y,t)\in\mathcal{Q}_{r_{1}}\,:\,|D^{\tau}d_{s}u|_{N}(x,y,t)>\lambda\right\}\right)\,d\lambda\\
    &=\int_{0}^{\mathcal{M}\lambda_{0}}\tilde{p}\lambda^{\tilde{p}-1}\mu_{\tau,t}\left(\left\{(x,y,t)\in\mathcal{Q}_{r_{1}}\,:\,|D^{\tau}d_{s}u|_{N}(x,y,t)>\lambda\right\}\right)\,d\lambda\\
    &\quad+\int_{\mathcal{M}\lambda_{0}}^{N}\tilde{p}\lambda^{\tilde{p}-1}\mu_{\tau,t}\left(\left\{(x,y,t)\in\mathcal{Q}_{r_{1}}\,:\,|D^{\tau}d_{s}u|_{N}(x,y,t)>\lambda\right\}\right)\,d\lambda\eqqcolon I_{1}+I_{2},
\end{aligned}
\end{equation*}
where $\mathcal{M}>1$ is a constant which will be determined later and $N>\mathcal{M}\lambda_{0}$. We now estimate $I_{1}$ and $I_{2}$.

\noindent
\textbf{Estimate of $I_{1}$.} 
A simple calculation yields that
\begin{equation*}
    I_{1}\leq \mu_{\tau,t}\left(\mathcal{Q}_{r_{1}}\right)(\mathcal{M}\lambda_{0})^{\tilde{p}}.
\end{equation*}
\textbf{Estimate of $I_{2}$.}
By a change of variable and \eqref{cz non pa : levelset} of Lemma \ref{cz non pa : coveringL} with $p=p_{h}$, we get that
{\small\begin{equation*}
\begin{aligned}
    I_{2}&=\int_{\lambda_{0}}^{N\mathcal{M}^{-1}}\tilde{p}\mathcal{M}(\mathcal{M}\lambda)^{\tilde{p}-1}\mu_{\tau,t}\left(\left\{(x,y,t)\in\mathcal{Q}_{r_{1}}\,:\,|D^{\tau}d_{s}u|_{N}(x,y,t)>\mathcal{M}\lambda\right\}\right)\,d\lambda\\
    &\leq \sum_{i\geq0} \int_{\lambda_{0}}^{N\mathcal{M}^{-1}}\tilde{p}\mathcal{M}(\mathcal{M}\lambda)^{\tilde{p}-1}\mu_{\tau,t}\left(\left\{(x,y,t)\in\mathcal{Q}_{5^{\frac{2}{s}}\rho_{i}}(z_{i})\,:\,|D^{\tau}d_{s}u|_{N}(x,y,t)>\mathcal{M}\lambda\right\}\right)\,d\lambda \\
    &\quad+\sum_{j\geq0} \int_{\lambda_{0}}^{N\mathcal{M}^{-1}}\tilde{p}\mathcal{M}(\mathcal{M}\lambda)^{\tilde{p}-1}\mu_{\tau,t}\left(\left\{(x,y,t)\in\mathcal{Q}_{5^{\frac{1}{s}}\overline{r}_{j}}(x_{1,j},x_{2,j},t_{0,j})\,:\,|D^{\tau}d_{s}u|_{N}(x,y,t)>\mathcal{M}\lambda\right\}\right)\,d\lambda\\
    &\eqqcolon I_{2,1}+I_{2,2},
\end{aligned}
\end{equation*}}\\
\noindent
where we have used the fact that 
\[\left\{(x,y,t)\in\mathcal{Q}_{r_{1}}\,:\,|D^{\tau}d_{s}u|_{N}(x,y,t)>\mathcal{M}\lambda\right\}\subset \left\{(x,y,t)\in\mathcal{Q}_{r_{1}}\,:\,|D^{\tau}d_{s}u|(x,y,t)>\lambda\right\}.\]
By assuming 
\begin{equation}
\label{cz non pa : mathcalMcond1}
    \mathcal{M}>c_{c}c_{d},
\end{equation}
where the constants $c_{c}=c_{c}(n,s,L,\tau)$ and $c_{d}=c_{d}(n,s,L,q,\tau)$ are given in Lemma \ref{cz non pa : nsccomp} with $\rho_{i}$ replaced by $5^{\frac{2}{s}}\rho_{i}$ and Remark \ref{cz non pa : diagonalpartlambda}, respectively, we now estimate $I_{2,1}$ as
\begin{equation}
\label{cz non pa : esti21b}
\begin{aligned}
    I_{2,1}&\leq \sum_{i} \int_{\lambda_{0}}^{N\mathcal{M}^{-1}}\tilde{p}\mathcal{M}(\mathcal{M}\lambda)^{\tilde{p}-1}\mu_{\tau,t}\left(\left\{(x,y,t)\in\mathcal{Q}_{5^{\frac{2}{s}}\rho_{i}}(z_{i})\,:\,|D^{\tau}d_{s}(u-v)|_{N}(x,y,t)>\mathcal{M}\lambda\right\}\right)\,d\lambda\\
    &\leq \sum_{i} \int_{\lambda_{0}}^{N\mathcal{M}^{-1}}\tilde{p}\mathcal{M}(\mathcal{M}\lambda)^{\tilde{p}-3}\int_{\mathcal{Q}_{5^{\frac{2}{s}}\rho_{i}}(z_{i})}|D^{\tau}d_{s}(u-v)|^{2}\dmutt\,d\lambda\\
    &\leq c\sum_{i} \int_{\lambda_{0}}^{N\mathcal{M}^{-1}}\mathcal{M}^{\tilde{p}-2}\lambda^{\tilde{p}-1}\epsilon^{2}\mu_{\tau,t}\left(\mathcal{Q}_{\rho_{i}}(z_{i})\right)\,d\lambda
\end{aligned}
\end{equation}
for some constant $c=c(n,s,L,q,\tau)$. In the above estimates, we have used weak 1-1 estimates, \eqref{cz non pa : nsccompres} and \eqref{cz non pa : doubling}. On the other hand, using weak 1-1 estimates, \eqref{cz non pa : pstra norm}, \eqref{cz non pa : paraconj} and \eqref{cz non pa : ph}, we have that
\begin{equation*}
\begin{aligned}
    I_{2,2}&\leq \tilde{p}\lambda^{\tilde{p}}\sum_{j} \int_{\lambda_{0}}^{N\mathcal{M}^{-1}}\mathcal{M}(\mathcal{M}\lambda)^{-(p_{h})_{\#}+\tilde{p}-1}\int_{\mathcal{Q}_{5^{\frac{1}{s}}\overline{r}_{j}}(x_{1,j},x_{2,j},t_{0,j})}|D^{\tau}d_{s}u|^{(p_{h})_{\#}}\dmutt\,d\lambda\\
    &\leq \frac{\tilde{p}(c_{od,h})^{(p_{h})_{\#}}\lambda^{\tilde{p}-1}}{\mathcal{M}^{\frac{2s}{n}}}\sum_{j} \int_{\lambda_{0}}^{N\mathcal{M}^{-1}}\mu_{\tau,t}\left(\mathcal{Q}_{\overline{r}_{j}}(x_{1,j},x_{2,j},t_{0,j})\right)\,d\lambda,
\end{aligned}
\end{equation*}
where the constant $c_{od,h}=c_{od,h}(n,s,p_{h},\tau)$ is determined in \eqref{cz non pa : pstra norm}. Note that the constant $c=\tilde{p}\times\max\limits_{h=0,1,\ldots,h_{q}-1}\left(c_{od,h}\right)^{q_{\#}}$ depends only on $n,s,L,q$ and $\tau$ and is bigger than the constant $\tilde{p}(c_{od,h})^{(p_{h})_{\#}}$, as $h_{q}$ depends only on $n,s,q$ and $\tau$ (see \eqref{cz non pa : hq}). Hence, we have that
\begin{equation}
\label{cz non pa : esti22b}
\begin{aligned}
    I_{2,2}\leq  \frac{c\lambda^{\tilde{p}-1}}{\mathcal{M}^{\frac{2s}{n}}}\sum_{j} \int_{\lambda_{0}}^{N\mathcal{M}^{-1}}\mu_{\tau,t}\left(\mathcal{Q}_{\overline{r}_{j}}(x_{1,j},x_{2,j},t_{0,j})\right)\,d\lambda
\end{aligned}
\end{equation}
for some constant $c=c(n,s,L,q,\tau)$.
Combine \eqref{cz non pa : diagonal estimate} with $p=p_{h}$, \eqref{cz non pa : esti21b} and \eqref{cz non pa : esti22b}  to see that
{\small\begin{equation*}
\begin{aligned}
    I_{2}&\leq c\int_{\lambda_{0}}^{N\mathcal{M}^{-1}}\lambda^{\tilde{p}-1}\left(\mathcal{M}^{\tilde{p}-3}\epsilon^{2}+\frac{1}{\mathcal{M}^{\frac{2s}{n}}}\right)\frac{c}{\lambda^{p_{h}}}\int_{\mathcal{Q}_{r_{2}}\cap\{|D^{\tau}d_{s}u|>a_{u}\lambda\}}|D^{\tau}d_{s}u|^{p_{h}}\dmutt\,d\lambda\\
    &\quad+c\int_{\lambda_{0}}^{N\mathcal{M}^{-1}}\lambda^{\tilde{p}-1}\left(\mathcal{M}^{\tilde{p}-3}\epsilon^{2}+\frac{1}{\mathcal{M}^{\frac{2s}{n}}}\right)\frac{c}{(\delta\lambda)^{\tilde{q}}}\int_{\mathcal{Q}_{r_{2}}\cap\{|D^{\tau}d_{0}f|>a_{f}\delta\lambda\}}|D^{\tau}d_{0}f|^{\tilde{q}}\dmutt\,d\lambda\\
    &\quad+c\int_{\lambda_{0}}^{N\mathcal{M}^{-1}}\lambda^{\tilde{p}-1}\left(\mathcal{M}^{\tilde{p}-3}\epsilon^{2}+\frac{1}{\mathcal{M}^{\frac{2s}{n}}}\right)\frac{c}{(\delta\lambda)^{b_{\tau}}}\int_{\mathcal{Q}_{r_{2}}\cap\{|G|^{\gamma}>(a_{g}\delta\lambda)^{b_{\tau}}G_{0}^{-1}\}}|G|^{\gamma}G_{0}\dmutt\,d\lambda\eqqcolon J_{1}+J_{2}+J_{3}
\end{aligned}
\end{equation*}}\\
\noindent
where the constants $a_{u}$, $a_{f}$ and $a_{g}$ are determined in \eqref{cz non pa : diagonal estimate},  and the constants $G_{0}$ and $\tilde{q}$ are defined in \eqref{cz non pa : constb} and \eqref{cz non pa : tildeq}, respectively. 
Using Fubini's theorem and taking $\mathcal{M}=\mathcal{M}(n,s,L,q,\tau)>1$ which satisfies \eqref{cz non pa : mathcalMcond1} and then choosing $\epsilon=\epsilon(n,s,L,q,\tau)\in(0,1)$, we have that
\begin{equation*}
    J_{1}\leq \frac{1}{10^{5nq}}\int_{\mathcal{Q}_{r_{2}}}|D^{\tau}d_{s}u|_{N}^{\tilde{p}}\dmutt.
\end{equation*}
On the other hand, if $\tilde{p}\leq\tilde{q}$, then we estimate
{\small\begin{equation}
\label{cz non pa : tildeqfirstcond}
\begin{aligned}
 J_{2}&\leq \frac{c}{\lambda_{0}^{q-\tilde{p}}}\int_{\lambda_{0}}^{N\mathcal{M}^{-1}}\lambda^{q-1}\left(\mathcal{M}^{\tilde{p}-3}\epsilon^{2}+\frac{1}{\mathcal{M}^{\frac{2s}{n}}}\right)\frac{c}{(\delta\lambda)^{\tilde{q}}}\int_{\mathcal{Q}_{r_{2}}\cap\{|D^{\tau}d_{0}f|>a\lambda\}}|D^{\tau}d_{0}f|^{\tilde{q}}\dmutt\,d\lambda\\
 &\leq \frac{c}{\lambda_{0}^{q-\tilde{p}}}\int_{\mathcal{Q}_{r_{2}}}|D^{\tau}d_{0}f|^{q}\dmutt\\
 &\leq c\left(\int_{\mathcal{Q}_{r_{2}}}|D^{\tau}d_{0}f|^{q}\dmutt\right)^{\frac{\tilde{p}}{q}},
\end{aligned}
\end{equation}}\\
\noindent
where we have used $1\leq\left(\frac{\lambda}{\lambda_{0}}\right)^{q-\tilde{p}}$, Fubini's theorem (thanks to the relation $\tilde{q}<q$, from \eqref{cz non pa : tildeqprop2}) and the fact that $\lambda_{0}^{\tilde{p}-q}\leq c\left(\int_{\mathcal{Q}_{r_{2}}}|D^{\tau}d_{0}f|^{q}\dmutt\right)^{\frac{\tilde{p}-q}{q}}$.
If $\tilde{p}>\tilde{q}$, Fubini's theorem and H\"older's inequality yield that
\begin{equation*}
    J_{2}\leq c\int_{\mathcal{Q}_{r_{2}}}|D^{\tau}d_{0}f|^{\tilde{p}}\dmutt\leq c\left(\int_{\mathcal{Q}_{r_{2}}}|D^{\tau}d_{0}f|^{q}\dmutt\right)^{\frac{\tilde{p}}{q}}.
\end{equation*}
Similarly,  we get that
\begin{equation*}
    J_{3}\leq c\int_{\mathcal{Q}_{r_{2}}}\left(|G|^{\gamma}G_{0}\right)^{\frac{\tilde{p}}{b_{\tau}}}\leq c\left(\dashint_{\mathcal{Q}_{2}}\left(2^{{s-\tau}}|G|\right)^{\frac{\tilde{p}\gamma}{b_{\tau}}}\dmutt\right)^{\frac{b_{\tau}}{\gamma}}.
\end{equation*}
Consequently, combining all the estimates $I_{1}$ and $I_{2}$ and recalling \eqref{cz non pa : size of b mut} and \eqref{cz non pa : phifunt}, we obtain the desired result \eqref{cz non pa : inddesried}. We recall the definitions of $\lambda_{0}$ given in \eqref{cz non pa : lambda0} and $G_{0}$ given in \eqref{cz non pa : constb}. Then using Lemma \ref{cz non pa : technicallemma} and passing to the limit $N\to 0$, we get \eqref{cz non pa : ind est}.
\end{proof}
With the aid of Lemma \ref{cz non pa : ind prof}, we now prove our main theorem.


\noindent
\textbf{Proof of Theorem \ref{cz non pa : main thm}.} Let us fix $Q_{r}(z_{0})\Subset \Omega_{T}$ with $r\in(0,R]$. We now take
\begin{equation}
\label{cz non pa : tausigma}
\tau=\frac{\sigma q}{q-2}<\min\left\{s-\frac{2s}{q},1-s\right\}. 
\end{equation}
Let us choose $\delta=\delta(n,s,L,q,\sigma)$ determined in Lemma \ref{cz non pa : ind prof}. We claim that
\begin{equation}
\label{cz non pa : ind est1}
\begin{aligned}
 \left(\dashint_{\mathcal{Q}_{\frac{r}{2}}(z_{0})}|D^{\tau}d_{s}u|^{\tilde{p}}\dmutt\right)^{\frac{1}{\tilde{p}}}
 &\leq c\mathcal{R}(u,f,g,p_{0},\tilde{p},Q_{r}(z_{0}))
\end{aligned}
\end{equation}
for some constant $c=c(\mathsf{data})$, where $\tilde{p}$ is given in \eqref{cz non pa : tildep} with $h=0$ and we denote 
{\small\begin{equation*}
\begin{aligned}
    \mathcal{R}(u,f,g,p_{0},\tilde{p},Q_{r}(z_{0}))=&\left(\dashint_{\mathcal{Q}_{r}(z_{0})}|D^{\tau}d_{s}u|^{p_{0}}\dmutt\right)^{\frac{1}{p_{0}}}+\Tail_{\infty,2s}\left(\frac{u-(u)_{B_{r}(x_{0})}(t)}{r^{s+\tau}}\,;\,Q_{r}(z_{0})\right)\\
 &\quad+\left(\sup_{t\in \Lambda_{r}(t_{0})}\dashint_{B_{r}(x_{0})}\frac{|u-(u)_{Q_{r}(z_{0})}|^{2}}{r^{2s+2\tau}}\dx\right)^{\frac{1}{2}}+\left(\dashint_{\mathcal{Q}_{r}(z_{0})}\left(r^{\frac{s-\tau}{\tilde{p}}}|G|\right)^{\frac{\tilde{p}\gamma}{b_{\tau}}}\dmutt\right)^{\frac{b_{\tau}}{\tilde{p}\gamma}}\\
    &\quad+\left(\dashint_{\mathcal{Q}_{r}(z_{0})}|D^{\tau}d_{0}f|^{q}\dmutt\right)^{\frac{1}{q}}+\Tail_{q,s}\left(\frac{f-(f)_{B_{r}(x_{0})}(t)}{r^{\tau}}\,;\,Q_{r}(z_{0})\right).
\end{aligned}
\end{equation*}}\\
\noindent
To this end, we define for any $x,y\in \mathbb{R}^{n}$, $t\in\Lambda_{2}$ and $\xi\in\mathbb{R}$,
{\small\begin{equation*}
\begin{aligned}
    &\tilde{u}(x,t)=\frac{u\left(r_{s}x+x_{1},r_{s}^{2s}t+t_{1}\right)}{r_{s}^{s+\tau}},\quad\tilde{f}(x,t)=\frac{f\left(r_{s}x+x_{1},r_{s}^{2s}t+t_{1}\right)}{r_{s}^{\tau}},\\
    &\tilde{g}(x,t)=r_{s}^{s-\tau}g\left(r_{s}x+x_{1},r_{s}^{2s}t+t_{1}\right),\quad
    \tilde{A}(x,y,t)=A\left(r_{s}x+x_{1},r_{s}y+x_{1},r_{s}^{2s}t+t_{1}\right),\quad \tilde{\Phi}(\xi)=\frac{\Phi(r_{s}^{\tau}\xi)}{r_{s}^{\tau}}
\end{aligned}
\end{equation*}}\\
\noindent
where $z_{1}\in \overline{Q}_{\frac{r}{2}}(z_{0})$ and $r_{s}=\left(\frac{s(\sqrt{2}-1)}{4}\right)^{\frac{2}{s}}r$, in order to see that $\tilde{u}$ is a weak solution to \eqref{cz non pa : localized pb} with $f=\tilde{f}$ $g=\tilde{g}$, $A=\tilde{A}$ and $\Phi=\tilde{\Phi}$. Moreover, we observe that
\begin{equation}
\label{cz non pa : cylinderinc}
Q_{r_{s}}(z_{1})\subset Q_{\frac{r}{\sqrt{2}}}(z_{0}).
\end{equation}
We now apply Lemma \ref{cz non pa : ind prof} with $u=\tilde{u}$, $f=\tilde{f}$, $g=\tilde{g}$ and $h=0$, and use change of the variables to get that
\begin{equation}
\label{cz non pa : ind est2}
\begin{aligned}
 \left(\dashint_{\mathcal{Q}_{\frac{r_{s}}{2}}(z_{1})}|D^{\tau}d_{s}u|^{\tilde{p}}\dmutt\right)^{\frac{1}{\tilde{p}}}
 &\leq c\mathcal{R}(u,f,g,p_{0},\tilde{p},Q_{r_{s}}(z_{1}))
\end{aligned}
\end{equation}
for some constant $c=c(\mathsf{data})$. After a few algebraic calculations along with Lemma \ref{cz non pa : taillem}, \eqref{cz non pa : size of b mut} and  \eqref{cz non pa : cylinderinc}, the expression in 
 \eqref{cz non pa : ind est2} is estimated as 
\begin{equation}
\label{cz non pa : ind est22}
\begin{aligned}
 \left(\dashint_{\mathcal{Q}_{\frac{r_{s}}{2}}(z_{1})}|D^{\tau}d_{s}u|^{\tilde{p}}\dmutt\right)^{\frac{1}{\tilde{p}}}
 &\leq c\mathcal{R}\left(u,f,g,p_{0},\tilde{p},Q_{\frac{r}{\sqrt{2}}}(z_{0}))\right)
\end{aligned}
\end{equation}
On the other hand, for any $\mathcal{Q}_{\frac{r_{s}}{8\sqrt{n}}}\equiv\mathcal{B}_{\frac{r_{s}}{8\sqrt{n}}}\left(x_{1},x_{2}\right)\times\Lambda_{\frac{r_{s}}{8\sqrt{n}}}(t_{2})$ with $(x_{1},t_{2}),(x_{2},t_{2})\in \overline{Q}_{\frac{r}{2}}(z_{0})$  satisfying \eqref{cz non pa : lem.cov.step4.rhsdist}, we use Lemma \ref{cz non pa : off rever lem} to obtain that
{\small\begin{equation*}
\begin{aligned}
    \left(\dashint_{\mathcal{Q}_{\frac{r_{s}}{8\sqrt{n}}}}|D^{\tau}d_{s}u|^{\tilde{p}}\dmutt\right)^{\frac{1}{\tilde{p}}}&\leq c\left(\dashint_{\mathcal{Q}_{\frac{r_{s}}{8\sqrt{n}}}}|D^{\tau}d_{s}u|^{p_{0}}\dmutt\right)^{\frac{1}{p_{0}}}+c\left[\sum_{d=1}^{2}E_{p,\tau}\left(u\,;\,{Q}_{\frac{r_{s}}{8\sqrt{n}}}(x_{d},t_{2})\right)\right]
\end{aligned}
\end{equation*}}
\\
\noindent
for some constant $c=c(\mathsf{data})$. As in \eqref{cz non pa : ind est22} along with Lemma \ref{cz non pa : taillem}, \eqref{cz non pa : inclusion measure} and \eqref{cz non pa : cylinderinc}, we deduce that
{\small\begin{equation}
\label{cz non pa : ind est3}
\begin{aligned}
 \left(\dashint_{\mathcal{Q}_{\frac{r_{s}}{8\sqrt{n}}}}|D^{\tau}d_{s}u|^{\tilde{p}}\dmutt\right)^{\frac{1}{\tilde{p}}}
 &\leq c\left(\dashint_{Q_{\frac{r}{\sqrt{2}}}(z_{0})}|D^{\tau}d_{s}u|^{p_{0}}\dmutt\right)^{\frac{1}{p_{0}}}\\
 &\quad+c\left(\sup_{t\in \Lambda_{\frac{r}{\sqrt{2}}}(t_{0})}\dashint_{B_{\frac{r}{\sqrt{2}}}(x_{0})}\frac{\big|u-(u)_{Q_{\frac{r}{\sqrt{2}}}(z_{0})}\big|^{2}}{(\frac{r}{\sqrt{2}})^{2s+2\tau}}\dx\right)^{\frac{1}{2}}.
\end{aligned}
\end{equation}}
\\
\noindent
Since $\overline{\mathcal{Q}}_{\frac{r}{2}}(z_{0})$ is a compact set, there are finite mutually disjoint open sets $\mathcal{Q}_{r_{s}}(z_{1,i})$ and $\mathcal{Q}_{r_{s}}(x_{1,j},x_{2,j},t_{2,j})$ for some points $z_{1,i}, (x_{1,j},t_{2,j}), (x_{2,j},t_{2,j})\in \overline{\mathcal{Q}}_{\frac{r}{2}}(z_{0})$ such that
\begin{equation*}
    \overline{\mathcal{Q}}_{\frac{r}{2}}(z_{0})\subset \left(\bigcup_{i}\mathcal{Q}_{r_{s}}(z_{1,i})\right)\bigcup\left(\bigcup_{j}\mathcal{Q}_{r_{s}}(x_{1,j},x_{2,j},t_{2,j})\right)\subset Q_{\frac{r}{\sqrt{2}}}(z_{0}),
\end{equation*}
and $\mathcal{Q}_{r_{s}}(x_{1,j},x_{2,j},t_{2,j})$ satisfies \eqref{cz non pa : lem.cov.step4.rhsdist}. Combine \eqref{cz non pa : ind est1} and \eqref{cz non pa : ind est3} to see that
\begin{equation}
\label{cz non pa : ind est11}
\begin{aligned}
 \left(\dashint_{\mathcal{Q}_{\frac{r}{2}}(z_{0})}|D^{\tau}d_{s}u|^{\tilde{p}}\dmutt\right)^{\frac{1}{\tilde{p}}}
 &\leq c\mathcal{R}\left(u,f,g,p_{0},\tilde{p},Q_{\frac{r}{\sqrt{2}}}(z_{0}))\right).
\end{aligned}
\end{equation}
Consequently, after a few simple calculations with Lemma \ref{cz non pa : taillem}, we estimate the right-hand side of \eqref{cz non pa : ind est11} to get that \eqref{cz non pa : ind est1} holds. In addition, using the standard covering argument along with \eqref{cz non pa : ind est3} and \eqref{cz non pa : ind est11}, we prove that $D^{\tau}d_{s}u\in L^{\tilde{p}}_{\mathrm{loc}}\left(\dmutt;\Omega\times\Omega\times(0,T)\right)$. If $h=0$, then by recalling \eqref{cz non pa : size of b mut}, \eqref{cz non pa : equbwGg}, \eqref{cz non pa : constb}, \eqref{cz non pa : hq}, \eqref{cz non pa : tildep} and \eqref{cz non pa : tausigma}, we obtain $d_{s}u\in L^{q}_{\mathrm{loc}}\left(\dmutt;\Omega\times\Omega\times(0,T)\right)$ and the desired estimate \eqref{cz non pa : main est}. Let us assume that $h>0$. We have shown that $D^{\tau}d_{s}u\in L^{p_{1}}_{\mathrm{loc}}\left(\dmutt;\Omega\times\Omega\times(0,T)\right)$ and \eqref{cz non pa : ind est3} and \eqref{cz non pa : ind est11} with $\tilde{p}=p_{1}$. Thus, by following the same line as in the proof for \eqref{cz non pa : ind est3} and \eqref{cz non pa : ind est11} with $p_{0}$ replaced by $p_{1}$, we have that
{\small\begin{equation}
\label{cz non pa : ind est32}
\begin{aligned}
 \left(\dashint_{\mathcal{Q}_{r_{s}}}|D^{\tau}d_{s}u|^{\tilde{p}_{1}}\dmutt\right)^{\frac{1}{\tilde{p}_{1}}}
 &\leq c\left(\dashint_{Q_{\frac{r}{\sqrt{2}}}(z_{0})}|D^{\tau}d_{s}u|^{p_{1}}\dmutt\right)^{\frac{1}{p_{1}}}\\
 &\quad+c\left(\sup_{t\in \Lambda_{\frac{r}{\sqrt{2}}}(t_{0})}\dashint_{B_{\frac{r}{\sqrt{2}}}(x_{0})}\frac{\big|u-(u)_{Q_{\frac{r}{\sqrt{2}}}(z_{0})}\big|^{2}}{(\frac{r}{\sqrt{2}})^{2s+2\tau}}\dx\right)^{\frac{1}{2}},
\end{aligned}
\end{equation}}
\\
for any $\mathcal{Q}_{\frac{r_{s}}{8\sqrt{n}}} =\mathcal{B}_{\frac{r_{s}}{8\sqrt{n}}}\left(x_{1},x_{2}\right)\times\Lambda_{\frac{r_{s}}{8\sqrt{n}}}(t_{2})$ with $(x_{1},t_{2}),(x_{2},t_{2})\in \overline{Q}_{\frac{r}{2}}(z_{0})$  satisfying \eqref{cz non pa : lem.cov.step4.rhsdist} and
{\small\begin{equation}
\label{cz non pa : ind est112}
\begin{aligned}
 \left(\dashint_{\mathcal{Q}_{\frac{r}{2}}(z_{0})}|D^{\tau}d_{s}u|^{\tilde{p}_{1}}\dmutt\right)^{\frac{1}{\tilde{p}_{1}}}
 &\leq c\mathcal{R}\left(u,f,g,p_{1},\tilde{p}_{1},Q_{\frac{r}{\sqrt{2}}}(z_{0}))\right)
\end{aligned}
\end{equation}}\\
\noindent
where $\tilde{p}_{1}$ is the constant defined in \eqref{cz non pa : tildep} with $h=1$. As a result, plugging \eqref{cz non pa : ind est11} to the first term in $\mathcal{R}\left(u,f,g,p_{1},\tilde{p}_{1},Q_{\frac{r}{\sqrt{2}}}(z_{0}))\right)$ and then after a few simple calculations with Lemma \ref{cz non pa : taillem}, we obtain \eqref{cz non pa : ind est1} with $p_{0}=p_{1}$ and $\tilde{p}=\tilde{p}_{1}$. In addition, using the standard covering argument along with \eqref{cz non pa : ind est32} and \eqref{cz non pa : ind est112}, we prove that $D^{\tau}d_{s}u\in L^{\tilde{p}_{1}}_{\mathrm{loc}}\left(\dmutt;\Omega\times\Omega_{T}\right)$. By iterating this procedure $l_{q}-1$ times, we obtain $D^{\tau}d_{s}u\in L^{q}_{\mathrm{loc}}\left(\dmutt;\Omega\times\Omega_{T}\right)$ with the estimate \eqref{cz non pa : ind est1} with $\tilde{p}=q$. By recalling \eqref{cz non pa : size of b mut}, \eqref{cz non pa : equbwGg}, \eqref{cz non pa : constb}, \eqref{cz non pa : hq}, \eqref{cz non pa : tildep} and \eqref{cz non pa : tausigma}, we conclude that $d_{s}u\in L^{q}_{\mathrm{loc}}\left(\dmutt;\Omega\times\Omega_{T}\right)$ with the desired estimate \eqref{cz non pa : main est}. 
\qed\\

We finish this section by proving Theorem \ref{cz non pa : main thm wrt g}.

\noindent
\textbf{Proof of Theorem \ref{cz non pa : main thm wrt g}.}
We note that in the proof of Lemma \ref{cz non pa : ind prof} and Theorem \ref{cz non pa : main thm}, we only use the condition \eqref{cz non pa : tautail} to apply Fubini's Theorem on the term $D^{\tau}d_{0}f$ (see \eqref{cz non pa : tildeqfirstcond}). Taking into account Remark \ref{cz non pa : taurmk}, if $f=0$, then the constant $\tau$ can be chosen in $(0,\min\{s,1-s\})$. Consequently, we allow for choosing $\sigma\in\left(0,\left(1-\frac{2}{q}\right)\min\{s,1-s\}\right)$ considering \eqref{cz non pa : tausigma} and we are able to prove the Theorem \ref{cz non pa : main thm wrt g} by following the same lines as in the proof of Theorem \ref{cz non pa : main thm} with $f=0$.
\qed

\appendix

\section{self-improving property of nonlocal parabolic equations}
\label{cz non pa : appen a}
\noindent
In this appendix, we prove a self-improving property of a weak solution $u$ to \eqref{cz non pa : eq1} with $f=g=0$. Throughout this section, we take 
\begin{equation}
\label{cz non pa : appentau}
\tau_{0}=\min\left\{\frac{s}{2},\frac{1-s}{2}\right\}.
\end{equation}
Before proving Lemma \ref{cz non pa : self impro}, we are going to prove a reverse H\"older's inequality on diagonal parts and obtain another covering lemma.


With the aid of the gluing lemma, we first obtain the following inequality.
\begin{lem}
\label{cz non pa : sobolevpoincare}
Let $u$ be a weak solution to \eqref{cz non pa : eq1} and let $Q_{\rho}(z_{0})\Subset \Omega_{T}$.
Then  for any $\varsigma\in(0,1]$, we have
{\small\begin{equation}
\label{cz non pa : paraSPI}
\begin{aligned}
    &\dashint_{Q_{\rho}(z_{0})}\frac{|u-(u)_{Q_{\rho}(z_{0})}|^{2}}{\rho^{2s+2\tau}}\dz\\
    &\leq \varsigma\sup_{t\in \Lambda_{\rho}(t_{0})}\dashint_{B_{\rho}(x_{0})}\frac{|u-(u)_{B_{\rho}(x_{0})}(t)|^{2}}{\rho^{2s+2\tau_{0}}}\dx+\frac{c}{\varsigma^{\beta}}\left(\frac{1}{\tau_{0}}\dashint_{\mathcal{Q}_{\rho}(z_{0})}|D^{\tau_{0}}d_{s}u|^{\gamma}\,d\mu_{\tau_{0},t}\right)^{\frac{2}{\gamma}}\\
    &\quad+c\frac{1}{\tau_{0}}\dashint_{\mathcal{Q}_{\rho}(z_{0})}|D^{\tau_{0}}d_{0}f|^{2}\,d\mu_{\tau_{0},t}+ c\Tail_{1,2s}\left(\frac{u-(u)_{B_{\rho}(x_{0})}(t)}{\rho^{s+\tau_{0}}};Q_{\rho}(z_{0})\right)^{2}\\
    &\quad+c\Tail_{1,s}\left(\frac{f-(f)_{B_{\rho}(x_{0})}(t)}{\rho^{\tau_{0}}};Q_{\rho}(z_{0})\right)^{2}+c\left(\frac{1}{\tau_{0}}\dashint_{\mathcal{Q}_{2\rho}(z_{0})}\left((2\rho)^{s-\tau}|G|\right)^{\gamma}\,d\mu_{\tau_{0},t}\right)^{\frac{2}{\gamma}}
\end{aligned}
\end{equation}}
for some constants $c=c(n,s,L)$ and $\beta=\beta(n,s)$, where the constant $\gamma$ is defined in \eqref{cz non pa : gamma}.
\end{lem}
\begin{proof}
We may assume that $z_{0}=0$. Note that
\begin{flalign*}
\dashint_{Q_{\rho}}|u-(u)_{Q_{\rho}}|^{2}\dz&\leq c\dashint_{Q_{\rho}}|u-(u)_{B_{\rho}}(t)|^{2}\dz+c\dashint_{Q_{\rho}}|(u)_{B_{\rho}}(t)-(u)_{Q_{\rho}}|^{2}\dz\eqqcolon I_{1}+I_{2}.
\end{flalign*}
In light of Lemma \ref{cz non pa : embed lem} with $h$, $p$ and $s$, replaced by $u$, $\gamma$ and $\tilde{s}\coloneqq s+\tau_{0}\left(1-\frac{2}{\gamma}\right)$, respectively, we first estimate $I_{1}$ as
{\small\begin{equation*}
\begin{aligned}
    I_{1}\leq c\left(\dashint_{Q_{\rho}}|u-(u)_{B_{\rho}}(t)|^{\tilde{\gamma}}\dz\right)^{\frac{2}{\tilde{\gamma}}} \leq  c\left(\frac{\rho^{\gamma(s+\tau_{0})}}{\tau_{0}}\dashint_{\mathcal{Q}_{\rho}}|D^{\tau_{0}}d_{s}u|^{\gamma}\,d\mu_{\tau_{0},t}\right)^{\frac{2}{\tilde{\gamma}}}\times\left(\sup_{t\in \Lambda_{\rho}}\dashint_{B_{\rho}}|u-(u)_{B_{\rho}(t)}|^{2}\dx\right)^{\frac{2\tilde{s}\gamma}{n\tilde{\gamma}}},
\end{aligned}
\end{equation*}}\\
\noindent
where $c=c(n,s)$  and $\tilde{\gamma}\coloneqq\gamma\left(1+\frac{2\tilde{s}}{n}\right)$. We next apply Young's inequality along with the fact that $\frac{\gamma}{\tilde{\gamma}}+\frac{2\tilde{s}\gamma}{n\tilde{\gamma}}=1$, in order to get 
\begin{equation*}
\begin{aligned}
    I_{1}\leq \vartheta\left(\frac{\rho^{\gamma(s+\tau_{0})}}{\tau_{0}}\dashint_{\mathcal{Q}_{\rho}}|D^{\tau_{0}}d_{s}u|^{\gamma}\,d\mu_{\tau_{0},t}\right)^{\frac{2}{\gamma}}+\frac{c}{\vartheta^{\beta}}\sup_{t\in \Lambda_{\rho}}\dashint_{B_{\rho}}|u-(u)_{B_{\rho}(t)}|^{2}\dx
\end{aligned}
\end{equation*}
for any $\vartheta\in(0,1]$ and for some constant $\beta=\beta(n,s)$. We now estimate $I_{2}$ as
\begin{equation*}
\begin{aligned}
I_{2}&\leq c\dashiint_{\Lambda_{\rho}\times \Lambda_{\rho}}|(u)_{B_{\rho}}(t)-(u)_{B_{\rho}}(t')|^{2}\dt\dt'\\
&\leq  c\dashiint_{\Lambda_{\rho}\times \Lambda_{\rho}}\left|(u)_{B_{\rho}}(t)-(u)_{B_{\rho}}^{\psi}(t)\right|^{2}\dt\dt'+c\dashiint_{\Lambda_{\rho}\times \Lambda_{\rho}}\left|(u)_{B_{\rho}}^{\psi}(t)-(u)^{\psi}_{B_{\rho}}(t')\right|^{2}\dt\dt'\\
&\quad+c\dashiint_{\Lambda_{\rho}\times \Lambda_{\rho}}\left|(u)_{B_{\rho}}^{\psi}(t')-(u)_{B_{\rho}}(t')\right|^{2}\dt\dt'\\
&\leq c\dashint_{\Lambda_{\rho}}\left|(u)_{B_{\rho}}(t)-(u)_{B_{\rho}}^{\psi}(t)\right|^{2}\dt+c\sup_{t_{1},t_{2}\in  \Lambda_{\rho}}\left|(u)_{B_{\rho}}^{\psi}(t_{1})-(u)^{\psi}_{B_{\rho}}(t_{2})\right|^{2}\eqqcolon I_{2,1}+I_{2,2},
\end{aligned}
\end{equation*}
where $\psi$ is the given function in Lemma \ref{cz non pa : GL} and 
\begin{equation*}
    (u)_{B_{\rho}}^{\psi}(t)\equiv\frac{1}{\|\psi\|_{L^{1}}}\int_{B_{\rho}}u(x,t)\psi(x)\dx.
\end{equation*}
Using the fact that $\|\psi\|_{L^{1}}\approx_{n} |B_{\rho}|$ and H\"older's inequality, we have 
\begin{equation*}
\begin{aligned}
    I_{2,1}&\leq c\dashint_{\Lambda_{\rho}}\left|\frac{1}{\|\psi\|_{L^{1}}}\int_{B_{\rho}}\left(u-(u)_{B_{\rho}}(t)\right)\dx\right|^{2}\dt\leq cI_{1}
\end{aligned}
\end{equation*}
for some constant $c=c(n,s)$. Using Lemma \ref{cz non pa : GL}, we get that
\begin{equation*}
\begin{aligned}
    I_{2,2}^{\frac{1}{2}}&\leq c\rho^{2s-1}\dashint_{Q_{\rho}}\int_{B_{\rho}}\frac{|u(x,t)-u(y,t)|}{|x-y|^{n+2s-1}}\dy\dz+c\rho^{2s}\dashint_{Q_{\rho}}\int_{\mathbb{R}^{n}\setminus B_{\rho}}\frac{|u(x,t)-u(y,t)|}{|y|^{n+2s}}\dy\dz\\
    &\quad+c\rho^{2s-1}\dashint_{Q_{\rho}}\int_{B_{\rho}}\frac{|f(x,t)-f(y,t)|}{|x-y|^{n+s-1}}\dy\dz+c\rho^{2s}\dashint_{Q_{\rho}}\int_{\mathbb{R}^{n}\setminus B_{\rho}}\frac{|f(x,t)-f(y,t)|}{|y|^{n+s}}\dy\dz\\
    &\quad+c\dashint_{Q_{\rho}}\rho^{2s}|g|^{\gamma}\dz\eqqcolon\sum_{i=1}^{5}I_{2,2,i}
\end{aligned}
\end{equation*}
for some constant $c=c(n,s,L)$. 
We now estimate $I_{2,2,1}$ and $I_{2,2,2}$.

\noindent
\textbf{Estimate of $I_{2,2,1}$.} Using H\"older's inequality and \eqref{cz non pa : size of b mut}, we have
{\small\begin{equation*}
\begin{aligned}
    I_{2,2,1}&\leq c\rho^{2s-1}\left(\dashint_{Q_{\rho}}\int_{B_{\rho}}\frac{|D^{\tau_{0}}d_{s}u|^{\gamma}}{|x-y|^{n-2\tau}}\dx\dy\dt\right)^{\frac{1}{\gamma}} \left(\dashint_{Q_{\rho}}\int_{B_{\rho}}\frac{\dx\dy\dt}{|x-y|^{n+\gamma'\left(s-1+\tau_{0}\left(\frac{2}{\gamma}-1\right)\right)}}\right)^{\frac{1}{\gamma'}}\\
    &\leq \frac{c\rho^{s+\tau_{0}}}{1-(s+\tau_{0})}\left(\frac{1}{\tau_{0}}\dashint_{\mathcal{Q}_{\rho}}|D^{\tau_{0}}d_{s}u|^{\gamma}\,d\mu_{\tau_{0},t}\right)^{\frac{1}{\gamma}}
\end{aligned}
\end{equation*}}\\
\noindent
\textbf{Estimate of $I_{2,2,2}$.} A simple algebraic computation yields that
\begin{equation*}
\begin{aligned}
    I_{2,2,2}&\leq c\rho^{2s}\dashint_{Q_{\rho}}\int_{\mathbb{R}^{n}\setminus B_{\rho}}\frac{|u(x,t)-(u)_{B_{\rho}}(t)|}{|y|^{n+2s}}\dy\dz+c\rho^{2s}\dashint_{Q_{\rho}}\int_{\mathbb{R}^{n}\setminus B_{\rho}}\frac{|(u)_{B_{\rho}}(t)-u(y,t)|}{|y|^{n+2s}}\dy\dz\\
    &\leq cI_{1}^{\frac{1}{2}}+ c\Tail_{1,s}\left(u-(u)_{B_{\rho}(x_{0})}(t);Q_{\rho}(z_{0})\right).
\end{aligned}
\end{equation*}
Similarly, we estimate $I_{2,2,3}$ and $I_{2,2,4}$ as
\begin{equation*}
    I_{2,2,3}+I_{2,2,4}\leq c\frac{\rho^{s+\tau_{0}}}{1-s}\left(\frac{1}{\tau_{0}}\dashint_{\mathcal{Q}_{\rho}(z_{0})}|D^{\tau_{0}}d_{0}f|^{2}\,d\mu_{\tau_{0},t}\right)^{\frac{1}{2}}+c\rho^{s}\Tail_{1,\frac{s}{2}}\left(f-(f)_{B_{\rho}(x_{0})}(t);Q_{\rho}(z_{0})\right).
\end{equation*}
Using H\"older's inequality, \eqref{cz non pa : size of b mut} and \eqref{cz non pa : equbwGg}, we estimate $I_{2,2,5}$ as
\begin{equation*}
    I_{2,2,5}\leq c\left(\frac{1}{\tau_{0}}\dashint_{\mathcal{Q}_{2\rho}(z_{0})}\left((2\rho)^{s-\tau_{0}}|G|\right)^{\gamma}\,d\mu_{\tau_{0},t}\right)^{\frac{2}{\gamma}}.
\end{equation*}
Take $\vartheta=\frac{\varsigma}{c}$ for some constant $c=c(n,s,L)$ and combine all the above estimates to get \eqref{cz non pa : paraSPI}.
\end{proof}
\begin{rmk}
By estimating $I_{1}$ in the proof of Lemma \ref{cz non pa : sobolevpoincare} as
\begin{equation*}
    I_{1}\leq \frac{c}{\tau_{0}}\dashint_{\mathcal{Q}_{\rho}(z_{0})}|D^{\tau_{0}}d_{s}u|^{2}\,d\mu_{\tau_{0},t},
\end{equation*} we find that for any $\tau_{0}\in (0,\min\{s,1-s\})$, there holds
{\small\begin{equation}
\label{cz non pa : paraSPI2}
\begin{aligned}
    \dashint_{Q_{\rho}(z_{0})}\frac{|u-(u)_{Q_{\rho}(z_{0})}|^{2}}{\rho^{2s+2\tau_{0}}}\dz &\leq \frac{c}{\tau_{0}}\dashint_{\mathcal{Q}_{\rho}(z_{0})}|D^{\tau_{0}}d_{s}u|^{2}\,d\mu_{\tau_{0},t}+c\Tail_{1,2s}\left(\frac{u-(u)_{B_{\rho}(x_{0})}(t)}{\rho^{s+\tau_{0}}};Q_{\rho}(z_{0})\right)^{2}
    \\
    &\quad+c\frac{1}{\tau_{0}}\dashint_{\mathcal{Q}_{\rho}(z_{0})}|D^{\tau_{0}}d_{0}f|^{2}\,d\mu_{\tau_{0},t} +c\Tail_{1,s}\left(\frac{f-(f)_{B_{\rho}(x_{0})}(t)}{\rho^{\tau_{0}}};Q_{\rho}(z_{0})\right)^{2}\\
    &\quad+c\left(\frac{1}{\tau_{0}}\dashint_{\mathcal{Q}_{2\rho}(z_{0})}\left((2\rho)^{s-\tau_{0}}|G|\right)^{\gamma}\,d\mu_{\tau_{0},t}\right)^{\frac{2}{\gamma}}
\end{aligned}
\end{equation}}
for some constant $c=c(n,s,L)$. 
\end{rmk}
In light of Lemma \ref{cz non pa : sobolevpoincare}, we now prove the following reverse H\"older's inequality.
\begin{lem}
\label{cz non pa : RHIs}
Let $u$ be a weak solution to \eqref{cz non pa : eq1} with $f=g=0$ and let $Q_{2\rho}(z_{0})\Subset\Omega_{T}$. Then we have
{\small\begin{equation*}
\begin{aligned}
E_{2,\tau_{0}}\left(u\,;\,Q_{\rho}(z_{0})\right)^{2}&\leq   c_{0}\left(\frac{1}{\tau_{0}}\dashint_{\mathcal{Q}_{2\rho}(z_{0})}|D^{\tau_{0}}d_{s}u|^{\gamma}\,d\mu_{\tau_{0},t}\right)^{\frac{2}{\gamma}}+c_{0}\Tail_{\gamma,2s}\left(\frac{|u-(u)_{B_{2\rho}(x_{0})}(t)|}{(2\rho)^{s+\tau_{0}}};Q_{2\rho}(z_{0})\right)^{2},
\end{aligned}
\end{equation*}}\\
\noindent
where $c_{0}=c_{0}(n,s,L)$ and $E_{2,\tau_{0}}(\cdot)$ is defined in \eqref{cz non pa : peneftnlofu}.
\end{lem}

\begin{proof}
From \eqref{cz non pa : size of b mut} and \eqref{cz non pa : CE} with $k=(u)_{Q_{r}(z_{0})}$, $f=0$ and $g=0$, we deduce that
{\small\begin{equation*}
\begin{aligned}
    E_{2,\tau_{0}}(u\,;\,Q_{\rho})&\leq \frac{cr^{n+2}}{(r-\rho)^{n+2}}\dashiint_{Q_{r}(z_{0})}\frac{|u-(u)_{Q_{r}(z_{0})}|^{2}}{r^{2s+2\tau_{0}}}\dz+\frac{cr^{2n+4s}}{(r-\rho)^{2n+4s}}\Tail_{\gamma,2s}\left(\frac{u-(u)_{B_{r}(x_{0})}(t)}{r^{s+\tau_{0}}};Q_{r}(z_{0})\right)^{2},
\end{aligned}
\end{equation*}}
\\
\noindent
where $c=c(n,s,L)$. Using the estimate \eqref{cz non pa : paraSPI} with $\rho=r$, $f=0$ and $g=0$ into the first term in the right-hand side of the above inequality and  H\"older's inequality, we get
{\small\begin{equation*}
\begin{aligned}
    E_{2,\tau_{0}}(u\,;\,Q_{\rho})&\leq \frac{1}{4}E_{2,\tau_{0}}(u\,;\,Q_{r})+\frac{cr^{n+2}}{(r-\rho)^{n+2}}\left(\frac{1}{\tau_{0}}\dashint_{\mathcal{Q}_{2\rho}(z_{0})}|D^{\tau_{0}}d_{s}u|^{\gamma}\,d\mu_{\tau_{0},t}\right)^{\frac{2}{\gamma}}\\
    &\quad+c\left(\frac{r}{r-\rho}\right)^{2(n+2s)}\Tail_{\gamma,2s}\left(\frac{u-(u)_{B_{2\rho}(x_{0})}(t)}{(2\rho)^{s+\tau_{0}}};Q_{2\rho}(z_{0})\right)^{2}.
\end{aligned}
\end{equation*}}
\\
\noindent
by taking $\varsigma\in(0,1)$ sufficiently small depending only on $n,s$ and $L$. Additionally, for the tail term, we have used Lemma \ref{cz non pa : taillem}. Finally, employing  Lemma \ref{cz non pa : technicallemma}, we obtain the desired estimate \eqref{cz non pa : reversehig}.    
\end{proof}

We next prove the following covering lemma.
\begin{lem}
\label{cz non pa : coveringLs}
Let $1\leq r_{1}<r_{2}\leq 2$ and $u$ be a weak solution to \eqref{cz non pa : localized pb} with $f=g=0$. Then, there are two families of countable disjoint cylinders $\left\{\mathcal{Q}_{\rho_{{i}}}(z_{i})\right\}_{i\geq0}$ and $\left\{\mathcal{Q}_{\widetilde{r}_{j}}(x_{1,j},x_{2,j},t_{0,j})\right\}_{j\geq0}$, such that
\begin{equation}
\label{cz non pa : levelsets}
U_{\lambda}=\left\{(x,y,t)\in \mathcal{Q}_{r_{1}}\,:\, |D^{\tau_{0}}d_{s}u(x,y,t)|\geq\lambda \right\}\subset\left(\bigcup\limits_{i}\mathcal{Q}_{5^{\frac{2}{s}}\rho_{{i}}}\left(z_{i}\right)\right)\bigcup \left(\bigcup\limits_{j}\mathcal{Q}_{5^{\frac{1}{s}}\widetilde{r}_{{j}}}\left(x_{1,j},x_{2,j},t_{0,j}\right)\right)
\end{equation}
whenever $\lambda\geq\lambda_{0}$, where
{\small\begin{equation}
\label{cz non pa : lambda0s}
\begin{aligned}
    \lambda_{0}&\coloneqq \frac{c}{(r_{2}-r_{1})^{\frac{5n}{s}}}\left[\left(\frac{1}{\tau_{0}}\dashint_{\mathcal{Q}_{2}}|D^{\tau_{0}}d_{s}u|^{2}\,d\mu_{\tau_{0},t}\right)^{\frac{1}{2}}+\left(\sup_{t\in \Lambda_{2}}\dashint_{B_{2}}\frac{|u-(u)_{Q_{2}}|^{2}}{2^{2s+2\tau_{0}}}\dx\right)^{\frac{1}{2}}\right]\\
    &\quad+\frac{c}{(r_{2}-r_{1})^{\frac{5n}{s}}}\Tail_{\infty,2s}\left(\frac{u-(u)_{B_{2}}(t)}{2^{s+\tau_{0}}};Q_{2}\right)
\end{aligned}
\end{equation}}\\
\noindent
for some constant $c=c(n,s,L)$. In particular, we have 
\begin{equation}
\label{cz non pa : diagonal estimates}
\begin{aligned}
&\sum_{i\geq0}\mu_{\tau_{0},t}\left(\mathcal{Q}_{\rho_{{i}}}(z_{i})\right)\leq\frac{c}{\lambda^{\gamma}}\int_{\mathcal{Q}_{r_{2}}\cap\{|D^{\tau_{0}}d_{s}u|>b_{u}\lambda\}}|D^{\tau_{0}}d_{s}u|^{\gamma}\,d\mu_{\tau_{0},t},
\end{aligned}
\end{equation}

\begin{equation}
\label{cz non pa : offdiagonal estimates}
\begin{aligned}
&\sum_{j\geq0}\mu_{\tau_{0},t}\left(\mathcal{Q}_{\overline{r}_{{j}}}\left(x_{1,j},x_{2,j},t_{0,j}\right)\right)\leq\frac{c}{\lambda^{2}}\int_{\mathcal{Q}_{r_{2}}\cap\{|D^{\tau_{0}}d_{s}u|>\frac{\lambda}{16}\}}|D^{\tau_{0}}d_{s}u|^{2}\,d\mu_{\tau_{0},t}++\sum_{i\geq0}\mu_{\tau_{0},t}\left(\mathcal{Q}_{\rho_{{i}}}(z_{i})\right)
\end{aligned}
\end{equation}
for some constant $b_{u}=b_{u}(n,s,L)\in(0,1]$, where the constant $\gamma$ is defined in \eqref{cz non pa : gamma}, and we also have
\begin{equation}
\label{cz non pa : pstra norms}
\left(\dashint_{\mathcal{Q}_{5^{\frac1s}\widetilde{r}_{{j}}}\left(x_{1,j},x_{2,j},t_{0,j}\right)}|D^{\tau_{0}}d_{s}u|^{2_{\#}}\dmut\right)^{\frac{1}{{2_{\#}}}}\leq c\lambda\quad\text{for any }j,
\end{equation}
where the constant $2_{\#}$ is given in \eqref{cz non pa : paraconj} with $p=2$.
\end{lem}

\begin{proof}
We first define a functional
\begin{equation*}
\begin{aligned}
    \Theta_{D}\left(z_{0},r\right)=\left(\dashint_{\mathcal{Q}_{r}(z_{0})}|D^{\tau_{0}}d_{s}u|^{2}\,d\mu_{\tau_{0},t}\right)^{\frac{1}{2}}+\left(\tau_{0}\sup_{t\in \Lambda_{r}(t_{0})}\dashint_{B_{r}(x_{0})}\frac{|u-(u)_{Q_{r}(z_{0})}|^{2}}{r^{2s+2\tau_{0}}}\dx\right)^{\frac{1}{2}}
\end{aligned}
\end{equation*}
for any $z_{0}\in Q_{r_{1}}$ and $r>0$ with $Q_{r}(z_{0})\subset Q_{2}$. Let us take 
{\small\begin{equation*}
\begin{aligned}
    \lambda_{0}&=\frac{M\tau_{0}^{\frac{1}{2}}\kappa^{-1}}{(r_{2}-r_{1})^{\frac{5n}{s}}}\left(\left(\frac{1}{\tau_{0}}\dashint_{\mathcal{Q}_{2}}|D^{\tau_{0}}d_{s}u|^{2}\,d\mu_{\tau_{0},t}\right)^{\frac{1}{2}}+\Tail_{\infty,2s}\left(\frac{u-(u)_{B_{2}}(t)}{2^{s+\tau_{0}}};Q_{2}\right)\right)\\
    &\quad+\frac{M\tau_{0}^{\frac{1}{2}}\kappa^{-1}}{(r_{2}-r_{1})^{\frac{5n}{s}}}\left(\sup_{t\in \Lambda_{2}}\dashint_{B_{2}}\frac{|u-(u)_{Q_{2}}|^{2}}{2^{2s+2\tau_{0}}}\dx\right)^{\frac{1}{2}},
\end{aligned}
\end{equation*}}\\
\noindent
where $M\geq1$ and $\kappa\in(0,1]$ are free parameters which will be determined later. We next take a positive integer $j_{0}=j_{0}(n,s,\Lambda)$ such that
\begin{equation}
\label{cz non pa : j0conds}
\frac{16(c_{0}+\tilde{c}+2c_{2})^{2}}{1-2^{-s+\tau_{0}}}\leq 2^{j_{0}\left(s-\tau_{0}\right)},
\end{equation}
where $c_{0}$ is the constant determined in Lemma \ref{cz non pa : RHIs}, and $\tilde{c}$ and $c_{2}$ are the constants determined in \eqref{cz non pa : tailestimate}.  
We then note that for any $z_{0}\in Q_{r_{1}}$,
\begin{equation*}
    Q_{5^{\frac{2}{s}}\times 2^{j_{0}+2}\mathcal{R}_{1,2}}(z_{0})\subset Q_{r_{2}},
\end{equation*}
where $\mathcal{R}_{1,2}$ is defined in \eqref{cz non pa : r1r2defn}.
Let us now define for $\lambda\geq\lambda_{0}$,
\begin{equation*}
    D_{\kappa\lambda}=\left\{z_{0}\in Q_{r_{1}}\,:\,\sup_{0<\rho\leq \mathcal{R}_{1,2}}\Theta_{D}\left(z_{0},\rho\right)>\kappa\lambda\right\}.
\end{equation*}
We now take the constant $M$ as in \eqref{cz non pa : condM1} with $q=2$. For any $z_{0}\in Q_{r_{1}}$ and $r\in \left[\mathcal{R}_{1,2},5^{\frac{2}{s}}\times 2^{j_{0}+3}\mathcal{R}_{1,2}\right]$, we first note that
\begin{equation*}
    \Theta_{D}(z_{0},r)\leq \kappa\lambda,
\end{equation*}
by following the same line as in the proof for \eqref{cz non pa : thetales} with $p=2$.
Therefore, there is an exit radius $\rho_{z}$ for each $z\in D_{\kappa\lambda}$ such that
\begin{equation}
\label{cz non pa : exitradiuss}
    \Theta_{D}\left(z,\rho_{z}\right)\geq\kappa\lambda\quad\text{and}\quad\Theta_{D}\left(z,\rho\right)\leq\kappa\lambda\quad\text{ if }\rho_{z}\leq\rho\leq 5^{\frac{2}{s}}\times 2^{j_{0}+3}\mathcal{R}_{1,2}.
\end{equation}
We now apply Vitali's covering lemma to find a family of mutually disjoint countable cylinders
\begin{equation}
\label{cz non pa : disjointcydis}
     \left\{Q_{2^{j_{0}}\rho_{z_{i}}}(z_{i})\right\}_{i\geq0}\text{ such that } D_{\kappa\lambda}\subset\bigcup\limits_{i=0}^{\infty}Q_{5^{\frac{1}{s}}\times 2^{j_{0}}\rho_{z_{i}}}(z_{i}).
\end{equation}
We now denote 
\begin{equation}
\label{cz non pa : rhoidefns}
    \rho_{i}=2^{j_{0}}\rho_{z_{i}}\quad\text{for each }i.
\end{equation}
From \eqref{cz non pa : exitradiuss}, we have
\begin{equation}
\label{cz non pa : kalams}
\begin{aligned}
    \frac{\kappa\lambda}{\tau_{0}^{\frac{1}{2}}}\leq \left(\frac{1}{\tau_{0}}\dashint_{\mathcal{Q}_{\rho_{z_{i}}}(z_{i})}|D^{\tau_{0}}d_{s}u|^{2}\,d\mu_{\tau_{0},t}\right)^{\frac{1}{2}}+\left(\sup_{t\in \Lambda_{\rho_{z_{i}}}(t_{i})}\dashint_{B_{\rho_{z_{i}}}(x_{i})}\frac{|u-(u)_{Q_{\rho_{z_{i}}}(z_{i})}|^{2}}{\rho_{z_{i}}^{2s+2\tau_{0}}}\dx\right)^{\frac{1}{2}}.
\end{aligned}
\end{equation}
On account of Lemma \ref{cz non pa : RHIs}, Lemma \ref{cz non pa : tailestimate} and \eqref{cz non pa : rhoidefns}, we estimate the right-hand side of \eqref{cz non pa : kalams} as 
{\small
\begin{align*}
    \frac{\kappa\lambda}{\tau_{0}^{\frac{1}{2}}}
    &\leq c\left(\frac{1}{\tau_{0}}\dashint_{\mathcal{Q}_{\rho_{i}}(z_{i})}|D^{\tau_{0}}d_{s}u|^{\gamma}\,d\mu_{\tau_{0},t}\right)^{\frac{1}{\gamma}}+ c_{0}\tilde{c}\frac{1}{\frac{1}{\tau_{0}}^{\frac{1}{2}}}\sum_{j=j_{0}+1}^{l}2^{i\left(-s+\tau_{0}\right)}\Theta_{D}\left(z_{j},2^{j}\rho_{z_{i}}\right)\\
    &\quad
    +c_{0}\tilde{c}\left(\frac{2}{2^{j_{0}}\mathcal{R}_{1,2}}\right)^{s+\tau_{0}}\left(\sup_{t\in \Lambda_{2}}\dashint_{B_{2}}\frac{|u-(u)_{B_{2}}(t)|^{2}}{2^{2s+2\tau_{0}}}\dx\right)^{\frac{1}{2}}+ \frac{c_{0}\tilde{c}}{(r_{2}-r_{1})^{5n}}\Tail_{\gamma,2s}\left(\frac{u-(u)_{B_{2}}(t)}{2^{s+\tau_{0}}};Q_{2}\right)\nonumber
\end{align*}
}\\
\noindent
where $c=c(n,s,L)$ and $l$ is the positive integer such that $
    2^{j_{0}+1}\mathcal{R}_{1,2}\leq 2^{l}\rho_{z_{i}}<2^{j_{0}+2}\mathcal{R}_{1,2}.$
For the detailed calculations of the above inequality, we refer to \eqref{cz non pa : findj0u} and \eqref{cz non pa : oneretildeqnec} with $f=0$ and $g=0$.
As a result, using \eqref{cz non pa : j0conds} and \eqref{cz non pa : exitradiuss}, we find that
\begin{equation*}
    \frac{\kappa\lambda}{{\tau_{0}}^{\frac{1}{2}}}
    \leq c\left(\frac{1}{\tau_{0}}\dashint_{\mathcal{Q}_{\rho_{i}}(z_{i})}|D^{\tau_{0}}d_{s}u|^{\gamma}\,d\mu_{\tau_{0},t}\right)^{\frac{1}{\gamma}}
\end{equation*}
for some constant $c=c(n,s,L)$. A suitable choice of the constant $\tilde{b}_{u}=\tilde{b}_{u}(n,s,L)\in\left(0,\frac{1}{8}\right]$ yields
\begin{equation}
\label{cz non pa : diabs}
\mu_{\tau_{0},t}\left(\mathcal{Q}_{\rho_{i}}(z_{i})\right)\leq \frac{c}{(\kappa\lambda)^{\gamma}}\int_{\mathcal{Q}_{\rho_{i}}(z_{i})\cap\{|D^{\tau_{0}}d_{s}u|>\tilde{b}_{u}\kappa\lambda\}}|D^{\tau_{0}}d_{s}u|^{\gamma}\,d\mu_{\tau_{0},t},
\end{equation}
as the constant $\tau_{0}$ depends only on $s$ (see \eqref{cz non pa : appentau}). 
With \eqref{cz non pa : exitradiuss} and \eqref{cz non pa : disjointcydis}, we follow the same lines as in the proof of step 2 through step 8 given in Lemma \ref{cz non pa : coveringL} with $p=2$, $f=0$ and $g=0$. As a result, by taking $\kappa$ as in \eqref{cz non pa : rangeofkappa1} with $\tau^{\frac{1}{\gamma}}$ replaced by $\tau_{0}^{\frac{1}{2}}$, we find that there is a collection of countable disjoint cylinders 
$\{\mathcal{Q}\}_{\mathcal{Q}\in \mathcal{A}}$ such that
\begin{equation*}
    \left\{(x,y,t)\in \mathcal{Q}_{r_{1}}\,:\, |D^{\tau_{0}}d_{s}u(x,y,t)|\geq\lambda \right\}\subset\left(\bigcup\limits_{i\geq0}\mathcal{Q}_{5^{\frac{2}{s}}\rho_{{i}}}(z_{i})\right)\bigcup \left(\bigcup\limits_{j\geq0}\mathcal{Q}_{5^{\frac{1}{s}}\widetilde{r}_{j}}(x_{1,j},x_{2,j},t_{0,j})\right),
\end{equation*}
\begin{equation*}
\begin{aligned}
&\sum_{j}\mu_{\tau_{0},t}\left(\mathcal{Q}_{\overline{r}_{j}}(x_{1,j},x_{2,j},t_{0,j})\right)\leq\frac{c}{\lambda^{2}}\int_{\mathcal{Q}_{r_{2}}\cap\left\{|D^{\tau_{0}}d_{s}u|>\frac{\lambda}{16}\right\}}|D^{\tau_{0}}d_{s}u|^{2}\,d\mu_{\tau_{0},t}+c\sum_{i}\mu_{\tau_{0},t}\left(\mathcal{Q}_{\rho_{i}}(z_{i})\right)
\end{aligned}
\end{equation*}
and
\begin{equation*}
\left(\dashint_{\mathcal{Q}_{5^{\frac1s}\widetilde{r}_{j}}(x_{1,j},x_{2,j},t_{0,j})}|D^{\tau_{0}}d_{s}u|^{2_{\#}}\dmut\right)^{\frac{1}{{2_{\#}}}}\leq c\lambda\quad\text{for any }j,
\end{equation*}
where $c=c(n,s,L)$ (see \eqref{cz non pa : sumofmeasqnd}and \eqref{cz non pa : reverseestimate} for the second and the third inequalities, respectively). 
Let $b_{u}=\tilde{b}_{u}\kappa$. Then the above three observations and \eqref{cz non pa : diabs} yield the desired results as the constant $\tau_{0}$ depends only on $s$. This completes the proof.
\end{proof}

We are now in the position to prove the following self-improving property for a weak solution to the corresponding homogeneous problem of \eqref{cz non pa : eq1}.
\begin{thm}
\label{cz non pa : selfprof}
Let $u$ be a weak solution to \eqref{cz non pa : eq1} with $f=0$ and $g=0$. Then there are constants $\epsilon=\epsilon(n,s,L)\in(0,1)$ and $c=c(n,s,L)$ such that
{\small\begin{equation}
\label{cz non pa : selfest}
\begin{aligned}
    \left(\frac{1}{\tau_{0}}\dashint_{\mathcal{Q}_{r}(z_{0})}|D^{\tau_{0}}d_{s}\tilde{u}|^{2+\epsilon}\,d\mu_{\tau_{0},t}\right)^{\frac{1}{2+\epsilon}}&\leq c\left(\frac{1}{\tau_{0}}\dashint_{\mathcal{Q}_{2r}(z_{0})}|D^{\tau_{0}}d_{s}{u}|^{2}d\mu_{\tau_{0},t}\right)^{\frac{1}{2}}+c\sup_{t\in \Lambda_{2r}}\left(\dashint_{B_{2r}}\frac{|{u}-({u})_{Q_{2r}(z_{0})}|^{2}}{(2r)^{2s+2\tau}}\right)^{\frac{1}{2}}\\
    &+c\Tail_{\infty,2s}\left(\frac{{u}-({u})_{B_{2r}(x_{0})}(t)}{(2r)^{s+\tau}};Q_{2r}(z_{0})\right)
\end{aligned}
\end{equation}}
whenever $Q_{2r}(z_{0})\Subset\Omega_{T}$.
\end{thm}
\begin{proof}
Let us fix $Q_{r}(z_{0})\Subset\Omega_{T}$. We now define for any $x,y\in\mathbb{R}^{n}$, $t\in \Lambda_{2}$ and $\xi\in\mathbb{R}$,
{\small\begin{equation*}
    \tilde{u}(x,t)=\frac{u\left({r}x+x_{0},{r}^{2s}t+t_{0}\right)}{{r}^{s+\tau}},\,\tilde{A}(x,y,t)=A\left({r}x+x_{0},{r}y+x_{0},{r}^{2s}t+t_{0}\right)\text{ and }\tilde{\Phi}(\xi)=\frac{\tilde{\Phi}\left({r}^{\tau}\xi\right)}{{r}^{\tau}}
\end{equation*}}\\
\noindent
to see that $\tilde{u}$ is a weak solution to \eqref{cz non pa : localized pb} with $f=g=0$, $A=\tilde{A}$ and $\Phi=\tilde{\Phi}$.
Let us take $\epsilon\in\left(0,\frac{2_{\#}-2}{2}\right)$ which will be determined later. For each $N>0$, we now define $\phi_{N}:[1,2]\to\mathbb{R}$ by
\begin{equation*}
    \phi_{N}(\rho)=\left(\dashint_{Q_{\rho}}|D^{\tau_{0}}d_{s}\tilde{u}|_{N}^{2+\epsilon}d\mu_{\tau_{0},t}\right)^{\frac{1}{2+\epsilon}}.
\end{equation*}    
For $\lambda_{0}$ as defined in \eqref{cz non pa : lambda0s} with $u=\tilde{u}$, we claim that if $N>\lambda_{0}$, then there is a constant $c=c(n,s,L)$ such that for any $1\leq r_{1}<r_{2}\leq2$,
\begin{equation}
\label{cz non pa : goalbt}
    \phi_{N}(r_{1})\leq \frac{1}{2}\phi_{N}(r_{2})+c\lambda_{0}.
\end{equation}
By Fubini's theorem, we observe that
\begin{equation*}
\begin{aligned}
    \int_{\mathcal{Q}_{r_{1}}}|D^{\tau_{0}}d_{s}\tilde{u}|_{N}^{2+\epsilon}d\mu_{\tau_{0},t}
    &=\epsilon\int_{0}^{\mathcal{M}\lambda_{0}}\lambda^{\epsilon-1}\nu\left(\left\{\mathcal{Q}_{r_{1}}\,:\,|D^{\tau_{0}}d_{s}\tilde{u}|_{N}>\lambda\right\}\right)\,d\lambda\\
    &\quad+\epsilon\int_{\mathcal{M}\lambda_{0}}^{N}\lambda^{\epsilon-1}\nu\left(\left\{\mathcal{Q}_{r_{1}}\,:\,|D^{\tau_{0}}d_{s}\tilde{u}|_{N}>\lambda\right\}\right)\,d\lambda\eqqcolon I+J,
\end{aligned}
\end{equation*}
where $d\,\nu=|D^{\tau_{0}}d_{s}\tilde{u}|^{2}\,d\mu_{\tau_{0},t}$ and $N>\mathcal{M}\lambda_{0}$ with $\mathcal{M}>1$ to be determined later. We first estimate $I$ as
\begin{equation*}
\begin{aligned}
    I\leq c\mathcal{M}^{\epsilon}\lambda_{0}^{\epsilon}\int_{\mathcal{Q}_{2}}|D^{\tau_{0}}d_{s}\tilde{u}|^{2}d\mu_{\tau_{0},t}\leq c\mathcal{M}^{\epsilon}\lambda_{0}^{2+\epsilon}\mu_{\tau_{0},t}\left(\mathcal{Q}_{2}\right),
\end{aligned}
\end{equation*}
where $c=c(n,s,L)$. We next estimate $J$ as 
{\small\begin{equation*}
\begin{aligned}
    J&= \epsilon\int_{\lambda_{0}}^{N\mathcal{M}^{-1}}\mathcal{M}^{\epsilon}\lambda^{\epsilon-1}\nu\left(\left\{(x,y,t)\in\mathcal{Q}_{r_{1}}\,:\,|D^{\tau_{0}}d_{s}\tilde{u}|_{N}>\mathcal{M}\lambda\right\}\right)\,d\lambda \\
&\leq \sum_{i\geq0}\epsilon\int_{\lambda_{0}}^{N\mathcal{M}^{-1}}\mathcal{M}^{\epsilon}\lambda^{\epsilon-1}\nu\left(\mathcal{Q}_{5^{\frac{2}{s}}
\rho_{{i}}}\left(z_{i}\right)\right)\,d\lambda\\
&\quad+\sum_{j\geq0}\epsilon\int_{\lambda_{0}}^{N\mathcal{M}^{-1}}\mathcal{M}^{\epsilon}\lambda^{\epsilon-1}\nu\left(\left\{\mathcal{Q}_{5^{\frac{1}{s}}\overline{r}_{j}}(x_{1,j},x_{2,j},t_{0,j})\,:\,|D^{\tau_{0}}d_{s}\tilde{u}|_{N}>\mathcal{M}\lambda\right\}\right)\,d\lambda\eqqcolon J_{1}+J_{2},
\end{aligned}
\end{equation*}}\\
\noindent
where we have used the change of variables and \eqref{cz non pa : levelsets}. In light of the definition of the measure $\nu$, \eqref{cz non pa : exitradiuss}, \eqref{cz non pa : doubling} and \eqref{cz non pa : diagonal estimates}, we estimate $J_{1}$ as
\begin{equation*}
\begin{aligned}
    J_{1}&\leq \sum_{i}\epsilon\int_{\lambda_{0}}^{N\mathcal{M}^{-1}}\mathcal{M}^{\epsilon}\lambda^{\epsilon-1}\int_{\mathcal{Q}_{5^{\frac{2}{s}}
\rho_{{i}}}\left(z_{i}\right)}|D^{\tau_{0}}d_{s}\tilde{u}|^{2}d\mu_{\tau_{0},t}\,d\lambda\\
&\leq c\sum_{i}\epsilon\int_{\lambda_{0}}^{N\mathcal{M}^{-1}}\mathcal{M}^{\epsilon}\lambda^{\epsilon-1}\lambda^{2}\mu_{\tau_{0},t}\left(\mathcal{Q}_{
\rho_{{i}}}\left(z_{i}\right)\right)\,d\lambda\\
&\leq \frac{c\epsilon \mathcal{M}^{\epsilon}}{\lambda^{\gamma-(1+\epsilon)}}\int_{\lambda_{0}}^{N\mathcal{M}^{-1}}\int_{\mathcal{Q}_{r_{2}}\cap\{|D^{\tau_{0}}d_{s}\tilde{u}|\geq b_{u}\lambda\}}|D^{\tau_{0}}d_{s}\tilde{u}|^{\gamma}d\mu_{\tau_{0},t}\,d\lambda,
\end{aligned}
\end{equation*}
where $c=c(n,s,L)$ and the constant $b_{u}$ is determined in Lemma \ref{cz non pa : coveringLs}. Using Fubini's theorem, we get
\begin{equation*}
    J_{1}\leq c\epsilon \mathcal{M}^{\epsilon}\int_{\mathcal{Q}_{r_{2}}}|D^{\tau_{0}}d_{s}\tilde{u}|_{N\mathcal{M}^{-1}}^{2+\epsilon}d\mu_{\tau_{0},t},
\end{equation*}
where $c=c(n,s,L)$, as $\frac{1}{2-\gamma}>0$ depends only on $n$ and $s$.
To estimate $J_{2}$, we first note from the weak 1-1 estimate and \eqref{cz non pa : pstra norms} that
{\small\begin{equation*}
\begin{aligned}
    \int_{\lambda_{0}}^{N\mathcal{M}^{-1}}\mathcal{M}^{\epsilon}\lambda^{\epsilon-1}\nu\left(\left\{\mathcal{Q}\,:\,|D^{\tau_{0}}d_{s}\tilde{u}|_{N}>\mathcal{M}\lambda\right\}\right)\,d\lambda&\leq c \int_{\lambda_{0}}^{N\mathcal{M}^{-1}}\mathcal{M}^{\epsilon}\lambda^{\epsilon-1}\int_{\mathcal{Q}}\frac{|D^{\tau}d_{s}\tilde{u}|_{N}^{2_{\#}-2}}{(\mathcal{M}\lambda)^{2_{\#}-2}}|D^{\tau_{0}}d_{s}u|^{2}d\mu_{\tau_{0},t}\,d\lambda\\
    &\leq c\int_{\lambda_{0}}^{N\mathcal{M}^{-1}}\mathcal{M}^{\epsilon+2-2_{\#}}\lambda^{\epsilon+1}\mu_{\tau_{0},t}\left(\mathcal{Q}\right)\,d\lambda,
\end{aligned}
\end{equation*}}\\
\noindent
where we denote $\mathcal{Q}=\mathcal{Q}_{5^{\frac{1}{s}}\overline{r}_{j}}(x_{1,j},x_{2,j},t_{0,j})$. We first note 
{\small\begin{equation*}
\begin{aligned}
    \frac{1}{\lambda^{\gamma}}\int_{\lambda_{0}}^{N\mathcal{M}^{-1}}\int_{\mathcal{Q}_{r_{2}}\cap\{|D^{\tau_{0}}d_{s}\tilde{u}|\geq b_{u}\lambda\}}|D^{\tau_{0}}d_{s}\tilde{u}|^{\gamma}d\mu_{\tau_{0},t}\,d\lambda&\leq \frac{\lambda_{0}^{2-\gamma}}{\lambda^{2}}\int_{\lambda_{0}}^{N\mathcal{M}^{-1}}\int_{\mathcal{Q}_{r_{2}}\cap\{|D^{\tau_{0}}d_{s}\tilde{u}|\geq b_{u}\lambda\}}|D^{\tau_{0}}d_{s}\tilde{u}|^{\gamma}d\mu_{\tau_{0},t}\,d\lambda\\
    &\leq\frac{c}{\lambda^{2}}\int_{\lambda_{0}}^{N\mathcal{M}^{-1}}\int_{\mathcal{Q}_{r_{2}}\cap\{|D^{\tau_{0}}d_{s}\tilde{u}|\geq b_{u}\lambda\}}|D^{\tau_{0}}d_{s}\tilde{u}|^{2}d\mu_{\tau_{0},t}\,d\lambda
\end{aligned}
\end{equation*}}\\
\noindent
for some constant $c=c(n,s,L)$ as $\gamma<2$ and  the constant $b_{u}$ depends only on $n,s$ and $L$.
Considering the above two inequalities, \eqref{cz non pa : offdiagonal estimates} and the estimate $J_{1}$, we estimate $J_{2}$ as
{\small\begin{equation*}
\begin{aligned}
    J_{2}&\leq c\epsilon\int_{\lambda_{0}}^{N\mathcal{M}^{-1}}\mathcal{M}^{\epsilon+2-2_{\#}}\lambda^{\epsilon-1}\int_{\mathcal{Q}_{r_{2}}\cap\{|D^{\tau_{0}}d_{s}\tilde{u}|\geq b_{u}\lambda\}}|D^{\tau_{0}}d_{s}\tilde{u}|^{2}d\mu_{\tau_{0},t}\,d\lambda\\
    &\leq c\mathcal{M}^{\epsilon+2-2_{\#}}\int_{\mathcal{Q}_{r_{2}}}|D^{\tau_{0}}d_{s}\tilde{u}|_{N\mathcal{M}^{-1}}^{2+\epsilon}d\mu_{\tau_{0},t}\\
    &\leq\frac{1}{2^{10n}}\int_{\mathcal{Q}_{r_{2}}}|D^{\tau_{0}}d_{s}\tilde{u}|_{N\mathcal{M}^{-1}}^{2+\epsilon}d\mu_{\tau_{0},t}
\end{aligned}
\end{equation*}}\\
\noindent
by taking $\mathcal{M}=\mathcal{M}(n,s,L)>1$ sufficiently large so that $c\mathcal{M}^{\frac{2-2_{\#}}{2}}\leq \frac{1}{2^{10n}}$ (thanks to $\mathcal{M}^{\epsilon+2-2_{\#}}\leq \mathcal{M}^{\frac{2-2_{\#}}{2}}$ as $\mathcal{M}\geq1$ and $\epsilon<\frac{2_{\#}-2}{2}$). We next select $\epsilon=\epsilon(n,s,L)<1$ such that
\begin{equation*}
    J_{1}\leq \frac{1}{2^{10n}}\int_{\mathcal{Q}_{r_{2}}}|D^{\tau_{0}}d_{s}\tilde{u}|_{N\mathcal{M}^{-1}}^{2+\epsilon}d\mu_{\tau_{0},t}.
\end{equation*}
Combining all the estimates of $I$ and $J$, we observe that
\begin{equation*}
    \int_{\mathcal{Q}_{r_{1}}}|D^{\tau_{0}}d_{s}u|_{N\mathcal{M}^{-1}}^{2+\epsilon}d\mu_{\tau_{0},t}\leq\int_{\mathcal{Q}_{r_{1}}}|D^{\tau_{0}}d_{s}u|_{N}^{2+\epsilon}d\mu_{\tau_{0},t}\leq \frac{1}{2^{9n}} \int_{\mathcal{Q}_{r_{2}}}|D^{\tau_{0}}d_{s}u|_{N\mathcal{M}^{-1}}^{2+\epsilon}d\mu_{\tau_{0},t}+ c\lambda_{0}^{2+\epsilon}\mu_{\tau_{0},t}\left(\mathcal{Q}_{2}\right)
\end{equation*}
holds if $N>\mathcal{M}\lambda_{0}$. After a few algebraic computations along with \eqref{cz non pa : size of b mut}, we have \eqref{cz non pa : goalbt}. Applying Lemma \ref{cz non pa : technicallemma} to \eqref{cz non pa : goalbt}, we obtain
\begin{equation*}
\begin{aligned}
    \left(\frac{1}{\tau_{0}}\dashint_{\mathcal{Q}_{1}}|D^{\tau_{0}}d_{s}\tilde{u}|_{N}^{2+\epsilon}\,d\mu_{\tau_{0},t}\right)^{\frac{1}{2+\epsilon}}&\leq c\left(\frac{1}{\tau_{0}}\dashint_{\mathcal{Q}_{2}}|D^{\tau_{0}}d_{s}\tilde{u}|^{2}d\mu_{\tau_{0},t}\right)^{\frac{1}{2}}+c\Tail_{\infty,2s}\left(\frac{\tilde{u}-(\tilde{u})_{B_{2}}(t)}{2^{s+\tau}};Q_{2}\right)\\
    &\quad+c\sup_{t\in \Lambda_{2}}\left(\dashint_{B_{2}}\frac{|\tilde{u}-(\tilde{u})_{Q_{2}}|^{2}}{2^{2s+2\tau}}\right)^{\frac{1}{2}}
\end{aligned}
\end{equation*}
for some constant $c=c(n,s,L)$. By passing to the limit $N\to\infty$ 
and using the change of variables, we get the desired estimate \eqref{cz non pa : selfest}.
\end{proof}

\section{Existence and uniqueness}
In this section, we present the existence result for the corresponding boundary value problem of \eqref{cz non pa : eq1} and the standard energy estimate. Before stating the  result, we first note from \cite[Proposition 1.2 in Chapter 3]{Smono} with $V=W^{s,2}(\Omega)$ and $H=L^{2}(\Omega)$ that if $h\in  L^{2}(0,T;W^{s,2}(\Omega))$ and $h_{t}\in \left(L^{2}(0,T;W^{s,2}(\Omega)\right)^{*}$, then $h\in C([0,T];L^{2}(\Omega))$ with the estimate
\begin{equation}
\label{cz non pa : imbeddingt}
    \sup_{t\in[0,T]}\|h(\cdot,t)\|_{L^{2}(\Omega)}\leq c \|h\|_{L^{2}(0,T;W^{s,2}(\Omega))}+ c\|h_{t}\|_{\left(L^{2}(0,T;W^{s,2}(\Omega)\right)^{*}}
\end{equation}
for some constant $c=c(n,s)$.
\begin{lem}\label{cz non pa : exst-unq}
Let $\Omega'$ be an open and bounded set such that $\Omega\Subset\Omega'$. Suppose that 
\begin{equation}
\label{cz non pa : dataconditionofIVP}
    \begin{aligned}
    &f\in L^{2}\left(0,T;L^{1}_{s}(\mathbb{R}^{n})\right)\quad\text{with}\quad d_{0}f\in L^{2}\left(\Omega'\times\Omega'\times(0,T)\,;\,\frac{\dx\dy\dt}{|x-y|^{n}}\right), \\
    &g\in L^{\frac{2(n+2s)}{n+4s}}(\Omega_{T}),\\
    &h\in L^{2}(0,T;W^{s,2}(\Omega'))\cap L^{\infty}(0,T;L^{1}_{2s}(\mathbb{R}^{n}))\quad\text{and}\quad h_{t}\in \left(L^{2}(0,T;W^{s,2}(\Omega'))\right)^{*}.
\end{aligned}
\end{equation}
Then there is a unique weak solution $u\in L^{2}(0,T;W^{s,2}(\Omega))\cap L^{\infty}\left(0,T;L^{1}_{2s}(\mathbb{R}^{n})\right)\cap C([0,T];L^{2}(\Omega))$  to 
\begin{equation}
\label{cz non pa : IVP}
\left\{
\begin{alignedat}{3}
u_{t}+\mathcal{L}^{\Phi}_{A}u&= (-\Delta)^{\frac{s}{2}}f+g&&\qquad \mbox{in  $\Omega\times(0,T)$}, \\
u&=h&&\qquad  \mbox{in $\mathbb{R}^{n}\setminus\Omega\times[0,T]$}, \\
u(\cdot,0)&=h(\cdot,0) &&\qquad \mbox{in  $\Omega$}
\end{alignedat} \right.
\end{equation}
with the estimate
\begin{equation}
\label{cz non pa : estiivp}
\begin{aligned}
&\sup_{t\in(0,T)}\int_{\Omega}|u(x,t)|^{2}\dx+\int_{0}^{T}\int_{\Omega}\int_{\Omega}|d_{s}u|^{2}\frac{\dx\dy\dt}{|x-y|^{n}}\\
&\leq c\int_{0}^{T}\int_{\Omega'}\int_{\Omega'}|d_{0}f|^{2}\frac{\dx\dy\dt}{|x-y|^{n}}+c\Tail_{2,s}(f-(f)_{\Omega'}(t);\Omega'_{T})^{2}+c\left(\int_{0}^{T}\int_{\Omega}|g|^{\frac{2(n+2s)}{n+4s}}\dz\right)^{\frac{n+4s}{n+2s}}\\
    &\quad+c\int_{0}^{T}\int_{\Omega'}\int_{\Omega'}|d_{s}h|^{2}\frac{\dx\dy}{|x-y|^{n}}+c\Tail_{2,2s}(h-(h)_{\Omega'}(t);\Omega'_{T})^{2}+c\|h_{t}\|^{2}_{\left(L^{2}(0,T;W^{s,2}(\Omega')\right)^{*}}
\end{aligned}
\end{equation}
for some constant $c=c(n,s,L,T,\Omega,\Omega')$.
\end{lem}
\begin{proof}
From \cite[Lemma 2.7]{BKc}, we observe that
\begin{equation*}
    T_{f(\cdot,t)}:\phi\mapsto \int_{\mathbb{R}^{n}}\int_{\mathbb{R}^{n}}(f(x,t)-f(y,t))(\phi(x)-\phi(y))\frac{\dx\dy}{|x-y|^{n+s}},\quad \phi\in X_{0}^{s,2}(\Omega,\Omega') 
\end{equation*}
is an element of the dual space of $X_{0}^{s,2}(\Omega,\Omega')$. This implies that 
\begin{equation*}
    (-\Delta)^{\frac{s}{2}}f\in\left(L^{2}(0,T;W^{s,2}(\Omega')\right)^{*}.
\end{equation*}
Therefore, combining \cite[Thoerem A.3]{BLS} and \cite[Lemma A.1]{BKKh}, we find a unique weak solution $u$ to \eqref{cz non pa : IVP}. 

We are now in the position to prove \eqref{cz non pa : estiivp}. Since $u-h\in L^{2}\left(0,T;X_{0}^{s,2}(\Omega,\Omega')\right)$, using the standard approximation argument, we have 
\begin{equation}
\label{cz non pa : firesten}
\begin{aligned}
    &\sup_{t\in(0,T)}\int_{\Omega}|(u-h)(x,t)|^{2}\dx+\int_{0}^{T}\int_{\mathbb{R}^{n}}\int_{\mathbb{R}^{n}}|d_{s}(u-h)|^{2}\frac{\dx\dy\dt}{|x-y|^{n}}\\
    &\leq c\left|\int_{0}^{T}\langle h_{t}, u-h\rangle_{X_{0}^{s,2}(\Omega,\Omega') ,\left(X_{0}^{s,2}(\Omega,\Omega')\right)^{*} }\dt\right|+c\int_{0}^{T}\int_{\Omega}|g(u-h)|\dz\\
    &\quad+c\int_{0}^{T}\int_{\mathbb{R}^{n}}\int_{\mathbb{R}^{n}}|d_{0}f||d_{s}(u-h)|\frac{\dx\dy\dt}{|x-y|^{n}}+c\int_{0}^{T}\int_{\mathbb{R}^{n}}\int_{\mathbb{R}^{n}}|d_{s}h||d_{s}(u-h)|\frac{\dx\dy\dt}{|x-y|^{n}}\eqqcolon I_{1}+I_{2}+I_{3}+I_{4}.
\end{aligned}
\end{equation}
From \eqref{cz non pa : dataconditionofIVP} and Young's inequality, we first estimate $I_{1}$ as
\begin{equation*}
    I_{1}\leq c \|h_{t}\|^{2}_{\left(L^{2}(0,T;W^{s,2}(\Omega')\right)^{*}}+\frac{1}{8}\int_{0}^{T}\int_{\Omega'}\int_{\Omega'}|d_{s}(u-h)|^{2}\frac{\dx\dy\dt}{|x-y|^{n}}.
\end{equation*}
We next estimate $I_{2}$ with the help of   \eqref{cz non pa : dataconditionofIVP} as below
\begin{equation*}
    I_{2}\leq c\left(\int_{0}^{T}\int_{\Omega}|g|^{\frac{2(n+2s)}{n+4s}}\dz\right)^{\frac{n+4s}{n+2s}}+\frac{1}{8}\sup_{t\in(0,T)}\int_{\Omega}|(u-h)(x,t)|^{2}\dx+\frac{1}{8}\int_{0}^{T}\int_{\Omega'}\int_{\Omega'}|d_{s}(u-h)|^{2}\frac{\dx\dy\dt}{|x-y|^{n}},
\end{equation*}
where we have used Lemma \ref{cz non pa : embed lem} and Young's inequality.
For the estimate of $I_{3}+I_{4}$ , we follow the same line as in the estimate of $J_{1}+J_{2}$ in \cite[Lemma 2.7]{BKc} to see that
{\small\begin{equation*}
\begin{aligned}
    I_{3}+I_{4}&\leq \frac{1}{8}\int_{0}^{T}\int_{\Omega'}\int_{\Omega'}|d_{s}(u-h)|^{2}\frac{\dx\dy\dt}{|x-y|^{n}}+c\int_{0}^{T}\int_{\Omega'}\int_{\Omega'}|d_{0}f|^{2}\frac{\dx\dy\dt}{|x-y|^{n}}+c\Tail_{2,s}(f-(f)_{\Omega'}(t);\Omega'_{T})^{2}\\
    &\quad+c\int_{0}^{T}\int_{\Omega'}\int_{\Omega'}|d_{s}h|^{2}\frac{\dx\dy}{|x-y|^{n}}+c\Tail_{2,2s}(h-(h)_{\Omega'}(t);\Omega'_{T})^{2}
\end{aligned}
\end{equation*}}\\
\noindent
for some constant $c=c(n,s,L,\Omega,\Omega',T)$. We now plug the estimates of $I_{1},I_{2}$ and $I_{3}$ into \eqref{cz non pa : firesten} to see that
{\small\begin{equation*}
\begin{aligned}
    &\sup_{t\in(0,T)}\int_{\Omega}|(u-h)(x,t)|^{2}\dx+\int_{0}^{T}\int_{\Omega}\int_{\Omega}|d_{s}(u-h)|^{2}\frac{\dx\dy\dt}{|x-y|^{n}}\\
    &\leq c\int_{0}^{T}\int_{\Omega'}\int_{\Omega'}|d_{0}f|^{2}\frac{\dx\dy\dt}{|x-y|^{n}}+c\Tail_{2,s}(f-(f)_{\Omega'}(t);\Omega'_{T})^{2}+c\left(\int_{0}^{T}\int_{\Omega}|g|^{\frac{2(n+2s)}{n+4s}}\dz\right)^{\frac{n+4s}{n+2s}}\\
    &\quad+c\int_{0}^{T}\int_{\Omega'}\int_{\Omega'}|d_{s}h|^{2}\frac{\dx\dy}{|x-y|^{n}}+c\Tail_{2,2s}(h-(h)_{\Omega'}(t);\Omega'_{T})^{2}+c\|h_{t}\|^{2}_{\left(L^{2}(0,T;W^{s,2}(\Omega')\right)^{*}}.
\end{aligned}
\end{equation*}}\\
\noindent
After a few simple calculations along with \eqref{cz non pa : imbeddingt}, we obtain the desired estimate \eqref{cz non pa : estiivp}.
\end{proof}

\end{document}